\documentclass[10pt]{amsart}
\usepackage{amsfonts}
\usepackage{amsmath,amssymb,amsthm,amsthm}
\usepackage{mathtools}
\usepackage{etoolbox}
\usepackage{algorithm}
\usepackage{algorithmic}
\usepackage[T1]{fontenc}
\usepackage{tikz}
\usepackage{graphicx,subfigure}
\usepackage{epsfig,epstopdf}
\usepackage{sidecap,appendix}
\usepackage{amsopn}
\usepackage{color}
\usepackage{colortbl}
\usepackage[english]{babel}
\usepackage{microtype}
\usepackage{fullpage}
\usepackage{bm}
\usepackage{multirow,cases}
\usepackage{float}

\usepackage{placeins}
\usepackage{flafter}
\allowdisplaybreaks[4]
\DisableLigatures{encoding = *, family = * }

\newtheorem{theorem}{Theorem}[section]

\newtheorem{lemma}[theorem]{Lemma}

\newtheorem{problem}[theorem]{Problem}
\newtheorem{corollary}[theorem]{Corollary}
\numberwithin{equation}{section}

\def\RR{{\mathbb{R}}}
\newcommand{\LL}{\mathcal{L}}

\title{A two-stage numerical approach for the sparse initial source identification of a diffusion-advection equation}

\author{Umberto Biccari\textsuperscript{\,2,3}}
\address{\textsuperscript{2}\,Chair of Computational Mathematics, Fundaci\'on Deusto, Avenida de las Universidades 24, 48007 Bilbao, Basque Country, Spain.}
\address{\textsuperscript{3}\,Facultad de Ingenier\'ia, Universidad de Deusto, Avenida de las Universidades 24, 48007 Bilbao, Basque Country, Spain.}
\email{umberto.biccari@deusto.es, u.biccari@gmail.com}

\thanks{This project has received funding from the European Research Council (ERC) under the European Union's Horizon 2020 research and innovation programme (grant agreement NO: 694126-DyCon). The work of U.B. and E.Z. is partially supported by the Grant PID2020-112617GB-C22 KILEARN of MINECO (Spain) and the Elkartek grant KK-2020/00091 CONVADP of the Basque Government. The work of E.Z. is partially funded by  the Alexander von Humboldt-Professorship program, the European Union's Horizon 2020 research and innovation programme under the Marie Sklodowska-Curie grant agreement No.765579-ConFlex, the Grant ICON-ANR-16-ACHN-0014 of the French ANR and the Transregio 154 Project ``Mathematical Modelling, Simulation and Optimization Using the Example of Gas Networks'' of the German DFG. The work of X.Y is supported by Seed Fund for Basic Research (project number: 202011159106) from The University of Hong Kong.}

\author{Yongcun Song\textsuperscript{\,1,4}}
\address{\textsuperscript{1}\,  Chair for Dynamics, Control and Numerics, Alexander von Humboldt-Professorship, Department of Data Science,  Friedrich-Alexander-Universit\"at Erlangen-N\"urnberg, 91058 Erlangen, Germany.}
\email{ysong307@gmail.com}

\author{Xiaoming Yuan\textsuperscript{\,4}}
\address{\textsuperscript{4}\, Department of Mathematics, The University of Hong Kong, Pok Fu Lam, Hong Kong, China.}
\email{xmyuan@hku.hk}

\author{Enrique Zuazua\textsuperscript{\,1,2,5}}
\address{\textsuperscript{1}\,  Chair for Dynamics, Control and Numerics, Alexander von Humboldt-Professorship, Department of Data Science,  Friedrich-Alexander-Universit\"at Erlangen-N\"urnberg, 91058 Erlangen, Germany.}
\address{\textsuperscript{2}\, Chair of Computational Mathematics, Fundaci\'on Deusto Avda. de las Universidades 24, 48007 Bilbao, Basque Country, Spain.}
\address{\textsuperscript{5}\, Departamento de Matem\'aticas, Universidad Aut\'onoma de Madrid, 28049 Madrid, Spain.}
\email{enrique.zuazua@fau.de}

\begin{document}

\bibliographystyle{acm}

\begin{abstract}

We consider the problem of identifying a sparse initial source condition to achieve a given state distribution of a diffusion-advection partial differential equation after a given final time. The initial condition is assumed to be a finite combination of Dirac measures. The locations and intensities of this initial condition are required to be identified. This problem is known to be exponentially ill-posed because of the strong diffusive and smoothing effects. We propose a two-stage numerical approach to treat this problem. At the first stage, to obtain a sparse initial condition with the desire of achieving the given state subject to a certain tolerance, we propose an optimal control problem involving sparsity-promoting and ill-posedness-avoiding terms in the cost functional, and introduce a generalized primal-dual algorithm for this optimal control problem. At the second stage, the initial condition obtained from the optimal control problem is further enhanced by identifying its locations and intensities in its representation of the combination of Dirac measures. This two-stage numerical approach is shown to be easily implementable and its efficiency in short time horizons is promisingly validated by the results of numerical experiments. Some discussions on long time horizons are also included.

\end{abstract}

\maketitle

\section{Introduction and motivations}

Among various inverse problems arising in scientific computing, an important one is the identification of moving pollution sources in either compressible or incompressible fluids that can be described by diffusion-advection systems. See e.g., \cite{el2005identification,li2006determining} for accurate estimation of pollution sources in the environmental safeguard of a densely populated city, and \cite{gurarslan2015solving,li2014heat} for other related problems. As many contributions in the literature have shown (\cite{casas2019using,casas2015sparse,GER1983,li2014heat,monge2019sparse}), this kind of pollution source identification problems can be mathematically modeled by initial source identification problems of diffusion-advection systems. Besides, as pointed out in \cite{casas2019using,casas2015sparse,el2005identification,leykekhman2020numerical,li2014heat,monge2019sparse}, the initial source is usually assumed to be sparse, i.e., its support is zero in Lebesgue measure. In this paper, we consider the problem of identifying a sparse initial source condition to achieve a given state distribution of a diffusion-advection partial differential equation (PDE) after a given final time. The initial condition is assumed to be a finite combination of Dirac measures, and the locations and intensities of this initial condition are required to be identified.
\subsection{Problem statement}

Let $\Omega\subset\RR^N$ with $N\geq 1$ be a bounded domain and $\partial\Omega$ its boundary. We consider the following linear diffusion-advection equation
\begin{align}\label{eq:mainPb}
	\begin{cases}
		\partial_t u - d\Delta u + v\cdot\nabla u = 0, & (x,t)\in\Omega\times (0,T),
		\\
		u = 0, & (x,t)\in\partial\Omega\times (0,T),
		\\
		u(x,0) = u_0(x), & x\in\Omega,
	\end{cases}	
\end{align}
where $0<T<+\infty$ is a given final time, $d > 0$ is the diffusivity coefficient and the vector $v\in\mathbb{R}^N$ is the velocity field of the advection. Here and in what follows, $d$ and $v$ are both assumed to be constants for simplicity, although our analysis to be presented can be adapted to the case where both diffusivity and velocity fields vary.
We further assume the initial condition ${u}_0(x)$ to be a finite combination of Dirac measures
\begin{align}\label{eq:deltas}
	u_0(x)= \sum_{i=1}^l {\alpha}_i\delta_{{x}}(x_i), \quad {x}_i\in\Omega,
\end{align}
where $\{{\alpha}_i\}_{i=1}^l\in\mathbb{R}^l$ and ${x}_i\in\Omega,1\leq i\leq l$, are the intensities and locations, respectively, with $1\leq l< +\infty$ the number of locations. The Dirac measure $\delta_{{x}}(x_i)$ is defined by $\delta_{{x}}(x_i) = 1$ if $x ={x}_i$, and $\delta_{{x}}(x_i) = 0$ otherwise. Note that (\ref{eq:deltas}) implies that the support of ${u}_0(x)$ is $\{{x}_i\}_{i=1}^l\subset\Omega$ and its Lebesgue measure is zero. With the assumption (\ref{eq:deltas}), one can show that there exists a unique solution $u$ of (\ref{eq:mainPb}) and $u$ belongs to the space $L^r(0,T;W_0^{1,p}(\Omega))$ for all $p,r\in[1,2)$, with $\frac{2}{r}+\frac{N}{p}>N+1$, see \cite{CK2016} and the references therein.

\begin{problem}\label{SIproblem}
Consider the diffusion-advection equation (\ref{eq:mainPb}). Let $u_T$ be a given or observed function. We aim at identifying an initial condition $\widehat{u}_0^*$ subject to (\ref{eq:deltas}), i.e.
\begin{align*}
	\widehat{u}_0^*(x)= \sum_{i=1}^l \widehat{\alpha}^*_i\delta_{x}({\widehat{x}^*_i}), \quad \text{ with }~ \widehat{\alpha}^*_i\in\mathbb{R},~ \widehat{x}^*_i\in\Omega
\end{align*}
 such that the corresponding final state $\widehat{u}^*(\cdot;T)$ of (\ref{eq:mainPb}) is as close as possible to $u_T$, in the sense that for $\varepsilon >0$ arbitrary small we have
\begin{align}\label{eq:distance}
	\|\widehat{u}^*(\cdot;T)- u_T\|_{L^2(\Omega)}\leq\varepsilon, \quad \mbox{a.e in }\Omega.
\end{align}	
\end{problem}
Problem 1.1 plays an important role in various areas such as pollution sources identification, precision mechanical, industrial mechatronic, hydrologic inversion, and image deblurring. We refer to \cite{ohnaka1989,ozis2000} and references therein for more discussions.
As well known (see, e.g., \cite{isakov2017}), due to the strong diffusive and smoothing properties of equation (\ref{eq:mainPb}), Problem \ref{SIproblem} is exponentially ill-posed, which means that a small perturbation on the data $u_T$ may cause an arbitrarily large error in $\widehat{u}_0^*$. {For instance, if we set $\Omega=[0,\pi], d=1$ and $v=0$ in (\ref{eq:mainPb}) and consider a reachable target $u_T$, then addressing Problem \ref{SIproblem} amounts to solving
\begin{align*}
	A_Tu_0:=\sum_{n=1}^{\infty}e^{-n^2T}\langle u_0,v_n\rangle v_n=u_T
\end{align*}
with $v_n$ defined by $v_n(x)=\sqrt{\frac{2}{\pi}}\sin(nx)$. Since $e^{-n^2T}\rightarrow 0$ as $n\rightarrow +\infty$, we see that the operator $A_T$ is compact, which in turn implies that the problem is ill-posed (more discussions on this specific issue can be referred to \cite{BBC1986,engl1996}). Moreover, it is easy to see that if $T$ becomes larger, the problem is increasingly ill-posed. Therefore, it is challenging to design some efficient numerical algorithms for solving Problem \ref{SIproblem}.

\subsection{State-of-the-art}
In the literature, some work has already been done for sparse initial source identification problems, based on the natural idea of taking advantage of the sparse nature of the initial condition. A widely used strategy to address sparse initial source identification problems is to formulate them as optimal control problems modeled by PDEs, in which the initial condition is assumed to play the role of a control term. This is the seminal idea at the basis of some research articles, see e.g., \cite{casas2019using,casas2015sparse,leykekhman2020numerical,monge2019sparse}.

In \cite{casas2019using},  sparse optimal control techniques are used to identify sparse initial sources for diffusion-convection equations. The existence and uniqueness of optimal controls are proved, and necessary and sufficient optimality conditions are obtained. Based on these conditions, the sparsity structure of the optimal control is derived. In \cite{casas2015sparse}, the adjoint methodology for sparse initial source identification problems governed by parabolic equations is introduced. It is proved that the sparse initial condition can be recovered by minimizing its measure-norm under the constraint that the corresponding solution and the given target are close at the final time. In \cite{leykekhman2020numerical}, the identification of an unknown sparse initial source for a homogeneous parabolic equation is addressed by considering an optimal control problem, where the control variable is considered in the space of regular Borel measures  and the corresponding norm is used as a regularization term in the objective functional. Under specific structural assumptions, the authors show that the initial source is a finite combination of Dirac measures as that in \eqref{eq:deltas}.

{It is remarkable that, in the above references, the sparse initial source identification problems are formulated as optimal control problems in measure spaces that can be (equivalently) written as
\begin{align}\label{eq:Mopt}
\min_{u_0\in\mathcal{M}(\Omega)} J(u_0) := \frac 12\|u(\cdot,T)-u_T\|_{L^2(\Omega)}^2 + \beta \|u_0\|_{\mathcal{M}(\Omega)},
\end{align}
where $u(\cdot,T)$ is the solution at $t=T$ of equation (\ref{eq:mainPb}) corresponding to $u_0$; $\beta>0$ is a regularization parameter; $\mathcal{M}(\Omega) = C_0(\Omega)^*$ denotes the space of regular Borel measures in $\Omega$, with $C_0(\Omega)$
the space of continuous functions in $\Omega$ vanishing on $\partial\Omega$, and the norm in this space is
defined by
$$
\|u_0\|_{\mathcal{M}(\Omega)}=|u_0|(\Omega)=\sup\left\{\int_\Omega z du_0~|z\in C_0(\Omega), \|z\|_{{\infty}}\leq1\right\},
$$
 $|u_0|$ being the total variation measure associated to $u$. Similar models can also be found in \cite{duval2015,KHB2020} and the references therein for sparse peak deconvolution. The presence of measures can guarantee the sparsity of the initial source but entails appropriate discretization for measure-valued quantities and may invalidate the application of some well-known numerical methods. For instance, the first-order optimality condition of (\ref{eq:Mopt}) cannot be reformulated in a non-smooth point-wise form and thus the well-known Semi-Smooth Newton (SSN) type methods cannot be applied directly, see e.g., \cite{glowinski2020admm,hinze2008optimization}.

It is shown in \cite{leykekhman2020numerical} that, after some proper discretization, problem (\ref{eq:Mopt}) can be reformulated as a finite-dimensional optimization problem with $\ell^1$-regularization, for which various well-developed optimization algorithms can be applied directly. See \cite{KHB2020} for related discussions on sparse peak deconvolution. However, in the context of optimal control of PDEs, such a direct application of finite-dimensional optimization algorithms may cause the so-called mesh-dependent issue, which means that the convergence behavior critically depends on the fineness of the discretization, see \cite{leykekhman2020numerical}.  Hence, some new numerical algorithms that can be described on the continuous level have to be deliberately designed from scratch.} In this regard, a Primal-Dual Active Point (PDAP) method is proposed in \cite{leykekhman2020numerical}. At each iteration of the PDAP, one entails the solutions of two parabolic equations to update the adjoint variable, an optimization subproblem to find a new support point, and a non-smooth optimization subproblem to compute a new iterate. This non-smooth optimization problem has no closed-form solution and can only be solved iteratively by some optimization algorithm, such as the SSN method suggested therein. Hence, nested iterations are resulted, which may cause some new challenges in the overall rigorous convergence and additional computational loads in the implementation.

To address Poblem \ref{SIproblem}, a two-stage numerical approach is proposed in \cite{monge2019sparse}. First, Poblem \ref{SIproblem} is formulated as an $L^1$-regularized optimal control problem, where the initial condition is treated as the control variable and is assumed to be in $L^1(\Omega)$ to promote the sparsity. As a result, measures are avoided. To solve the optimal control problem, a Gradient Descent (GD) method is suggested. Then, the optimal locations are identified by determining these local maxima/minima of the optimal control, and the corresponding optimal intensities are identified by solving a least squares problem. Several test cases validate that this two-stage approach can accurately identify the sparse initial sources even in heterogeneous media. Despite this fact, we shall remark that the focus in \cite{monge2019sparse} is on the development and discussion of the numerical algorithm, but from a mathematical viewpoint, the optimal control problem considered in \cite{monge2019sparse} is not well-posed. In particular, since the control variable is considered in the non-reflexive space $L^1(\Omega)$, the existence of a solution in $L^1(\Omega)$ to the optimal control problem cannot be guaranteed. See \cite{casas2015sparse,stadler2009elliptic} for some related discussions.

In \cite{li2014heat}, sparse initial sources are identified from some sparsely sampled solutions of the heat equation, where the initial sources are assumed to satisfy \eqref{eq:deltas}. After some proper discretization, the initial source identification problem is formulated as a finite-dimensional constrained $\ell^1$ minimization problem with respect to the initial condition, under the constraint that the corresponding final states of the discretized heat equation are close to the observations. The classical Bregman iteration method \cite{bregman1967relaxation} combined with two acceleration strategies (support restriction and domain exclusion) is suggested to solve the constrained  $\ell^1$ minimization problem. The effectiveness and efficiency of this approach are validated by some numerical experiments, which show that, for two-dimensional spaces, one can recover the sparse initial condition accurately from some point-wise observations at the final time. The Bregman iteration method solves the constrained problem as a sequence of unconstrained subproblems that have no closed-form solutions and can only be solved iteratively. Thus, inner iterations have to be embedded into the implementation of the Bregman iteration method. Hierarchically nested iterations and hence the lack of rigorous analysis for the convergence of the overall scheme are thus caused.  Moreover, as mentioned earlier, such a direct application of the Bregman method may lead to the mesh-dependent issue implying that the convergence depends strongly on the fineness of the discretization.

For completeness, we mention that other types of optimal control problems with sparsity properties have also been widely discussed in the existing literature. In \cite{casas2012approximation,casas2013spike} for elliptic problems and in \cite{casas2013parabolic,kunisch2014measure} for parabolic problems, sparse controls are obtained by considering optimal control problems in the space of measures.  Some $L^1$-regularized elliptic and parabolic optimal control problems are discussed in \cite{schindele2017proximal,stadler2009elliptic}. The use of $L^1$-regularization has been shown to be efficient to obtain optimal controls with support in small regions of the domain; and the support can be adjusted by tuning the $L^1$-regularization parameter in the cost functional.

\subsection{Our numerical approach}
To address Problem \ref{SIproblem}, we propose a new  two-stage numerical approach, which consists of a sparsity promotion stage and a structure enhancement stage. Our approach keeps all advantageous features of the framework in \cite{monge2019sparse} while avoids the aforementioned issues encountered therein. First, in the sparsity promotion stage, we treat the initial condition $u_{0}$ as a control variable and formulate Problem \ref{SIproblem} as an optimal control problem with $L^2+L^1$-regularization term. As to be shown in Section \ref{se:optimalcontrol}, the presence of the $L^1$-regularization can promote the sparsity of the initial source. However, the identified initial source from the optimal control problem is not sparse as desired due to the smoothing property of the $L^2$-regularization term. Hence, a structure enhancement stage should be complemented to ensure that (\ref{eq:deltas}) holds while identify the locations $\{\widehat{x}_i^*\}_{i=1}^l$ and the intensities $\{\widehat{\alpha}^*_i\}_{i=1}^l$.

Concretely, we formulate Problem \ref{SIproblem} in terms of the following optimal control problem:
\begin{align}\label{eq:SIopt}
	\min_{u_0\in L^2(\Omega)} J(u_0) := \frac 12\int_\Omega |u(\cdot,T)-u_T|^2\,dx + \frac \tau2 \int_\Omega |u_0|^2\,dx + \beta \int_\Omega |u_0|\,dx,
\end{align}
where $u(\cdot,T)$ is the solution at $t=T$ of equation (\ref{eq:mainPb}) corresponding to $u_0$. In \eqref{eq:SIopt}, the constants $\tau>0$ and $\beta>0$ are regularization parameters. Similar as the problem in \cite{monge2019sparse}, the first term of $J(u_0)$ seeks for an initial condition $u_0$ such that the corresponding final state of equation (\ref{eq:mainPb}) is as close as possible to $u_T$; and the last term promotes the sparsity of the initial source.   Meanwhile, inspired by \cite{stadler2009elliptic}, we introduce the $L^2$-regularization $\frac \tau2 \int_\Omega |u_0|^2\,dx$ to guarantee the well-posedness of (\ref{eq:SIopt}) while improving the conditioning to allow for a more efficient numerical resolution. For any fixed $\tau>0$, as to be shown in Section \ref{se:struc_pro}, we can always tune $\beta$  to get an optimal control $u_0^*$ with small support. Note that if $\tau=0$ and $u_0\in L^1(\Omega)$, problem (\ref{eq:SIopt}) is not well-posed. To address this issue, a natural way is to consider $u_0\in \mathcal{M}(\Omega)$ and relax $\beta \int_\Omega |u_0|\,dx$ to $\beta\|u\|_{\mathcal{M}(\Omega)}$ so that problem (\ref{eq:Mopt}) is obtained. From this perspective, problem (\ref{eq:SIopt}) can be viewed as a regularized version of (\ref{eq:Mopt}); related discussions can be referred to \cite{ck2011}.

Notice that the control variable $u_0$ in (\ref{eq:SIopt}) is considered as a general function in $L^2(\Omega)$ and it is not assumed to satisfy (\ref{eq:deltas}). To identify the locations and intensities directly, one may further assume that ${u}_0(x)= \sum_{i=1}^l \alpha_i\delta_x({x_i})$, with $\alpha_i\in\mathbb{R}$ and $x_i\in\Omega$, in the formulation of (\ref{eq:SIopt}). As a result, the intensities $\{\alpha_i\}_{i=1}^l$ and the locations $\{x_i\}_{i=1}^l$ become the control variables. However, this leads to a non-convex optimization problem which is challenging to be solved both in terms of theory and algorithms. Meanwhile, it causes practical difficulties related to the computation of the derivatives with respect to $\{x_i\}_{i=1}^l$. By contrast, problem (\ref{eq:SIopt}) is convex and the computation of the derivatives with respect to $u_0$ is relatively easier.

Clearly, problem (\ref{eq:SIopt}) operates in function spaces and avoids the employment of measures. As a consequence, it can be easily addressed numerically and various well-developed optimization algorithms can be applied directly. Furthermore, due to the introduction of the $L^2$-regularization term, problem (\ref{eq:SIopt}) allows identifying the sparse initial sources much more efficiently than the one in \cite{monge2019sparse}, as to be validated in Section \ref{se:numerical}. Notwithstanding that, due to the presence of the $L^2$-regularization term and its smoothing property, the recovered initial condition $u_0$ by solving (\ref{eq:SIopt}) is not sparse as desired in \eqref{eq:deltas}.
To validate this fact, we set $\Omega=(0,2)\times(0,1)$, $T=0.01$, $d=1$, $v=(0,0)^\top$, $\tau=10^{-2}$ and $\beta=3\times10^{-1}$, then solve (\ref{eq:SIopt}) by the primal-dual algorithm described in Section \ref{sec:optimAlgo}. Additional details are presented in Section \ref{se:numerical}. The numerical results are visualized in Figure \ref{fig:figure1}, where the left plot corresponds to the reference initial datum $\widehat{u}_0$ assigned a priori in the form of \eqref{eq:deltas}, while the middle plot shows the recovered initial datum $u_0^\ast$ by solving (\ref{eq:SIopt}). We can clearly see that $\widehat{u}_0$ and $u_0^\ast$ do not coincide. In particular, the recovered initial datum $u_0^\ast$ has a small support but it
is not sparse as the reference $\widehat{u}_0$.  The intensities of $u_0^\ast$ are below
the ones of $\widehat{u}_0$.

\begin{figure}[htpb]
	\caption{\small Reference initial datum $\widehat{u}_0$ (left), the recovered initial datum $u_0^\ast$ (middle) by solving (\ref{eq:SIopt}), and the recovered initial datum $\widehat{u}_0^*$ (right) by the two-stage numerical approach. ($\Omega=(0,2)\times(0,1)$, $T=0.01$, $d=1$, $v=(0,0)^\top$, $\tau=10^{-2}$ and $\beta=3\times10^{-1}$)}\label{fig:figure1}
	\includegraphics[width=0.3\textwidth]{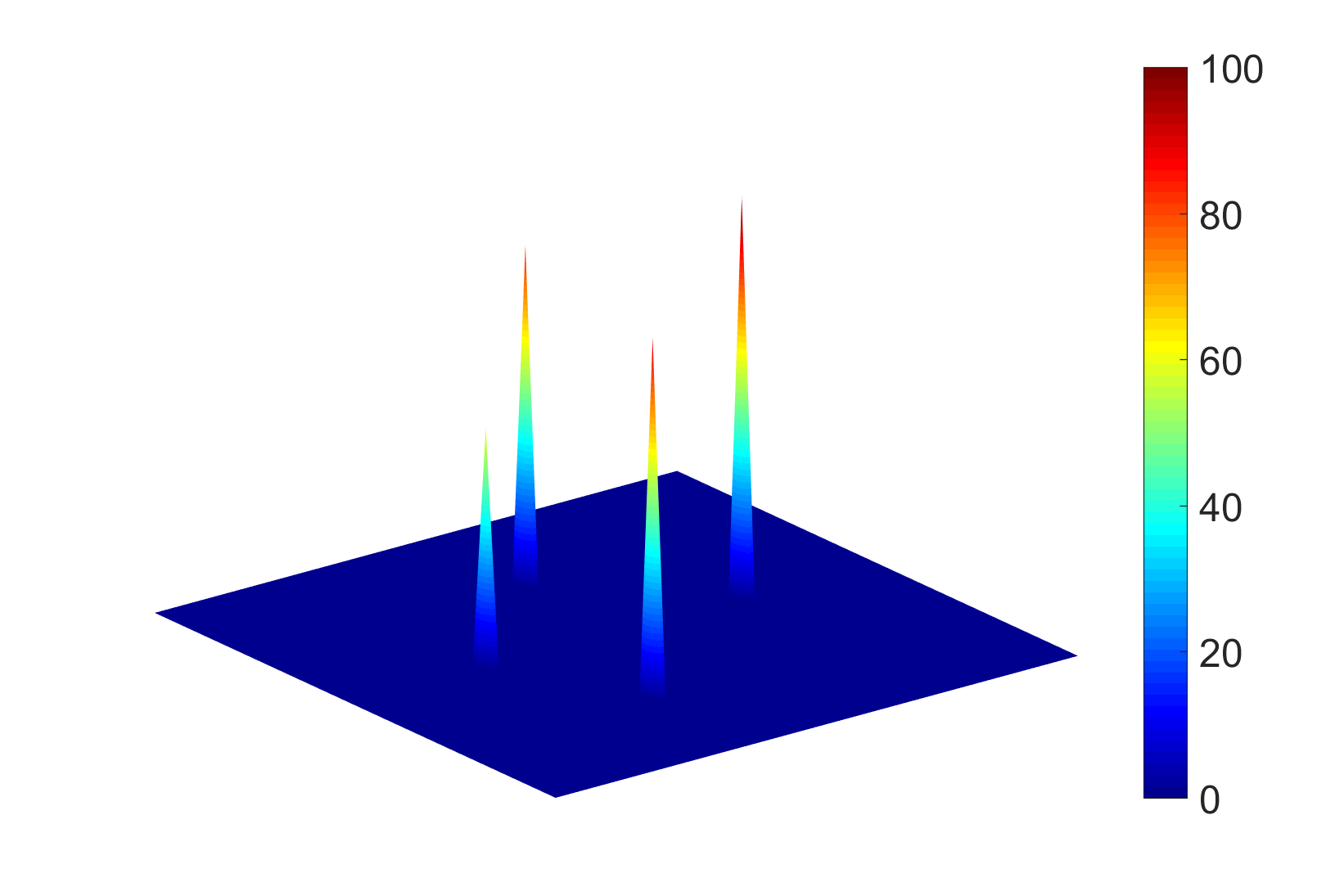}
	\includegraphics[width=0.3\textwidth]{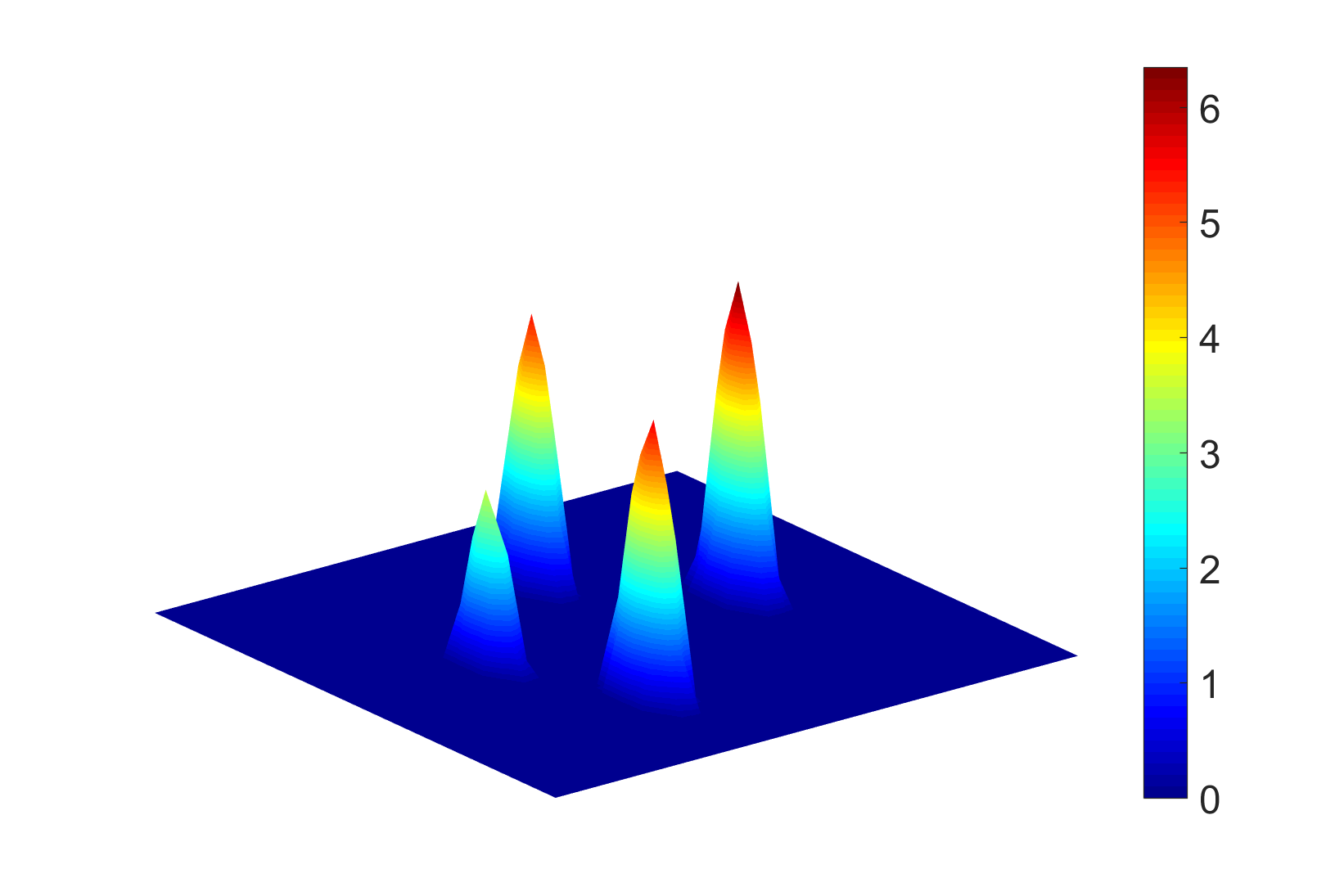}
	\includegraphics[width=0.3\textwidth]{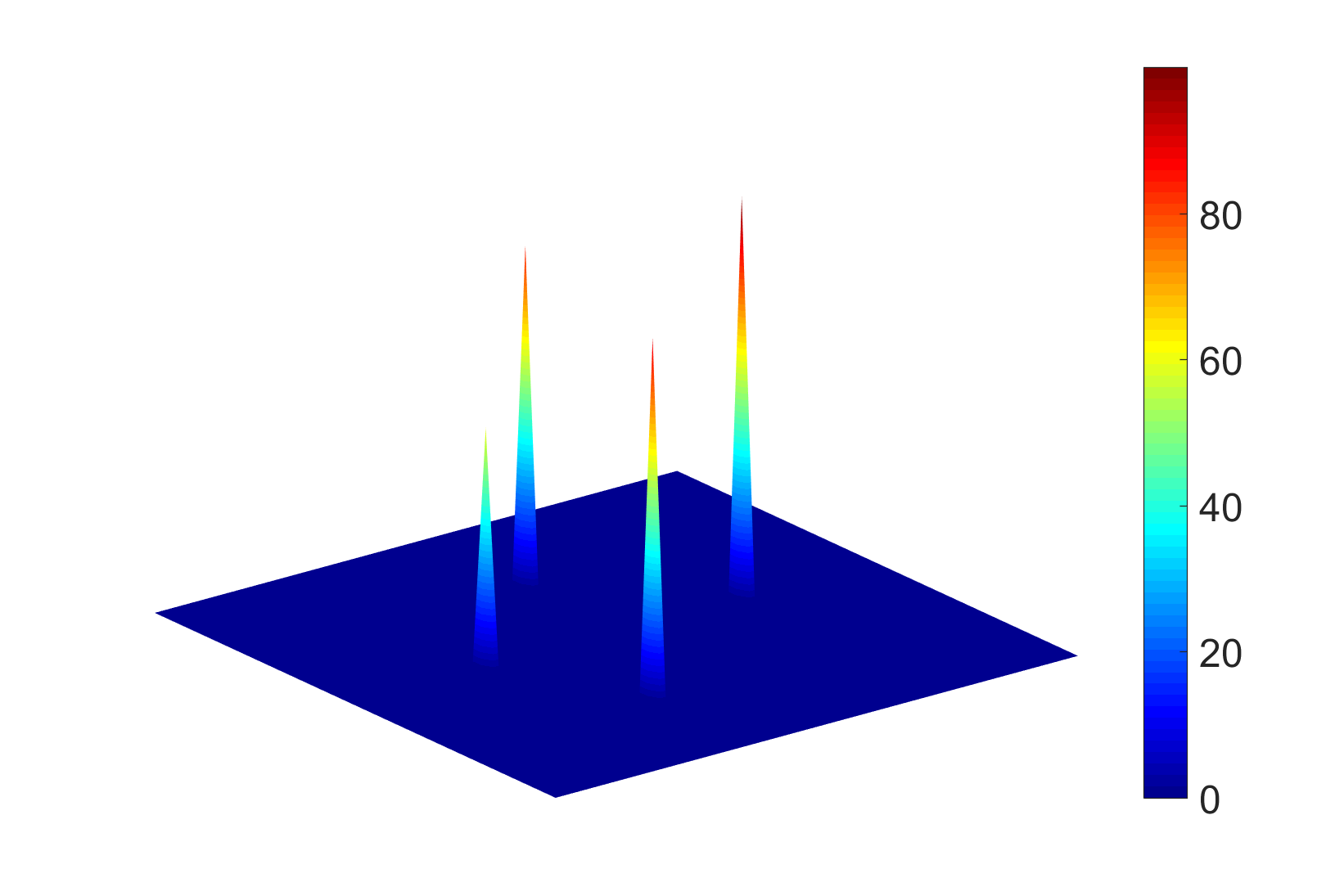}
\end{figure}
For the above reasons, once a numerical solution of  (\ref{eq:SIopt}) is computed, a structure enhancement stage exploiting (\ref{eq:deltas}) is necessary to identify the optimal locations $\{\widehat{x}_i^*\}_{i=1}^l$ and the intensities $\{\widehat{\alpha}^*_i\}_{i=1}^l$. To this end, we propose to solve two simple and low-dimensional optimization problems. More precisely, to identify the optimal locations $\{\widehat{x}_i^*\}_{i=1}^l$, we consider an optimization problem in terms of the spatial variable $x\in \Omega$. Then, motivated by the facts that the initial source $\widehat{u}_0^*$ to be recovered is a finite combination of Dirac measures and the associated final state $\widehat{u}^*(\cdot,T)$ should be as close as possible to $u_T$, we solve a least squares problem to identify the optimal intensities $\{\widehat{\alpha}^*_i\}_{i=1}^l$. A two-stage numerical approach is thus proposed for solving Problem \ref{SIproblem}.  The right plot in Figure \ref{fig:figure1} depicts the recovered initial datum $\widehat{u}_0^*$ by the two-stage numerical approach, which clearly is a highly accurate approximation to the reference initial datum $\widehat{u}_0$. Therefore, the proposed two-stage numerical approach allows identifying the sparse initial sources very accurately, even for some heterogeneous materials or coupled models as validated by some numerical experiments in Section \ref{se:numerical}.

\subsection{Primal-dual algorithms for the solution of (\ref{eq:SIopt})}
Note that the identification of the optimal locations and intensities is based on the solution of (\ref{eq:SIopt}). Thus it is crucial to solve (\ref{eq:SIopt}) efficiently. Recall that  (\ref{eq:SIopt}) is modeled in function spaces. Hence, various well-developed optimization algorithms can be applied directly. For instance, SSN-type methods \cite{ulbrich2011semismooth} and the Alternating Direction Method of Multipliers (ADMM) \cite{glowinski1975approximation} can be conceptually applied and they indeed have been successful in solving some other types of optimal control problems in the literature (see \cite{glowinski2020admm,glowinski2022,hinze2008optimization} and the references therein). Nevertheless, we note that at each iteration of SSN and ADMM, a complicated large-scale and ill-conditioned saddle point system and an optimal control subproblem should be iteratively solved, respectively. Both of them are numerically challenging and expensive for such a time-dependent model. { Consequently, some numerical algorithms tailored for these subproblems have to be deliberately designed. The same concerns apply to the Bregman iteration method in \cite{li2014heat}, which can also be considered for solving (\ref{eq:SIopt}).}

To avoid the above issues, we advocate the primal-dual algorithm proposed in \cite{chambolle2011first}, which has been widely used in various areas such as image processing, inverse problems, and statistical learning. As to be shown in Section \ref{sec:optimAlgo}, when the primal-dual algorithm in \cite{chambolle2011first} is applied to problem (\ref{eq:SIopt}), the main computation at each iteration is solving only two PDEs which can be efficiently addressed by various well-developed PDE solvers. Hence, the implementation of the primal-dual algorithm in \cite{chambolle2011first} is easy and computationally cheap for (\ref{eq:SIopt}). To further speed up the convergence, we propose a generalized version of the primal-dual algorithm mainly by following the ideas in \cite{GT1979,he2017algorithmic,he2012convergence}. Moreover, we show that the generalized primal-dual algorithm performs significantly better than the GD described in \cite{monge2019sparse} for the initial source identification procedure.

\subsection{Organization}
The rest of this paper is organized as follows. Some preliminaries including the existence and uniqueness of a solution, the first-order optimality condition, and the structural property of the solution are given in Section \ref{se:optimalcontrol} . A generalized primal-dual algorithm and its implementation details for solving  (\ref{eq:SIopt}) are discussed in Section \ref{sec:optimAlgo}, and its strong global convergence and worst-case convergence rate are analyzed in Section \ref{se:convergence}. A structure enhancement stage is introduced in Section \ref{se:post} to identify the optimal locations and intensities. A two-stage numerical approach is thus proposed, and its efficiency is illustrated in Section \ref{se:numerical} through some numerical experiments. Finally, Section \ref{se:conclusion} gathers some final remarks and future perspectives.

\section{Preliminaries}\label{se:optimalcontrol}
In this section, we analyze some properties of the optimal control problem (\ref{eq:SIopt}). First, the existence and uniqueness of an optimal control $u_0^*$ are discussed. Then, we derive the optimality conditions and deduce some structural properties of $u_0^*$.

\subsection{Analysis of the optimal control problem (\ref{eq:SIopt})}
Let us start by discussing the existence and uniqueness of an optimal control $u_0^\ast$ to \eqref{eq:SIopt}. This comes from a very standard argument and can be easily obtained by adapting the proof of \cite[Lemma 2.3]{casas2015sparse}.
\begin{theorem}
There exists a unique solution $u_0^\ast\in L^2(\Omega)$ of the optimal control problem \eqref{eq:SIopt}.
\end{theorem}
Then, using some similar arguments as those in \cite{casas2017review,stadler2009elliptic}, we have the following result.
\begin{theorem}
	Suppose that $u_0^*\in L^2(\Omega)$ is the unique solution of the optimal control problem (\ref{eq:SIopt}). Then, the following first-order optimality condition holds:
	\begin{equation}\label{oc}
	\psi^*(\cdot,0) + \tau u^*_0 + \lambda_{u_0}^*=0,
	\end{equation}
	where $\lambda_{u_0}^*\in \partial \varphi(u_0^*)$ with $ \varphi(u_0^*)=\beta\int_\Omega|u_0^*|dx$, and $\psi^*$ is the corresponding adjoint variable that is the successive solution of the state equation (\ref{eq:mainPb}) and the adjoint equation
	\begin{align}\label{eq:dualGD}
		\begin{cases}
				\partial_t\psi + d\Delta \psi + v\cdot\nabla \psi = 0, & (x,t)\in\Omega\times (0,T),
				\\
				\psi = 0, & (x,t)\in \partial\Omega\times (0,T),
				\\
				\psi(\cdot,T) = u(\cdot,T)-u_T:= \psi_T, & x\in\Omega,
			\end{cases}
	\end{align}
	 provided the initial datum $u^*_0$.
\end{theorem}

\subsection{Structural properties of $u_0^*$}\label{se:struc_pro}

Recall that $\lambda_{u_0}^*\in \partial \varphi(u_0^*)=\beta\partial\int_\Omega|u_0^*|dx$. Moreover, it follows from the results of \cite{justen2009general} that
\begin{align*}
	\lambda_{u_0}^*\in\beta\text{sign}(u_0^*),
\end{align*}
where the set-valued function $\text{sign}(\cdot)$ is given by
\begin{align*}
	\text{sign}(v)= \begin{cases}
		\displaystyle\frac{v}{|v|}, & \text{if}~v\neq0,
		\\
		\{\eta:|\eta|\leq 1\}, & \text{otherwise}.
	\end{cases}
\end{align*}

Then, one can consider the optimality condition (\ref{oc}) for all $x\in \Omega$ and get a pointwise
relation of $u_0^*$ and $\psi^*(\cdot,0)$ as displayed in Figure \ref{sparse}. To be concrete, for any $x\in \Omega$, we have
\begin{align*}
	\begin{cases}
		\displaystyle u_0^*(x)=\frac{1}{\tau}(-\psi^*(x,0)-\beta), &\quad\text{if}~u_0^*(x)>0,
		\\[10pt]
		\displaystyle u_0^*(x)=\frac{1}{\tau}(-\psi^*(x,0)+\beta), &\quad\text{if}~u_0^*(x)<0,
		\\[10pt]
		|\psi^*(x,0)|\leq\beta, &\quad\text{if}~u_0^*(x)=0,
	\end{cases}
\end{align*}
which implies that
\begin{equation*}
	u_0^*(x)=-\text{sign}(\psi^*(x,0))\max\left\{\frac{1}{\tau}\Big(|\psi^*(x,0)|-\beta\Big),0\right\}.
\end{equation*}

\begin{figure}[htpb]
\caption{Relationship between $\psi^*(x,0)$ and $u_0^*(x)$.}\label{sparse}
\begin{tikzpicture}
\draw[->](-5,0)--(5,0) node[right]{$\psi^*(x,0)$};
\draw[->](0,-2)--(0,2) node[above]{$u_0^*(x)$};
\draw (-3,0)--(-3,0.1)  node[below=3.6pt,color=blue]{-$\beta$};
\draw (3,0)--(3,0.1) node[below=3.6pt,color=blue]{$\beta$};
\draw[line width =1pt, color=blue] (-3,0)--(3,0);
\draw[domain =3:5,line width =1pt, color=blue] plot (\x ,{-0.8*(\x-3)}) node[below] {$u_0^*(x)=-\frac{1}{\tau}\psi^*(x,0)+\frac{\beta}{\tau}$};
\draw[domain =-3:-5,line width =1pt, color=blue] plot (\x ,{-0.8*(\x+3)}) node[above] {$u_0^*(x)=-\frac{1}{\tau}\psi^*(x,0)-\frac{\beta}{\tau}$};
\end{tikzpicture}
\end{figure}
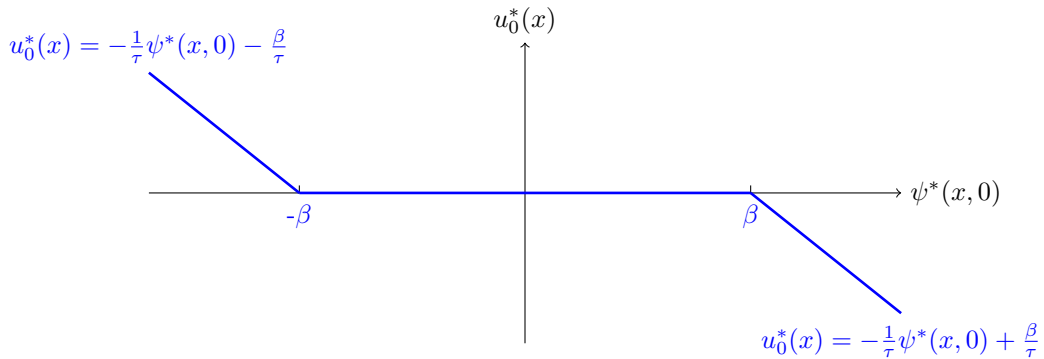

\noindent We thus have the following structural property of $u_0^*$.
\begin{theorem}
Let $u_0^*\in L^2(\Omega)$ be the unique solution of  problem (\ref{eq:SIopt}), and $\psi^*$ be the corresponding adjoint variable.
Then, for a.e. $x\in \Omega$, we have that $|\psi^*(x,0)|\leq\beta$ implies $u_0^*(x)=0$.
\end{theorem}

When $\beta$ is sufficient large, using some similar arguments as those in \cite{stadler2009elliptic}, we can prove that $u_0^*=0$ on the whole domain $\Omega$.
\begin{theorem}
Let $\mathcal L:L^2(\Omega)\to L^2(\Omega)$ be the solution operator associated with the  diffusion-advection equation \eqref{eq:mainPb}, i.e. $\mathcal Lu_0 = u(\cdot,T)$, and let $\mathcal L^*$ denote its adjoint. Let $\beta_0 :=\|\mathcal L^*u_T\|_{L^{\infty}(\Omega)}$, where $\mathcal L^*u_T=\psi(\cdot, 0)$ with $\psi$ the solution of (\ref{eq:dualGD}) corresponding to $\psi(\cdot,T)=u_T$. Then, if $\beta \geq \beta_0$, the unique solution of problem (\ref{eq:SIopt}) is $u_0^*=0$.
\end{theorem}

\begin{proof}
We first note that, with $\mathcal L u_0 = u(\cdot,T)$, the objective functional $J(u_0)$ in (\ref{eq:SIopt}) can be rewritten as
\begin{align*}
	J(u_0) = \frac 12\int_\Omega |\mathcal Lu_0-u_T|^2\,dx + \frac \tau2 \int_\Omega |u_0|^2\,dx + \beta \int_\Omega |u_0|\,dx.
\end{align*}
Then, it is easy to obtain that
\begin{align*}
	J(u_0)-J(0) &= \frac 12\int_\Omega |\mathcal Lu_0|^2\,dx-\int_\Omega \mathcal Lu_0u_T\,dx + \frac \tau2 \int_\Omega |u_0|^2\,dx + \beta \int_\Omega |u_0|\,dx
	\\
	&= \frac 12\|\mathcal Lu_0\|_{L^2(\Omega)}^2-\int_\Omega u_0 \mathcal L^*u_T \,dx + \frac \tau2 \|u_0\|_{L^2(\Omega)}^2 + \beta \|u_0\|_{L^1(\Omega)}
	\\
	&\geq \frac 12\|\mathcal Lu_0\|_{L^2(\Omega)}^2-\|u_0\|_{L^1(\Omega)}\|\mathcal L^*u_T\|_{L^{\infty}(\Omega)} + \frac \tau2 \|u_0\|_{L^2(\Omega)}^2 + \beta \|u_0\|_{L^1(\Omega)}
	\\
	&= \frac 12\|\mathcal Lu_0\|_{L^2(\Omega)}^2+(\beta-\|\mathcal L^*u_T\|_{L^{\infty}(\Omega)}) \|u_0\|_{L^1(\Omega)}+ \frac \tau2 \|u_0\|_{L^2(\Omega)}^2.
\end{align*}
	
If $\beta\geq\beta_0$, we have that $J(0)\leq J(u_0)$ for any $u_0\in L^2(\Omega)$, which implies that the unique solution of problem (\ref{eq:SIopt}) is $u_0^*=0$.
\end{proof}

{Moreover, for $\beta = 0$, it follows from (\ref{oc}) that $u_0^*$
	is not zero whenever $\psi^*(\cdot,0)$ is not zero. Typically in this case, $u_0^*$ is nonzero almost everywhere in $\Omega$.
	Therefore, we can tune $\beta$ in the interval $(0,\beta_0)$ to get an optimal control $u_0^*$ with small support.}

\section{A generalized primal-dual algorithm for the optimal control problem (\ref{eq:SIopt})}\label{sec:optimAlgo}

In this section, we propose a generalized primal-dual algorithm for the optimal control problem (\ref{eq:SIopt}) and delineate its implementation details. We are inspired by a number of existing works including \cite{chambolle2011first,GT1979,he2017algorithmic,he2012convergence}.

\subsection{A generalized primal-dual algorithmic framework}\label{se:PDHG1}

Let us define
\begin{align*}
	f(\LL u_0)= \frac{1}{2}\int_\Omega |\LL u_0-u_T|^2\,dx \quad \textrm{ and } \quad g(u_0) = \frac \tau2\int_\Omega |u_0|^2\,dx + \beta\int_\Omega |u_0|\,dx.
\end{align*}
Then, the optimal control problem \eqref{eq:SIopt} can be reformulated as
\begin{align}\label{eq:SIopt2}
	\min_{u_0\in L^2(\Omega)} \Big(f(\LL u_0) + g(u_0)\Big).
\end{align}
With an auxiliary variable $p\in L^2(\Omega)$, \eqref{eq:SIopt2} can be reformulated as the saddle point problem
\begin{equation}\label{eq:SIopt3}
	 \min_{u_0\in L^2(\Omega)}\;\max_{p\in L^2(\Omega)} \Big(g(u_0)+\int_\Omega p\LL u_0\,dx-f^\ast(p)\Big),
\end{equation}
where $f^\ast(p) := \sup_{q\in L^2(\Omega)} \Big(\int_\Omega pq\,dx - f(q)\Big)$ is the convex conjugate of $f(q)$ and can be specified as
\begin{align*}
	f^*(p)=\frac{1}{2}\int_\Omega|p|^2\,dx+\int_\Omega p u_T\,dx.
\end{align*}
Inspired by \cite{chambolle2011first,he2017algorithmic}, we propose a generalized primal-dual algorithmic framework for solving problem (\ref{eq:SIopt3}).

\begin{algorithm}[htpb]
	\caption{A generalized primal-dual algorithm for (\ref{eq:SIopt3})}\label{alg:pcalgPDHG}
	\begin{algorithmic}
		\STATE{\textbf{input:}} initial values $u_0^0\in L^2(\Omega)$ and $p^0\in L^2(\Omega)$. Choose constants $\theta\in(0,1]$, $r>0$ and $s>0$ satisfying
		\begin{equation}\label{rs}
			rs<\frac{1}{\|\mathcal{L}\mathcal{L}^*\|},
		\end{equation}
		 and $\rho$ and $\sigma$ satisfying
		\begin{subequations}\label{para_relax}
			\begin{numcases}
				~\rho=\sigma\in (0,2), \quad\hbox{if}\;\;\theta=1,
				\\
				\rho\in \Big(0, 1+\theta-\sqrt{1-\theta}\,\Big] \;\; \mbox{and}\;\; \sigma=\frac \theta\rho,\quad\hbox{if}\;\;	\theta\in(0,1).
			\end{numcases}
		\end{subequations}
		\WHILE{not converged}
		\STATE{\begin{subequations}
				\begin{numcases}
					~\widetilde{u}_0^k=\arg\min_{u_0\in L^2(\Omega)} \left(g(u_0)+\int_\Omega p^k\mathcal{L} u_0\,dx+\frac{1}{2r}\|u_0-u_0^k\|^2_{L^2(\Omega)}\right),\label{semi_pd1}
					\\
					\bar{u}_0^k=\widetilde{u}_0^k+\theta(\widetilde{u}_0^k-u^k_0),\label{semi_pd2}
					\\
					\widetilde{p}^{\,k}=\arg\max_{p\in L^2(\Omega)}\left(\int_\Omega p \mathcal{L}\bar{u}^k_0\,dx-f^*(p)-\frac{1}{2s}\|p-p^k\|^2_{L^2(\Omega)}\right),\label{semi_pd3}
					\\
					u_0^{k+1}=u_0^k-\rho(u_0^k-\widetilde{u}^k_0),\label{correct_u}
					\\
					p^{k+1}=p^k-\sigma(p^k-\widetilde{p}^{\,k}.)\label{correct_p}
				\end{numcases}
			\end{subequations}
		}
		\ENDWHILE
	\end{algorithmic}
\end{algorithm}

Algorithm \ref{alg:pcalgPDHG} includes some existing works as special cases. For example, when $\rho=\sigma= 0$ and $\theta=1$, it reduces to the application of the primal-dual algorithm in \cite{chambolle2011first} to (\ref{eq:SIopt3}). That is,
\begin{subequations}\label{CP}
	\begin{numcases}
	{}u_0^{k+1} = \underset{u_0\in L^2(\Omega)}{\arg\min} \left(g(u_0) + \int_\Omega{p}^k\mathcal L u_0\,dx + \frac{1}{2r} \|u_0-u_0^k\|^2_{L^2(\Omega)}\right),\label{cp1}
	\\
	\bar{u}_0^k =2u_0^{k+1}-u_0^k ,\label{cp2}
	\\
	p^{k+1} = \underset{p\in L^2(\Omega)}{\arg\max} \left(\int_\Omega p\mathcal L \bar{u}_0^k\,dx - f^\ast(p) - \frac{1}{2s} \|p-p^k\|^2_{L^2(\Omega)}\right).\label{cp3}
	\end{numcases}
\end{subequations}
Thus, Algorithm \ref{alg:pcalgPDHG} generalizes the primal-dual algorithm (\ref{CP}) with more flexible choices for $\rho$, $\sigma$, and $\theta$, which may result in numerical accelerations accordingly.

\subsection{Implementation of Algorithm \ref{alg:pcalgPDHG}}\label{se:PDHG2}
In this subsection, we discuss the implementation details of Algorithm \ref{alg:pcalgPDHG}. To this end, it is sufficient to focus on the solutions of subproblems (\ref{semi_pd1}) and (\ref{semi_pd3}).

First of all, we observe that the $u$-subproblem \eqref{semi_pd1} can be  reformulated as
\begin{align}\label{eq:PDHG4}
\widetilde{u}_0^k = \underset{u_0\in L^2(\Omega)}{\arg\min} \bigg(\frac\tau2\int_\Omega |u_0|^2\,dx + \beta\int_\Omega |u_0|\,dx + \frac{1}{2r} \|u_0-u_0^k + r\LL^\ast{p}^k\|^2_{L^2(\Omega)}\bigg),
\end{align}
where $\LL^\ast{p}^k:=\zeta^k(\cdot,0)$ is the solution at time $t=0$ of the following backward equation:
\begin{align}\label{eq:dualPDHG}
\begin{cases}
\partial_t\zeta^k + d\Delta \zeta^k + v\cdot \nabla\zeta^k = 0, & (x,t)\in\Omega\times (0,T),
\\
\zeta^k = 0, & (x,t)\in \partial\Omega\times (0,T),
\\
\zeta^k(\cdot,T) = {p}^k, & x\in\Omega.
\end{cases}
\end{align}

In addition, it can be readily checked (see e.g., \cite{justen2009general}) that problem \eqref{eq:PDHG4} has the following closed-form solution
\begin{equation*}
\widetilde{u}_0^k= \mathcal S_{\frac{\beta r}{\tau r + 1}}\left(\frac{u_0^k - r\zeta^k(\cdot,0)}{\tau r + 1}\right),
\end{equation*}
where, for any constant $\gamma>0$, we denoted by $\mathcal S_\gamma$ the Shrinkage operator defined as
\begin{align*}
\mathcal S_\gamma(a) = \begin{cases} a-\gamma, & a>\gamma, \\ 0, & |a|\leq\gamma, \\ a+\gamma, & a<-\gamma. \end{cases}
\end{align*}

Concerning the solution of the $p$-subproblem \eqref{semi_pd3}, it can be computed explicitly by taking into account that $\widetilde{p}^k$ has to satisfy
\begin{align*}
	\nabla_p \left(\int_\Omega p\mathcal L \bar{u}_0^k\,dx - f^\ast(p) - \frac{1}{2s} \|p-p^k\|^2_{L^2(\Omega)}\right)\,\bigg|_{p=\widetilde{p}^k}=0.
\end{align*}
In particular, we have
\begin{align*}
\widetilde{p}^k = \frac{1}{s+1} p^k + \frac{s}{s+1}\Big(\mathcal L \bar{u}_0^k-u_T\Big),
\end{align*}
where $\LL{\bar{u}_0}^k:=\bar{u}^k(\cdot,T)$ is the solution at time $t=T$ of the equation (\ref{eq:mainPb}).

At each iteration, the main computation of Algorithm \ref{alg:pcalgPDHG} only requires the solutions of one forward equation \eqref{eq:mainPb} and one backward equation \eqref{eq:dualPDHG}, and both of them can be efficiently solved by various well-developed PDE solvers. Hence, Algorithm \ref{alg:pcalgPDHG} is easy and computationally cheap to implement.

\section{Convergence analysis of Algorithm \ref{alg:pcalgPDHG}}\label{se:convergence}
In this section, we prove the strong global convergence and derive the worst-case $O(1/K)$ convergence rate measured by the iteration complexity in both the ergodic and non-ergodic senses for Algorithm \ref{alg:pcalgPDHG} in the context of optimal control problems. All the results can be directly extended to the primal-dual algorithm (\ref{CP}) and its relaxed version since they are special cases of Algorithm \ref{alg:pcalgPDHG} with specific choices of parameters. For ease of presentation, we denote by $(\cdot,\cdot)$ the canonical inner product in $L^2$ spaces in the following discussions.

\subsection{Preliminaries}

Denote $(u_0^*,p^*)^\top\in L^2(\Omega)\times L^2(\Omega)$ the saddle point of (\ref{eq:SIopt3}), which in particular means that $u_0^\ast$ is the unique solution of \eqref{eq:SIopt}. Then, the following variational inequalities (VIs) hold:
\begin{subequations}
	\begin{align}
		&\varphi(u_0)-\varphi(u_0^*)+\Big(u_0-u_0^*, \tau u_0^*+\mathcal{L}^*p^*\Big)\geq 0, &\forall u_0\in L^2(\Omega),\label{oc1}
		\\
		&\Big(p-p^*,p^*+u_T-\mathcal{L}u_0^*\Big)\geq 0, &\forall p\in L^2(\Omega),\label{oc2}
	\end{align}
\end{subequations}
where $ \varphi(u_0^*)=\beta\int_\Omega|u_0^*|dx$. We observe that the VIs (\ref{oc1}) and (\ref{oc2}) can be written in a compact form:
\begin{equation}\label{com_oc}
	\varphi(u_0)-\varphi(u_0^*)+\Big(w-w^{*},F(w^{*})\Big)\geq 0,\quad\forall w\in W,
\end{equation}
where
\begin{equation}\label{def_com}
	W=L^2(\Omega)\times L^2(\Omega),\quad
	w= \begin{pmatrix} u_0 \\ p \end{pmatrix}, \quad F(w)= \begin{pmatrix} \tau u_0+\mathcal{L}^*p \\ p-\mathcal{L}u_0+u_T \end{pmatrix}.
\end{equation}
Moreover, a direct calculation shows that, for all $w_1, w_2\in W$,
\begin{align}\label{monotoneF}
	\Big(w_1-w_2, F(w_1)-F(w_2)\Big) =
	 \|p_1-p_2\|^2_{L^2(\Omega)}+\tau\|u_{0,1}-u_{0,2}\|^2_{L^2(\Omega)},
\end{align}
which implies that $F$ is strongly monotone.

Then, we rewrite also the iterative scheme (\ref{semi_pd1})-(\ref{semi_pd3}) in a VI form.  For this purpose, we first note that the optimality conditions of \eqref{semi_pd1} and \eqref{semi_pd3} are
\begin{subequations}
	\begin{align*}
		&\varphi(u_0)-\varphi(\widetilde{u}_0^k)+\Big(u_0-\widetilde{u}_0^k,\tau \widetilde{u}_0^k+\mathcal{L}^*p^k+\frac 1r(\widetilde{u}_0^k-u_0^k)\Big)\geq 0, &\forall u_0\in L^2(\Omega),
		\\
		&\Big(p-\widetilde{p}^{\,k},\widetilde{p}^{\,k}+u_T-\mathcal{L}\bar{u}_0^k+\frac{1}{s}(\widetilde{p}^{\,k}-p^k)\Big)\geq 0, &\forall p\in L^2(\Omega),
	\end{align*}
\end{subequations}
respectively. Taking (\ref{semi_pd2}) into account, we obtain the following VIs:
\begin{subequations}
	\begin{align}
		&\varphi(u_0)-\varphi(\widetilde{u}_0^k)+\Big(u_0-\widetilde{u}_0^k,\tau \widetilde{u}_0^k+\mathcal{L}^*\widetilde{p}^{\,k}-\mathcal{L}^*(\widetilde{p}^{\,k}-p^k)+\frac{1}{r}(\widetilde{u}_0^k-u_0^k)\Big)\geq 0, &\forall u_0\in L^2(\Omega),\label{oc3_pdhg}
		\\
		&\Big(p-\widetilde{p}^{\,k},\widetilde{p}^{\,k}+u_T-\mathcal{L}\widetilde{u}_0^k-\theta \mathcal{L}(\widetilde{u}_0^k-u_0^k)+\frac{1}{s}(\widetilde{p}^{\,k}-p^k)\Big)\geq 0, &\forall p\in L^2(\Omega).\label{oc4_pdhg}
	\end{align}
\end{subequations}

To simplify the notation, we define the following matrix-form operators
\begin{equation}\label{notations}
	\mathcal{D}:=\begin{pmatrix} \rho I & 0 \\ 0 & \sigma I \end{pmatrix},\quad
	G:= \begin{pmatrix} \frac{1}{r}I&- \mathcal{L}^* \\ -\theta \mathcal{L}& \frac{1}{s}I \end{pmatrix}, \quad
	\mathcal{K}:= G\mathcal{D}^{-1}, \quad
	\mathcal{N}:= G+G^*-\mathcal{D}^*\mathcal{K}\mathcal{D}.
\end{equation}
With the notations in (\ref{def_com}) and (\ref{notations}), the VIs (\ref{oc3_pdhg}) and (\ref{oc4_pdhg}), as well as the correction steps (\ref{correct_u}) and (\ref{correct_p}), can be respectively written in the following compact forms
\begin{equation}\label{predictorVI1}
	\varphi(u_0)-\varphi(\widetilde{u}_0^k)+\Big(w-\widetilde{w}^k, F(\widetilde{w}^k)+G(\widetilde{w}^k-w^k)\Big)\geq 0, \quad\forall w\in W,
\end{equation}
and
\begin{equation}\label{correctionstep}
	w^{k+1}=w^k-\mathcal{D}(w^k-\widetilde{w}^k).
\end{equation}


Using some similar arguments as those in  \cite{he2017algorithmic}, we have the following result.
\begin{lemma}\label{res_correction_step}
Let $\theta\in (0,1]$, $r$ and $s$ satisfy \eqref{rs}, $\rho$ and $\sigma$ satisfy (\ref{para_relax}). Then, the matrix-form operators $\mathcal{K}$ and $\mathcal{N}$ defined in \eqref{notations}  are self-adjoint and positive definite, namely,
\begin{equation}\label{convergence_assu}
\begin{array}{lll}
\mathcal{K}=\mathcal{K}^* & \mbox{ and } & (\mathcal{K}w,w)\geq c_1\|w\|^2_{L^2(\Omega)},
\\
\mathcal{N}=\mathcal{N}^* & \mbox{ and } & (\mathcal{N}w,w)\geq c_2\|w\|^2_{L^2(\Omega)}, \quad \forall  w\in W, w\neq 0,
\end{array}
\end{equation}
where $c_1$ and $c_2$ are two positive constants.
\end{lemma}
In the following discussions, we denote by $\|w\|_\mathcal{A}:=(\mathcal{A}w,w),\forall w\in W$, the norm induced by a self-adjoint and positive definite matrix-form operator $\mathcal{A}$. Clearly, it follows from (\ref{convergence_assu}) that the norms $\|w\|_\mathcal{K}$ and $\|w\|_\mathcal{N}$,
$\forall w\in W$, are well-defined.
\subsection{Global convergence of Algorithm \ref{alg:pcalgPDHG}}
In this subsection, we prove the convergence of Algorithm \ref{alg:pcalgPDHG} under the conditions \eqref{rs} and (\ref{para_relax}).
First, we show that the sequence $\{w^k=(u_0^k, p^k)^\top\}_{k\geq 1}$ generated by Algorithm \ref{alg:pcalgPDHG} is strictly contractive.

\begin{theorem}\label{fejer_monotone_pc}
Let $\{w^k=(u_0^k, p^k)^\top\}_{k\geq 1}$ be the sequence generated by Algorithm \ref{alg:pcalgPDHG} and $w^*=(u_0^*, p^*)^\top$ be the solution of problem \eqref{eq:SIopt3}. Suppose that the conditions \eqref{rs} and (\ref{para_relax}) hold. Then, we have
\begin{equation}\label{contraction_pcPDHG}
	\|w^{k+1}-w^*\|_\mathcal{K}^2\leq \|w^{k}-w^*\|_\mathcal{K}^2- \|w^{k}-\widetilde{w}^k\|_\mathcal{N}^2-2\|\widetilde{p}^{\,k}-p^*\|^2_{L^2(\Omega)}-2\tau\|\widetilde{u}_0^{\,k}-u_0^*\|^2_{L^2(\Omega)}.
\end{equation}
\end{theorem}
\begin{proof}
		First of all, it follows from \eqref{notations} and \eqref{correctionstep} that the VI \eqref{predictorVI1} can be written as
	\begin{equation}\label{predictorVI2}
	\varphi(u_0)-\varphi(\widetilde{u}_0^k)+\Big(w-\widetilde{w}^k,F(\widetilde{w}^k)\Big)\geq \Big(w-\widetilde{w}^k, \mathcal{K}({w}^k-w^{k+1})\Big),\quad\forall w\in W.
	\end{equation}
	Then, we apply the identity
	\begin{equation*}
	\big(a-b,\mathcal{K}(c-d)\big)=\frac 12 \left(\|a-d\|_\mathcal{K}^2-\|a-c\|^2_\mathcal{K}\right)+\frac 12\left(\|c-b\|_\mathcal{K}^2-\|d-b\|^2_\mathcal{K}\right)
	\end{equation*}
	to the right-hand side of \eqref{predictorVI2} with
	\begin{equation*}
	a=w, \quad b=\widetilde{w}^k,\quad c={w}^k,\quad\hbox{and}\quad d=w^{k+1}.
	\end{equation*}
	We thus obtain
	\begin{align}\label{identity_result}
	\Big(w-\widetilde{w}^k, \mathcal{K}(w^k-w^{k+1})\Big)=\frac 12\left(\|w-w^{k+1}\|_\mathcal{K}^2-\|w-w^k\|^2_\mathcal{K}\right) +
	\frac 12\left(\|w^k-\widetilde{w}^k\|_\mathcal{K}^2-\|w^{k+1}-\widetilde{w}^k\|^2_\mathcal{K}\right).
	\end{align}
	Considering the last two terms in \eqref{identity_result} and using \eqref{notations} and \eqref{correctionstep}, we have
	\begin{align}
	\|w^k-\widetilde{w}^k\|_\mathcal{K}^2-\|w^{k+1}-\widetilde{w}^k\|^2_\mathcal{K} &= \|w^k-\widetilde{w}^k\|_\mathcal{K}^2-\|(w^k-\widetilde{w}^k)-(w^k-w^{k+1})\|^2_\mathcal{K}\nonumber
	\\
	&= \|w^k-\widetilde{w}^k\|_\mathcal{K}^2-\|(w^k-\widetilde{w}^k)-\mathcal{D}(w^{k}-\widetilde{w}^k)\|^2_\mathcal{K}\nonumber
	\\
	&= 2\Big(w^k-\widetilde{w}^k,\mathcal{K}\mathcal{D}(w^{k}-\widetilde{w}^k)\Big)-
	\Big(\mathcal{D}(w^k-\widetilde{w}^k),\mathcal{K}\mathcal{D}(w^{k}-\widetilde{w}^k)\Big)\nonumber
	\\
	&= 2\Big(w^k-\widetilde{w}^k,G(w^{k}-\widetilde{w}^k)\Big)-
	\Big(w^k-\widetilde{w}^k,\mathcal{D}^\ast\mathcal{K}\mathcal{D}(w^{k}-\widetilde{w}^k)\Big)\nonumber
	\\
	&=\Big(w^k-\widetilde{w}^k, (G+G^*-\mathcal{D}^*\mathcal{K}\mathcal{D})(w^k-\widetilde{w}^k)\Big)\nonumber
	\\
	&=\|w^k-\widetilde{w}^k\|^2_\mathcal{N}.\label{matrixop_result}
	\end{align}
	Combining \eqref{predictorVI2}, \eqref{identity_result} and \eqref{matrixop_result}, we obtain that
	\begin{align}\label{com_VIestimate}
	\varphi(u_0)-\varphi(\widetilde{u}_0^k)+\Big(w-\widetilde{w}^k,F(\widetilde{w}^k)\Big)\geq \frac 12\Big(\|w-w^{k+1}\|_\mathcal{K}^2-\|w-w^k\|^2_\mathcal{K}\Big) +\frac 12\|w^k-\widetilde{w}^k\|^2_\mathcal{N},\quad\forall w\in W.
	\end{align}
	It follows from \eqref{com_VIestimate} that, for all $w\in W$,
	\begin{align}\label{pc_VIestimate1}
	\varphi(\widetilde{u}_0^k) - \varphi(u_0) &+ \Big(\widetilde{w}^k-w,F(w)\Big)+\Big(\widetilde{w}^k-w,F(\widetilde{w}^k)-F(w)\Big)\notag
	\\
	&\leq\frac 12\Big(\|w^k-w\|_\mathcal{K}^2-\|w^{k+1}-{w}\|^2_\mathcal{K}\Big) -\frac 12\|w^k-\widetilde{w}^k\|^2_\mathcal{N}.
	\end{align}
	Moreover, we recall that (see \eqref{monotoneF})
	\begin{align*}
	\Big(\widetilde{w}^k-w,F(\widetilde{w}^k)-F(w)\Big)=\|\widetilde{p}^{\,k}-p\|^2_{L^2(\Omega)}+\tau\|\widetilde{u}_0^{\,k}-u_0\|^2_{L^2(\Omega)}.
	\end{align*}
	Hence, setting $w=w^*$ in \eqref{pc_VIestimate1}, and using \eqref{com_oc}, we finally obtain
	\begin{equation*}
	\|w^{k+1}-w^*\|_\mathcal{K}^2\leq \|w^{k}-w^*\|_\mathcal{K}^2- \|w^{k}-\widetilde{w}^k\|_\mathcal{N}^2-2\|\widetilde{p}^{\,k}-p^*\|^2_{L^2(\Omega)}-2\tau\|\widetilde{u}_0^{\,k}-u_0^*\|^2_{L^2(\Omega)}.
	\end{equation*}
\end{proof}
Theorem \ref{fejer_monotone_pc} shows that the square of distance to a solution point can be reduced by the quantity $\|w^{k}-\widetilde{w}^k\|_\mathcal{N}^2+2\|\widetilde{p}^{\,k}-p^*\|^2_{L^2(\Omega)}+2\tau\|\widetilde{u}_0^{\,k}-u_0^*\|^2_{L^2(\Omega)}$ at the $(k+1)$th iteration. Hence, the sequence $\{w^k=(u_0^k, p^k)^\top\}_{k\geq 1}$ generated by Algorithm \ref{alg:pcalgPDHG} is strictly contractive with respect to the solution $w^*$. This, in turn, implies the convergence of $w^k$ to the solution point $w^*$ of problem \eqref{eq:SIopt3}, as we shall see in the following theorem.

\begin{theorem}\label{thm:convergence_pc}
Let $\{w^k=(u_0^k, p^k)^\top\}_{k\geq 1}$ be the sequence generated by Algorithm \ref{alg:pcalgPDHG} and $w^*=(u_0^*, p^*)^\top$ be the solution of problem \eqref{eq:SIopt3}.  Suppose that the conditions \eqref{rs} and (\ref{para_relax}) hold. Then, $\{u_0^k\}$ converges to $u_0^*$ strongly in $L^2(\Omega)$ and $p^k$ converges to $p^*$ strongly in $L^2(\Omega)$.
\end{theorem}
\begin{proof}
		First of all, it follows from \eqref{contraction_pcPDHG} that
	\begin{align*}
	\sum_{k=0}^{\infty}\left(\|\widetilde{w}^{k}-w^{k}\|_{\mathcal{N}}^2+2\|\widetilde{p}^{\,k}-p^*\|^2_{L^2(\Omega)}+2\tau\|\widetilde{u}_0^{\,k}-u_0^*\|^2_{L^2(\Omega)}\right)\leq \|w^0-w^*\|_\mathcal{K}^2.
	\end{align*}
	This means that the series
	\begin{align*}
	\sum_{k=0}^{\infty}\left(\|\widetilde{w}^{k}-w^{k}\|_{\mathcal{N}}^2+2\|\widetilde{p}^{\,k}-p^*\|^2_{L^2(\Omega)}+2\tau\|\widetilde{u}_0^{\,k}-u_0^*\|^2_{L^2(\Omega)}\right)
	\end{align*}
	is convergent which, in particular, implies
	\begin{equation}\label{pc2}
	\|\widetilde{w}^k-w^{k}\|_{\mathcal{N}}^2\rightarrow 0, \quad \|\widetilde{u}_0^{\,k}-u_0^*\|^2_{L^2(\Omega)}\rightarrow 0,\quad\text{and}\quad \|\widetilde{p}^{\,k}-p^*\|^2_{L^2(\Omega)}\rightarrow 0,\quad\text{as}~k\rightarrow \infty.
	\end{equation}
	Thus
	\begin{equation}\label{pc_p}
	\widetilde{p}^{\,k}\rightarrow p^*,~ \widetilde{u}_0^{\,k}\rightarrow u_0^*,\quad \text{strongly in}~L^2(\Omega).
	\end{equation}
	It follows from \eqref{convergence_assu} and \eqref{pc2} that
	\begin{align*}
	\|\widetilde{w}^k-w^k\|^2_{L^2(\Omega)}=\|\widetilde{u}^k_0-u^k_0\|^2_{L^2(\Omega)}+\|\widetilde{p}^{\,k}-p^k\|^2_{L^2(\Omega)}\rightarrow 0,	
	\end{align*}
	which, in particular, yields
	\begin{equation*}
	\|\widetilde{p}^{\,k}-p^k\|^2_{L^2(\Omega)}\rightarrow 0,\quad\text{and}\quad \|\widetilde{u}_0^{\,k}-u_0^k\|^2_{L^2(\Omega)}\rightarrow 0,\quad\text{as}\quad k\rightarrow \infty.
	\end{equation*}
	This, together with \eqref{pc_p}, implies that
	\begin{equation*}
	p^{k}\rightarrow p^*,~u_0^{k}\rightarrow u_0^*\quad \text{strongly in}~L^2(\Omega).
	\end{equation*}
	Our proof is then concluded.
\end{proof}

\subsection{Convergence rate of Algorithm \ref{alg:pcalgPDHG}}
In this subsection, we analyze the convergence rate of Algorithm \ref{alg:pcalgPDHG}. In particular, we establish an $O(1/K)$ worst-case convergence rate in both ergodic and non-ergodic senses.

Recall that an $O(1/K)$ worst-case convergence rate means that an iterate whose accuracy to the solution under certain criterion is of the order $O(1/K)$
can be found after $K$ iterations of an iterative scheme. This can also be understood as the need of at most $O(1/\varepsilon)$ iterations to find an approximate solution with an accuracy of $\varepsilon$. Besides, we emphasize that such a convergence rate is in the worst-case nature, meaning that it provides a worst-case but universal estimate on the speed of convergence. Hence, it does not contradict with some much faster speeds which might be observed empirically for a specific application (as to be shown in Section \ref{se:numerical}).

\subsubsection{Convergence rate in the ergodic sense}

We first establish the $O(1/K)$ worst-case  convergence rate in the ergodic sense for Algorithm \ref{alg:pcalgPDHG} by following the work \cite{heyuan2012SINUM}.
\begin{theorem}\label{thm:ergodic_pc}
Let $\{w^k=(u_0^k, p^k)^\top\}_{k\geq 1}$ and $\{\widetilde{w}^k=(\widetilde{u}^k_0,\widetilde{p}^{\,k})^\top\}_{k\geq 1}$ be the sequences generated by Algorithm \ref{alg:pcalgPDHG} and $w^*=(u_0^*, p^*)^\top$ be the solution of problem \eqref{eq:SIopt3}. For any $K\in\mathbb{N}$, define
\begin{equation}\label{average_w_pc}
	w_K=\frac{1}{K+1}\sum_{k=0}^K\widetilde{w}^k \quad\text{and}\quad u_{0,K}=\frac{1}{K+1}\sum_{k=0}^K\widetilde{u}_0^k.
\end{equation}
Then, we have
\begin{equation}\label{pc_ergodic}
	\varphi(u_{0,K})-\varphi(u_0^*)+\Big(w_K-w^*,F(w^*)\Big)\leq \frac{1}{2(K+1)}\|w^0-w^*\|_\mathcal{K}^2.
\end{equation}
\end{theorem}
\begin{proof}
		Setting $w=w^*$ in \eqref{pc_VIestimate1}, it follows from the monotonicity of $F$ that
	\begin{align}\label{pc_VIestimate2}
	\varphi(\widetilde{u}_0^k)-\varphi(u_0^*)+\Big(\widetilde{w}^k-w^*,F(w^*)\Big)\leq \frac 12\Big(\|w^k-w^*\|_\mathcal{K}^2-\|w^{k+1}-{w}^*\|^2_\mathcal{K}\Big).
	\end{align}
	Summing the inequality \eqref{pc_VIestimate2} over $k=0,\ldots K$, we then have
	\begin{equation*}
	\frac{1}{K+1}\sum_{k=0}^K \Big(\varphi(\widetilde{u}_0^k)-\varphi(u_0^*)\Big)+\left(\frac{1}{K+1}\sum_{k=0}^K\widetilde{w}^k-w^*,F(w^*)\right)\leq \frac{1}{2(K+1)}\|w^0-w^*\|_\mathcal{K}^2.
	\end{equation*}
	Then, from the convexity of $\varphi$ and \eqref{average_w_pc}, we immediately obtain
	\begin{equation*}
	\varphi(u_{0,K})-\varphi(u_0^*)+\Big(w_K-w^*,F(w^*)\Big)\leq \frac{1}{2(K+1)}\|w^0-w^*\|_\mathcal{K}^2,
	\end{equation*}
	and complete the proof.
\end{proof}
The above theorem shows that, after $K$ iterations of Algorithm \ref{alg:pcalgPDHG}, we can find an approximate solution with an $O(1/K)$ accuracy. This approximate solution is given by $w_K$, and it is the average of all the points $\widetilde{w}^k$ which can be computed by all the known iterates generated Algorithm \ref{alg:pcalgPDHG}. Hence, this is an $O(1/K)$ worst-case convergence rate in the ergodic sense for Algorithm \ref{alg:pcalgPDHG}.

{
As a corollary of Theorem \ref{thm:ergodic_pc}, we have the following convergence rate estimate for Algorithm \ref{alg:pcalgPDHG} with $\theta=1$.
\begin{corollary}
Let $\{w^k=(u_0^k, p^k)^\top\}_{k\geq 1}$ and $\{\widetilde{w}^k=(\widetilde{u}^k_0,\widetilde{p}^{\,k})^\top\}_{k\geq 1}$ be the sequences generated by Algorithm \ref{alg:pcalgPDHG} with $\theta=1$, and $w^*=(u_0^*, p^*)^\top$ be the solution of problem \eqref{eq:SIopt3}. For any $K\in\mathbb{N}$, $w_K$ and $u_{0,K}$ are defined in (\ref{average_w_pc}).
Then, for a constant $c\in (0,2)$, we have
\begin{equation}\label{pc_ergodic_1}
\varphi(u_{0,K})-\varphi(u_0^*)+\Big(w_K-w^*,F(w^*)\Big)\leq \frac{1}{2(K+1)c}\|w^0-w^*\|_{\widetilde{G}}^2,
\end{equation}
where ${\widetilde{G}}$ is obtained from $G$ in (\ref{notations}) by setting $\theta=1$.
\end{corollary}
\begin{proof}
When $\theta=1$, the condition (\ref{para_relax}) implies that $\rho$ and $\sigma$ should be chosen such that $\rho=\sigma\in (0,2)$. Moreover, if we let $c=\rho=\sigma$, the matrix-form operator $\mathcal{K}$ in (\ref{notations}) turns out to be
 $c^{-1}\widetilde{G}$. Then, the desired result (\ref{pc_ergodic_1}) follows from (\ref{pc_ergodic}) directly.
\end{proof}
The above result implies that, to implement Algorithm \ref{alg:pcalgPDHG} with $\theta=1$, it is beneficial to choose $c$ (i.e., $\rho$ and $\sigma$) as close to 2 as possible, in order to reduce the constant on the right hand side of (\ref{pc_ergodic_1}) and thus improve the convergence rate. Moreover, recall that the original primal-dual algorithm (\ref{CP}) is obtained by setting $\rho=\sigma=1$ (i.e., $c=1$) and $\theta=1$ in Algorithm \ref{alg:pcalgPDHG}. Hence, Algorithm \ref{alg:pcalgPDHG} converges faster than the original primal-dual algorithm (\ref{CP}), and this will be validated by some numerical experiments in Section \ref{se:numerical}.
}

\subsubsection{Convergence rate in the non-ergodic sense}
Next, we establish the $O(1/K)$ worst-case  convergence rate in a non-ergodic sense for Algorithm \ref{alg:pcalgPDHG} by following the work \cite{heyuan2015NM}. For this purpose, we first need to define a criterion to precisely measure the accuracy of an iterate.

It follows from (\ref{predictorVI1}) and $G=\mathcal{K}\mathcal{D}$ that the sequence $\{w^k\}_{k\geq 1}$ generated by Algorithm \ref{alg:pcalgPDHG} is a solution point of (\ref{com_oc}) if $\|\mathcal{D}(w^k-\widetilde{w}^k)\|_\mathcal{K}=0$. Hence, it is reasonable to use $\|\mathcal{D}(w^k-\widetilde{w}^k)\|_\mathcal{K}$ or $\|\mathcal{D}(w^k-\widetilde{w}^k)\|_\mathcal{K}^2$ to measure the accuracy of an iterate $w^k$ to a solution point. We have the following result.

\begin{theorem}\label{thm_nonergodic}
Let $\{w^k=(u_0^k, p^k)^\top\}_{k\geq 1}$ and $\{\widetilde{w}^k=(\widetilde{u}^k_0,\widetilde{p}^{\,k})^\top\}_{k\geq 1}$ be the sequences generated by Algorithm \ref{alg:pcalgPDHG} and $w^*=(u_0^*, p^*)^\top$ be the solution of problem \eqref{eq:SIopt3}. Then, for any $K\in\mathbb{N}$, we have
\begin{equation}\label{pc_nonergodic}
	\|\mathcal{D}(w^{K}-\widetilde{w}^K)\|_\mathcal{K}^2 \leq\frac{1}{c_0(K+1)}\|w^0-w^*\|^2_{\mathcal{K}}.
\end{equation}
\end{theorem}
\begin{proof}
We set $w=\widetilde{w}^{k+1}$ in \eqref{predictorVI1} and obtain
\begin{equation}\label{pc_nonergodic1}
	\varphi(\widetilde{u}_0^{k+1})-\varphi(\widetilde{u}_0^k)+\Big(\widetilde{w}^{k+1}-\widetilde{w}^k,F(\widetilde{w}^k)+G(\widetilde{w}^k-w^k)\Big)\geq 0.
\end{equation}
Moreover, we notice that \eqref{predictorVI1} also holds for $k:=k+1$, which yields
\begin{equation*}
	\varphi(u_0)-\varphi(\widetilde{u}_0^{k+1})+\Big(w-\widetilde{w}^{k+1},F(\widetilde{w}^{k+1})+G(\widetilde{w}^{k+1}-w^{k+1})\Big)\geq 0, \quad\forall w\in W.
\end{equation*}
Let $w=\widetilde{w}^k$ in the above inequality. Hence, we have that
\begin{equation}\label{pc_nonergodic2}
	\varphi(\widetilde{u}_0^k)-\varphi(\widetilde{u}_0^{k+1})+\Big(\widetilde{w}^k-\widetilde{w}^{k+1},F(\widetilde{w}^{k+1})+G(\widetilde{w}^{k+1}-w^{k+1})\Big)\geq 0.
\end{equation}
Adding up \eqref{pc_nonergodic1} and \eqref{pc_nonergodic2}, and taking into account \eqref{monotoneF}, we obtain that
\begin{equation*}
	\Big(\widetilde{w}^k-\widetilde{w}^{k+1}, G(\widetilde{w}^{k+1}-w^{k+1})-G(\widetilde{w}^k-w^{k})\Big)\geq 0.
\end{equation*}

Furthermore, observing that $\widetilde{w}^k-\widetilde{w}^{k+1}=\widetilde{w}^k-\widetilde{w}^{k+1}+w^k-w^k+w^{k+1}-w^{k+1}$, the above inequality yields
\begin{equation}\label{pc_nonergodic3}
	\Big(w^k-w^{k+1}, G(\widetilde{w}^{k+1}-w^{k+1})-G(\widetilde{w}^k-w^k)\Big)\geq    \frac 12\|(\widetilde{w}^k-w^k)-(\widetilde{w}^{k+1}-w^{k+1})\|^2_{G^*+G},
\end{equation}
where we used the fact that
$$
\Big(w, Gw\Big)=\frac{1}{2}\Big(w, (G^*+G)w\Big), ~\forall w\in W.
$$
It follows from (\ref{notations}) and (\ref{correctionstep}) that (\ref{pc_nonergodic3}) is equivalent to
\begin{equation}\label{pc_nonergodic4}
	\Big(w^k-\widetilde{w}^k, \mathcal{D}^*\mathcal{K}\mathcal{D}\big((\widetilde{w}^{k+1}-w^{k+1})-(\widetilde{w}^k-w^k)\big)\Big)\geq
\frac 12\|(\widetilde{w}^k-w^k)-(\widetilde{w}^{k+1}-w^{k+1})\|^2_{G^*+G}.
\end{equation}
Applying the identity
$$
	\big(a,{\mathcal{K}}(a-b)\big) = \frac 12\left(\|a\|_{\mathcal{K}}^2-\|b\|_{\mathcal{K}}^2 + \|a-b\|^2_{\mathcal{K}}\right)
$$
to the left-hand side of (\ref{pc_nonergodic4}) with $a=\mathcal{D}(w^k-\widetilde{w}^k)$ and $b=\mathcal{D}(w^{k+1}-\widetilde{w}^{k+1})$, we obtain
\begin{align}\label{pc_nonergodic5}
\Big(w^k-\widetilde{w}^k, & \mathcal{D}^*\mathcal{K}\mathcal{D}\big((\widetilde{w}^{k+1}-w^{k+1})-(\widetilde{w}^k-w^k)\big)\Big)\\
=&\;\frac 12 \|\mathcal{D}(w^k-\widetilde{w}^k)\|_{\mathcal{K}}^2-\frac 12\|\mathcal{D}(w^{k+1}-\widetilde{w}^{k+1})\|_{\mathcal{K}}^2 \notag
+\frac 12 \|\mathcal{D}(w^k-\widetilde{w}^k)-\mathcal{D}(w^{k+1}-\widetilde{w}^{k+1})\|_{\mathcal{K}}^2.
\end{align}
Combining (\ref{pc_nonergodic4}) and (\ref{pc_nonergodic5}), we thus obtain
\begin{align*}
&\|\mathcal{D}(w^k-\widetilde{w}^k)\|_{\mathcal{K}}^2 - \|\mathcal{D}(w^{k+1}-\widetilde{w}^{k+1})\|_{\mathcal{K}}^2
\\
\geq& \|(\widetilde{w}^k-w^k) - (\widetilde{w}^{k+1}-w^{k+1})\|^2_{G^*+G} - \|\mathcal{D}(w^k-\widetilde{w}^k)-\mathcal{D}(w^{k+1}-\widetilde{w}^{k+1})\|_{\mathcal{K}}^2
\\
=& \|(\widetilde{w}^k-w^k)-(\widetilde{w}^{k+1}-w^{k+1})\|^2_{G^*+G-\mathcal{D}^*\mathcal{K}\mathcal{D}}\geq0.
\end{align*}
This implies that the sequence $\|\mathcal{D}(w^k-\widetilde{w}^k)\|_{\mathcal{K}}^2$ is non-increasing, i.e.
\begin{equation}\label{pc_nonergodic6}
\|\mathcal{D}(w^{k+1}-\widetilde{w}^{k+1})\|_{\mathcal{K}}^2\leq\|\mathcal{D}(w^k-\widetilde{w}^k)\|_{\mathcal{K}}^2, \quad \forall k\geq 0.
\end{equation}
Furthermore, it follows from \eqref{convergence_assu} and \eqref{contraction_pcPDHG} that there exists a positive constant $c_0>0$ such that
\begin{equation*}
\|w^{k+1}-w^*\|_\mathcal{K}^2\leq \|w^k-w^*\|_\mathcal{K}^2- c_0\|\mathcal{D}(w^k-\widetilde{w}^k)\|_\mathcal{K}^2,
\end{equation*}
which implies that
\begin{equation}\label{pc_nonergodic7}
c_0\sum_{k=0}^\infty\|\mathcal{D}(w^k-\widetilde{w}^k)\|_\mathcal{K}^2\leq\|w^0-w^*\|^2_{\mathcal{K}}.
\end{equation}
Therefore, it follows from \eqref{pc_nonergodic6} and \eqref{pc_nonergodic7} that for any integer $K>0$, we have
\begin{equation*}
(K+1)\|\mathcal{D}(w^K-\widetilde{w}^K)\|_\mathcal{K}^2\leq\sum_{k=0}^K\|\mathcal{D}(w^k-\widetilde{w}^k)\|_\mathcal{K}^2   \leq\frac{1}{c_0}\|w^0-w^*\|^2_{\mathcal{K}}.
\end{equation*}
Our proof is then complete.
\end{proof}
We note that the number in the right-hand side of (\ref{pc_nonergodic}) is of order $O(1/K)$. Therefore, Theorem \ref{thm_nonergodic} provides an $O(1/K)$ worst-case convergence rate in a non-ergodic sense for Algorithm \ref{alg:pcalgPDHG}.

\section{A structure enhancement stage for identifying the optimal locations and intensities}\label{se:post}

As discussed in the introduction, the numerical solution of the optimal control problem (\ref{eq:SIopt}) is not sparse as desired. This suggests the need of introducing a second procedure to project the obtained non-sparse initial source into the set of admissible sparse solutions in the form of (\ref{eq:deltas}) and identify the locations $\widehat{x}^*:=\{\widehat{x}_i^*\}_{i=1}^l$ and the intensities $\widehat{\bm{\alpha}}^*:=\{\widehat{\alpha}_i^*\}_{i=1}^l$. We thus obtain a two-stage numerical approach for solving Problem \ref{SIproblem}.

\subsection{Optimal locations identification}\label{se:location}
To identify the optimal locations, we recall (see \eqref{eq:deltas}) that the initial condition to be identified is assumed to be a finite combination of Dirac measures.

It was numerically observed in \cite{monge2019sparse} that all local maxima of $|u_0^*(x)|$ fall into the optimal locations. Consequently, one can consider identifying the optimal locations $\widehat{x}^*$ by solving
\begin{equation}\label{eq:location}
	\widehat{x}^*=\arg\max_{x\in \text{supp}(u_0^*)}|u_0^*(x)|,
\end{equation}
where $\text{supp}(u_0^*)$ denotes the support of $u_0^*$ and
 the notation "max"  refers to local maximum. Recall that by tuning the regularization parameter $\beta$, one can always obtain an optimal control $u_0^*$ with small support. Hence, problem (\ref{eq:location}) is usually low-dimensional and computationally cheap to solve.  Let us stress that this approach is a heuristic that has been verified by numerical observations, and it is very interesting to address its related theoretical arguments.
\subsection{Optimal intensities identification}
In this subsection, we explain how to find the intensities $\{\widehat{\alpha}_i^*\}_{i=1}^l$ of the initial source once
we have identified their locations $\{\widehat{x}_i^*\}_{i=1}^l$ by solving (\ref{eq:location}). To this end, we first note that the state equation (\ref{eq:mainPb}) is linear. As a consequence, for any $u_0(x)=\sum_{i=1}^l\alpha_i\delta_x({x_i})$ with $\alpha_i\in\mathbb{R}$ and $x_i\in\Omega$, the solution operator $\mathcal{L}$ verifies
$$
	\mathcal{L}u_0=\sum_{i=1}^l \alpha_i\mathcal{L}\delta_x({x_i}), \quad x_i\in\Omega.
$$

Recall that we aim at identifying a sparse initial condition ${u}_0$ such that $\mathcal{L}u_0$ is as close as possible to the given target $u_T$. Hence, to find the optimal intensities of the initial source, it is sufficient to consider the following least squares problem:
\begin{equation}\label{eq:den_least}
	\{\widehat{\alpha}_i^*\}_{i=1}^l=\underset{{\{\alpha_i\}_{i=1}^l\in \mathbb{R}^l}}{\arg\min}\;\frac 12\left\|\sum_{i=1}^l \alpha_i\mathcal{L}\delta_x({\widehat{x}_i^*})-u_T\right\|^2_{L^2(\Omega)}.
\end{equation}
After a suitable space-time discretization, the discretized formulation of (\ref{eq:den_least}) reads
\begin{equation}\label{eq:LS1}
	\widehat{\bm{\alpha}}^*=\underset{{\bm{\alpha}\in \mathbb{R}^l}}{\arg\min}\frac 12\|{\bf L}\bm{\alpha}-\bm{u_T}\|^2,
\end{equation}
where $\bm{\alpha}=\{\alpha_i\}_{i=1}^l$, the vector $\bm{u_T}\in \mathbb{R}^{N_x}$ is a discretized version of $u_T$ with $N_x$ the number of grid points on $\Omega$, and each column of the matrix ${\bf L}\in\RR^{N_x\times l}$ contains the solution of \eqref{eq:mainPb} with $u(x,0)=\delta_x({\widehat{x}^*_i}), 1\leq i\leq l$.
Note that the support of the desired sparse initial source usually consists of a few points, i.e. $l$ is generally small. Hence, the dimension of problem (\ref{eq:LS1}) is low and it can be solved efficiently through various existing techniques. Here, we suggest to solve the corresponding normal equation
\begin{equation}\label{eq:LS2}
	{\bf L}^\top{\bf L}\widehat{\bm{\alpha}}^*= {\bf L}^\top \bm{u_T},
\end{equation}
to find the vector of intensities  $\widehat{\bm{\alpha}}^*$. Clearly, problem (\ref{eq:LS2}) is a $l\times l$ symmetric positive definite linear system and can be easily solved.

Finally, with the computed locations $\{\widehat{x}_i^*\}_{i=1}^l$ and intensities $\{\widehat{\alpha}_i^*\}_{i=1}^l$, the recovered initial source is thus given by
$$
\widehat{u}_0^*=\sum_{i=1}^l \widehat{\alpha}_i^*\delta_x(\widehat{x}_i^*).
$$

\subsection{A two-stage numerical approach for  Problem \ref{SIproblem}}
In view of the above considerations, the procedure for our initial source identification Problem \ref{SIproblem} needs to be complemented with the structure enhancement stage we just described. The complete methodology is given by Algorithm \ref{alg:Adjoint}.
\begin{algorithm}[htpb]
	\caption{A two-stage numerical approach for solving Problem \ref{SIproblem}.}\label{alg:Adjoint}
	\begin{algorithmic}
		\STATE{\textbf{procedure} SparseIdentification($u_T$)}
		\\
		\STATE{compute $u_0^*$ from the optimal control problem (\ref{eq:SIopt}) by Algorithm \ref{alg:pcalgPDHG}};
		\\
		\STATE{compute $\psi^*(\cdot,0)$ by solving the state equation \eqref{eq:mainPb} and the adjoint equation \eqref{eq:dualGD}};
		\\
		\STATE{find the locations by solving  (\ref{eq:location}).}
		\\
		\FOR{$i = 1,2,\ldots,l$}
		\STATE{compute ${\bf L}(:,i)$ by solving (\ref{eq:mainPb}) with $u(x,0)=\delta_x(\widehat{x}_i^*)$}
		\\
		\ENDFOR
		\\
		\STATE{$\widehat{\bm{\alpha}}^*= ({\bf L}^\top{\bf L})\textbackslash{\bf L}^\top \bm{u_T}$}\\
		\STATE{compute $\widehat{u}_0^* = \sum_{i=1}^l \widehat{\alpha}^*_i\delta_x(\widehat{x}_i^*)$}
	\end{algorithmic}
\end{algorithm}

\section{Numerical experiments}\label{se:numerical}

In this section, we show several test cases to validate that Algorithm \ref{alg:Adjoint} allows identifying the
sparse initial sources accurately from reachable targets or noisy observations, even for some heterogeneous materials or coupled models.  For numerical discretization, we employ the backward Euler finite difference method (with
step size $\Delta t$) for the time discretization and the finite element method (with mesh size
$\Delta x$) described in \cite{monge2019sparse,ww2010} for the space discretization. All our numerical results have been produced by implementing Algorithm \ref{alg:Adjoint} in MATLAB R2016b on a Surface Pro 5 laptop with 64-bit Windows 10.0 operation system, Intel(R) Core(TM) i7-7660U CPU (2.50 GHz), and 16 GB RAM.

\subsection{Generalities}\label{subse:generalities}
We consider Problem \ref{SIproblem} on the domain $\Omega\times(0,T)$ with $\Omega=(0,2)\times(0,1)$ and $T=0.1$; and we test Algorithm \ref{alg:Adjoint} for two scenarios:
\begin{itemize}
	\item[] \textbf{Scenario 1:} the given function $u_T$ is reachable.
	\item[] \textbf{Scenario 2:} the given function $u_T$ is observed with noise.
\end{itemize}
For each scenario, we further consider the following three cases:
\begin{enumerate}
	\item []\textbf{Case I:} diffusivity coefficient $d=0.05$; advection vector $v=(2,-2)^\top$ on $\Omega$. In this case, several initial sources are to be identified in a homogeneous medium, namely, the domain $\Omega$ is constituted by materials with same diffusivity constants.
	\item []\textbf{Case II:} diffusivity coefficient $d = 0.08$ on $\Omega_1 = (0, 1) \times(0,1)$ and $d = 0.05$ on $\Omega_2 = (1,2) \times (0,1)$; advection vector $v = (1, 2)^\top$ on $\Omega$. Here, we consider the advection-diffusion equation modeled in a heterogeneous medium. To be concrete, the left half subdomain $\Omega_1 = (0,1) \times (0,1)$ and the right half one $\Omega_2 = (1,2) \times (0, 1)$ are constituted by materials with different diffusivity constants. Consequently, the dynamics of the problem behaves differently in each of them.
	\item []\textbf{Case III:} diffusivity coefficient $d=0.05$ on $\Omega$; advection vector $v = (0,0)^\top$ on $\Omega_1 = (0,1) \times(0,1)$ and $v = (0,-3)^\top$ on $\Omega_2 = (1,2) \times (0,1)$.  This means that we identify several initial sources for coupled-models, namely, different equations are modeled on the left half ($\Omega_1 = (0,1) \times (0,1)$) and the right half $(\Omega_2 = (1,2) \times (0,1))$ of the domain $\Omega$. More precisely, the heat equation is used on $\Omega_1$ and the diffusion-advection equation is used on $\Omega_2$.
\end{enumerate}

The reference initial datum $\widehat{u}_0$ to be recovered for all cases is set as
\begin{equation}\label{initial_data}
\widehat{u}_0=100\delta{(1.5,0.5)}+85\delta{(1,0.75)}+60\delta{(0.5,0.5)}+90\delta{(0.75,0.25)}.
\end{equation}

We implement the original primal-dual algorithm (\ref{CP}) and Algorithm \ref{alg:pcalgPDHG} to solve the optimal control problem (\ref{eq:SIopt}). Both of them are repeated until the following stopping criterion is fulfilled:
\begin{align*} e_k:=\max\left\{{\|u^{k+1}_0-u_0^k\|_{L^2(\Omega)}}/{\|u_0^{k+1}\|_{L^2(\Omega)}},{\|p^{k+1}-p^k\|_{L^2(\Omega)}}/{\|p^{k+1}\|_{L^2(\Omega)}}\right\}\leq tol
\end{align*}
with $tol=10^{-5}$ or until we reach a maximum number of iterations $k_{max}=1000$.
Moreover, if there are no other specifications, we always use the following parameters:
\begin{itemize}
	\item Mesh sizes: $\Delta x=0.02$ and $\Delta t=0.05$.
	\item Regularization parameters: $\beta=(\Delta x)^4, \tau=10^{-2}.$
	\item The original primal-dual algorithm (\ref{CP}): $ r=6, s=0.193 (\approx\frac{0.999}{r\|\mathcal L^*\mathcal L\|})$.
	\item Algorithm 1: $\theta=1$, $ r=6, s=0.193,  \rho=\sigma=1.9$.
	\item Initial values: $u_0^0=0, p^0=0$.
\end{itemize}

Moreover, we compare the numerical efficiency of our approach with the one described in \cite{monge2019sparse}, and show that our
methodology yields significant improvements in the performance of the initial source identification procedure. For completeness, we  review the approach in \cite{monge2019sparse} briefly.

In \cite{monge2019sparse}, Problem \ref{SIproblem} was formulated as an optimal control problem but in the absence of an $L^2$-regularization in the cost functional (that is, taking $\tau= 0$ in \eqref{eq:SIopt}). To address the resulting optimal control problem numerically, a GD approach was employed, which consists of looking for the minimizer $u_0^\ast$ as the limit $k\to +\infty$ of the following iterative process:
\begin{align*}
	u_0^{k+1} = u_0^k-\eta_k\nabla J(u_0^k).
\end{align*}
Applying the above iterative scheme to the optimal control problem (\ref{eq:SIopt})  yields
\begin{align}\label{eq:GDscheme}
	u_0^{k+1} = u_0^k - \eta_k(\psi^k_0 + \tau u_0^k + \lambda_{u_0^k}),
\end{align}
where $\psi^k_0 = \psi^k(\cdot,0)$ with $\psi^k$ the solution of (\ref{eq:dualGD}). It is clear that the computational load of each GD iteration (\ref{eq:GDscheme}) is the same as that of Algorithm \ref{alg:pcalgPDHG}.  It is worth noting that the $L^1$-regularization is nonsmooth in the optimal control problem (\ref{eq:SIopt}). Thus, a subgradient of the objective functional is used as the proxy of its gradient for implementation. For the convenience of comparison, we follow the notation in \cite{monge2019sparse} and still call it a GD method. 

In \eqref{eq:GDscheme}, the parameter $\eta_k > 0$ is called the step-size and plays a fundamental role in the convergence of the scheme. It is by now well-known that, if one takes $\eta_k$ constant small enough and the objective functional is sufficiently regular (convex, differentiable, and with Lipschitz gradient), then \eqref{eq:GDscheme} will eventually converge to the minimum (see, e.g., \cite[Section 2.1.5]{nesterov2004introductory}).

Nevertheless, the choice of a constant step-size is most often not optimal: if $\eta_k$ is too small, the convergence velocity of GD may drastically decrease while, if $\eta_k$ is too large, one can generate overshooting phenomena and not be able to reach the minimum of $J$.
Hence, in numerical implementations, an adaptive choice of the step-size is usually introduced (e.g., Armijo line search). In this regard, it is worth recalling that these adaptive strategies require the evaluation of the objective function value repeatedly, which in our case is numerically expensive because each one of these evaluations requires solving (\ref{eq:mainPb}).
For the above reasons, in our implementation of GD we always considered a constant step-size although, as we shall see, this choice contributes to making the GD methodology less efficient.

\subsection{Reachable target $u_T$}
We first test Algorithm \ref{alg:Adjoint} for Problem \ref{SIproblem} where the target function $u_T$ is reachable. In particular, we set
the target function $u_T$ as the solution of (\ref{eq:mainPb}) at $T=0.1$ corresponding to the initial condition $u(x;0)=\widehat{u}_0$ in (\ref{initial_data}).

We apply the original primal-dual algorithm (\ref{CP}), Algorithm \ref{alg:pcalgPDHG}, and the GD method in \cite{monge2019sparse} to the optimal control problem (\ref{eq:SIopt}). The efficiency (in terms of the number of iterations to converge) is collected in Table \ref{tab:performance_comparison}. First of all, we observe that the iteration numbers of the algorithm (\ref{CP}) and Algorithm \ref{alg:pcalgPDHG} are almost unchanged for different cases. We thus conclude that their convergence are robust with respect to the diffusion coefficient $d$ and the convection coefficient $v$, at least for the cases we considered. We also observe from Table \ref{tab:performance_comparison} that Algorithm \ref{alg:pcalgPDHG} improves the numerical efficiency of the original primal-dual algorithm (\ref{CP}) by a factor about $40\%$, and both of them are more efficient than the GD method.

\begin{table}[htpb]
	\setlength{\abovecaptionskip}{0pt}
	\setlength{\belowcaptionskip}{4pt}
	\centering
	\caption{\small Numerical comparisons of different algorithms for Cases I-III. ("Iter": the number of iterations to converge; "Err": the relative error $\|u^{k+1}_0-u_0^k\|_{L^2(\Omega)}$/$\|u_0^{k+1}\|_{L^2(\Omega)}$; "CPU": the CPU time listed in seconds) }
	\label{tab:performance_comparison}
	{\footnotesize \begin{tabular}{|c|c|c|c|c|c|c|}
		\hline
		\multirow{2}{*}{}&
		\multicolumn{3}{c|}{ Model (\ref{eq:SIopt})}&\multicolumn{3}{c|}{ Model in \cite{monge2019sparse}}\cr\cline{2-7}
		&\footnotesize{Algorithm (\ref{CP})}&\footnotesize{Algorithm \ref{alg:pcalgPDHG}}&\footnotesize{GD}&\footnotesize{Algorithm (\ref{CP})}&\footnotesize{Algorithm \ref{alg:pcalgPDHG}}&\footnotesize{GD}\cr\cline{2-7}
		&\footnotesize{Iter/Err/CPU}&\footnotesize{Iter/Err/CPU}&\footnotesize{Iter/Err/CPU}&\footnotesize{Iter/Err/CPU}&\footnotesize{Iter/Err/CPU}&\footnotesize{Iter/Err/CPU}\cr
		\hline
		Case I&53/$3\times 10^{-6}$/22&32/$4\times 10^{-6}$/13&86/$9\times 10^{-6}$/39&629/$8\times 10^{-6}$/260&589/$8\times 10^{-6}$/242&673/$9\times 10^{-6}$/270\cr\hline
		Case II&54/$3\times 10^{-6}$/22&32/$4\times 10^{-6}$/13&87/$9\times 10^{-6}$/40&632/$8\times 10^{-6}$/261&612/$8\times 10^{-6}$/256&650/$9\times 10^{-6}$/265\cr\hline
		Case III&52/$3\times 10^{-6}$/21&32/$4\times 10^{-6}$/13&87/$9\times 10^{-6}$/40&648/$8\times 10^{-6}$/266&601/$8\times 10^{-6}$/251&667/$9\times 10^{-6}$/269\cr\hline
	\end{tabular}
}
\end{table}

For comparison purposes, we also implement the original primal-dual algorithm (\ref{CP}), Algorithm \ref{alg:pcalgPDHG}, and the GD method for the model introduced in \cite{monge2019sparse}. The efficiency of each methodology is once again collected in Table \ref{tab:performance_comparison}. It is not surprising that a significantly higher number of iterations is required because the model considered in \cite{monge2019sparse} excludes the term $\frac\tau2 \int_\Omega |u_0|^2\,dx$ and is much more ill-conditioned than (\ref{eq:SIopt}).

Furthermore, we recall that Algorithm \ref{alg:pcalgPDHG} is described on the continuous level and its convergence property is analyzed in function spaces. Hence, mesh independent property of Algorithm \ref{alg:pcalgPDHG} can be expected in practice, which means that the convergence behavior is independent of the fineness of the discretization. This is confirmed by our numerical results presented in Table \ref{mesh_indepdent}. The same conclusion also applies to the original primal-dual algorithm (\ref{CP}).

\begin{table}[htpb]
	\setlength{\abovecaptionskip}{0pt}
	\setlength{\belowcaptionskip}{4pt}
	\centering
	\caption{\small Iteration numbers with respect to different mesh sizes for Case I}\label{mesh_indepdent}
	{\footnotesize\begin{tabular}{|c|c|c|c|c|}
			\hline
			Mesh size&$\Delta t=0.1,\Delta x=0.05$&$\Delta t=0.05,\Delta x=0.02$&$\Delta t=0.025,\Delta x=0.0125$&$\Delta t=0.0156,\Delta x=0.00781$\\
			\hline
			Algorithm (\ref{CP})& 61&53&49 & 46 \\
			\hline
			Algorithm \ref{alg:pcalgPDHG}& 37 &32&29& 27\\
			\hline
			
		\end{tabular}
	}
\end{table}

For Case I, the recovered initial datum $\widehat{u}_0^*$ by Algorithm \ref{alg:Adjoint} and the corresponding final state $\widehat{u}^*(\cdot,T)$ are displayed in Figure \ref{Case1}.  One can observe that both the locations and the intensities of the initial condition are recovered very accurately, which validates the effectiveness and efficiency of Algorithm \ref{alg:Adjoint}.

\begin{figure}[htpb]
		\caption{\small Sparse initial sources identification by Algorithm \ref{alg:Adjoint} for Case I ($d=0.05,v=(2,-2)^{\top}$ on $\Omega$) with a reachable target $u_T$ at $T=0.1$.
 }\label{Case1}
	\centering
	\subfigure[Reference initial state (front view)]{
		\includegraphics[width=0.32\textwidth]{./figures/reference_initial_above}}
\subfigure[Reference initial state (above view)]{	\includegraphics[width=0.32\textwidth]{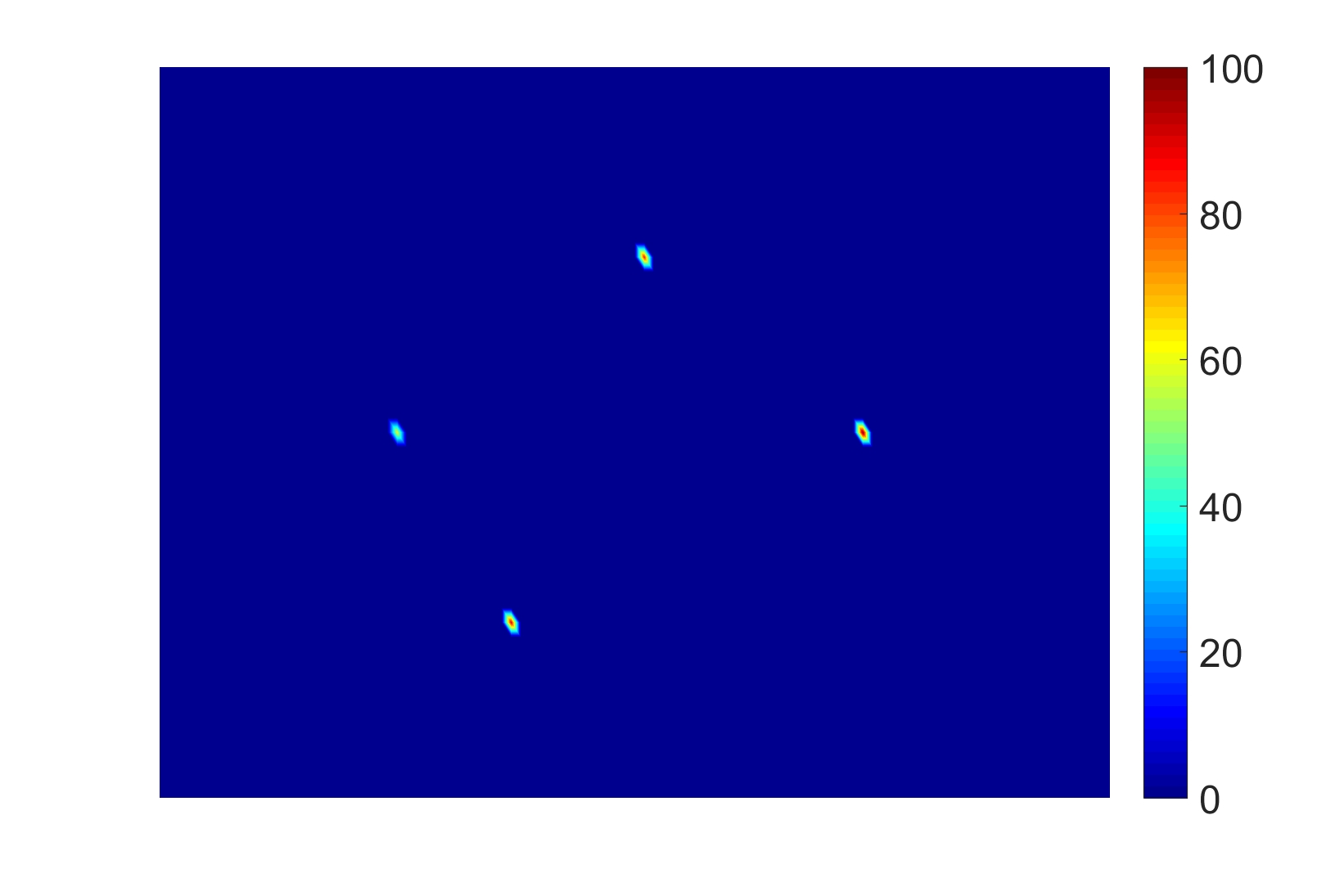}}
	\subfigure[Reachable target $u_T$]{	\includegraphics[width=0.32\textwidth]{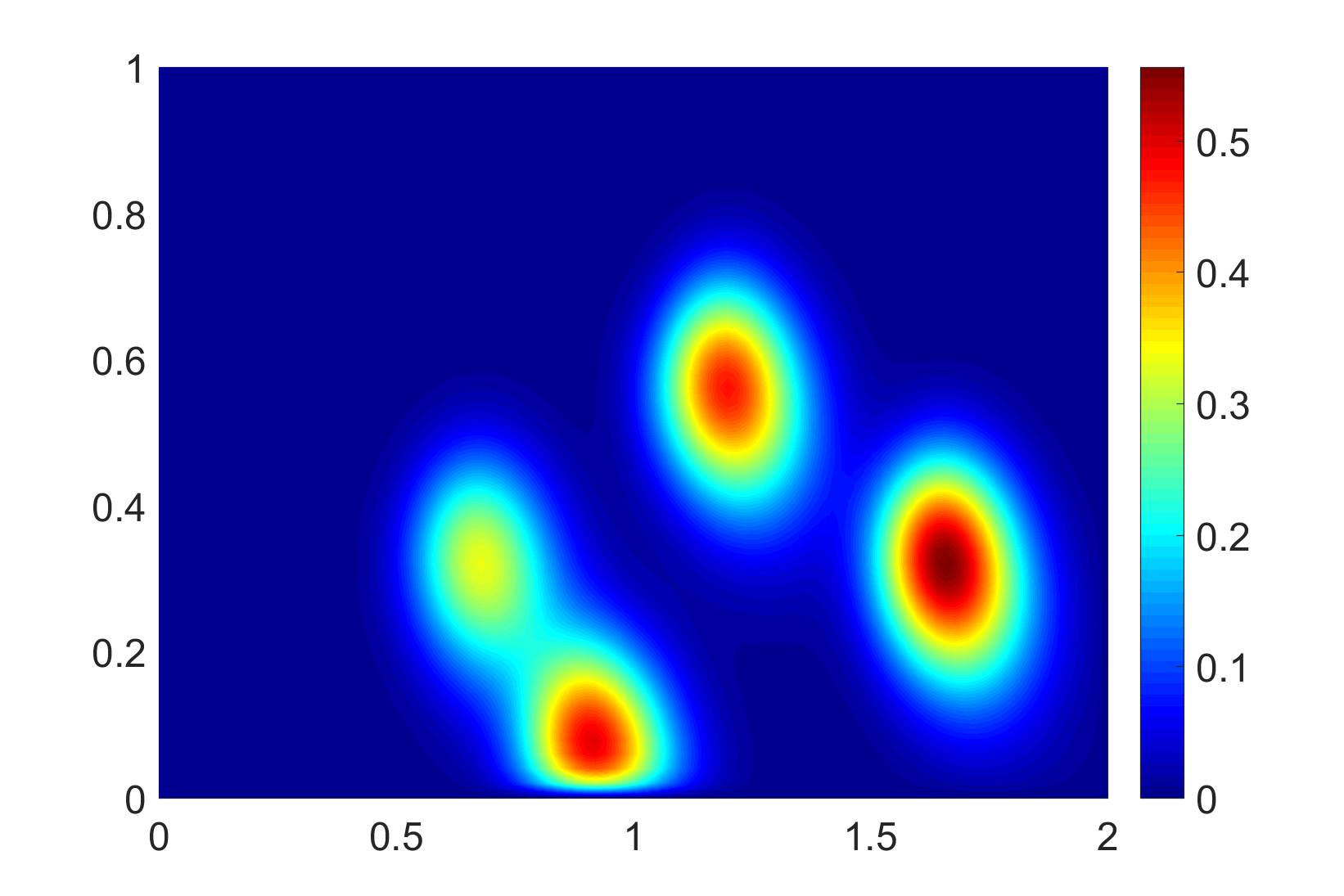}}\\
	\subfigure[Recovered initial state (front view)]{\includegraphics[width=0.32\textwidth]{./figures/recovered_initial_front_c1}}
	\subfigure[Recovered initial state (above view)]{\includegraphics[width=0.32\textwidth]{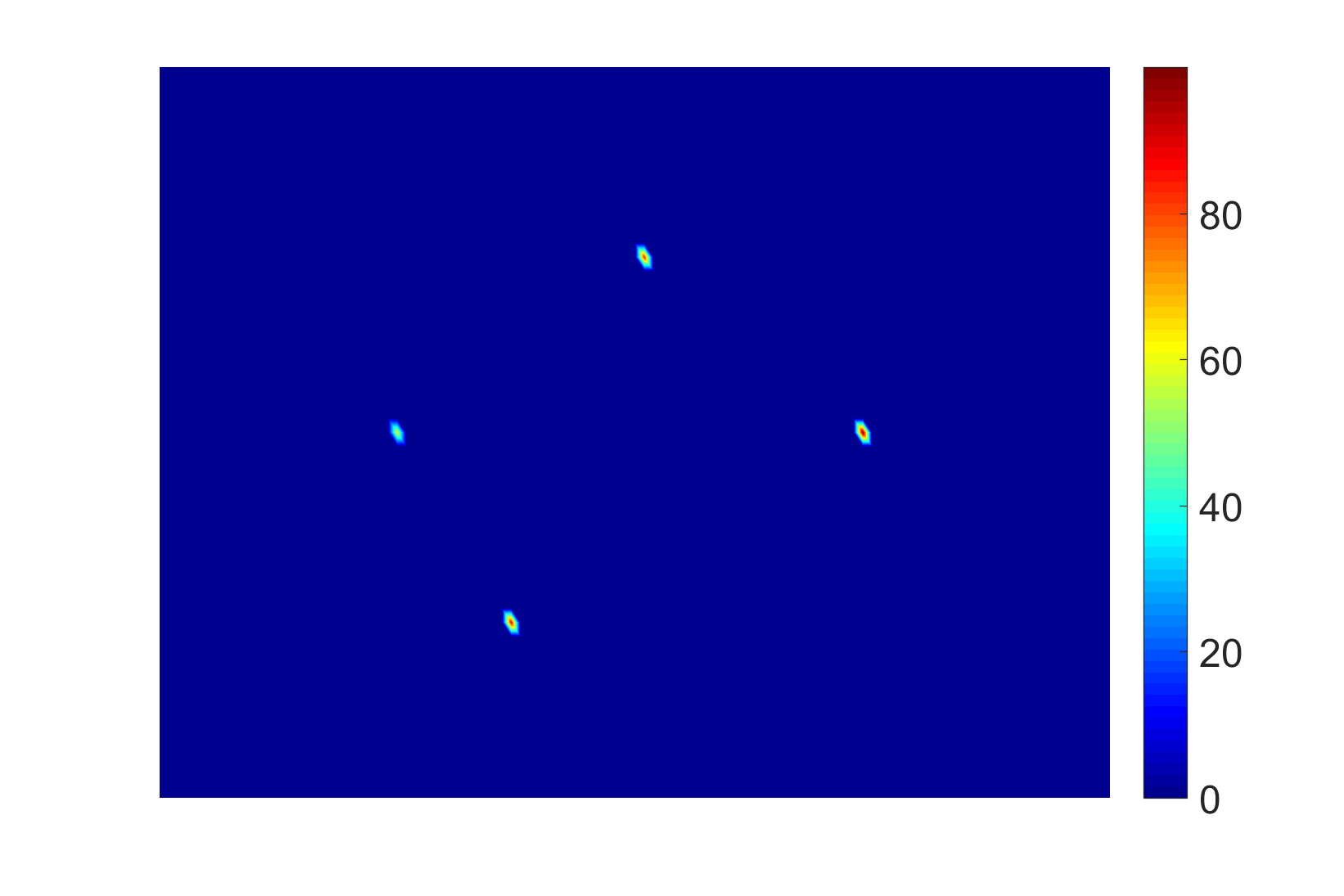}}
	\subfigure[Recovered final state]{\includegraphics[width=0.32\textwidth]{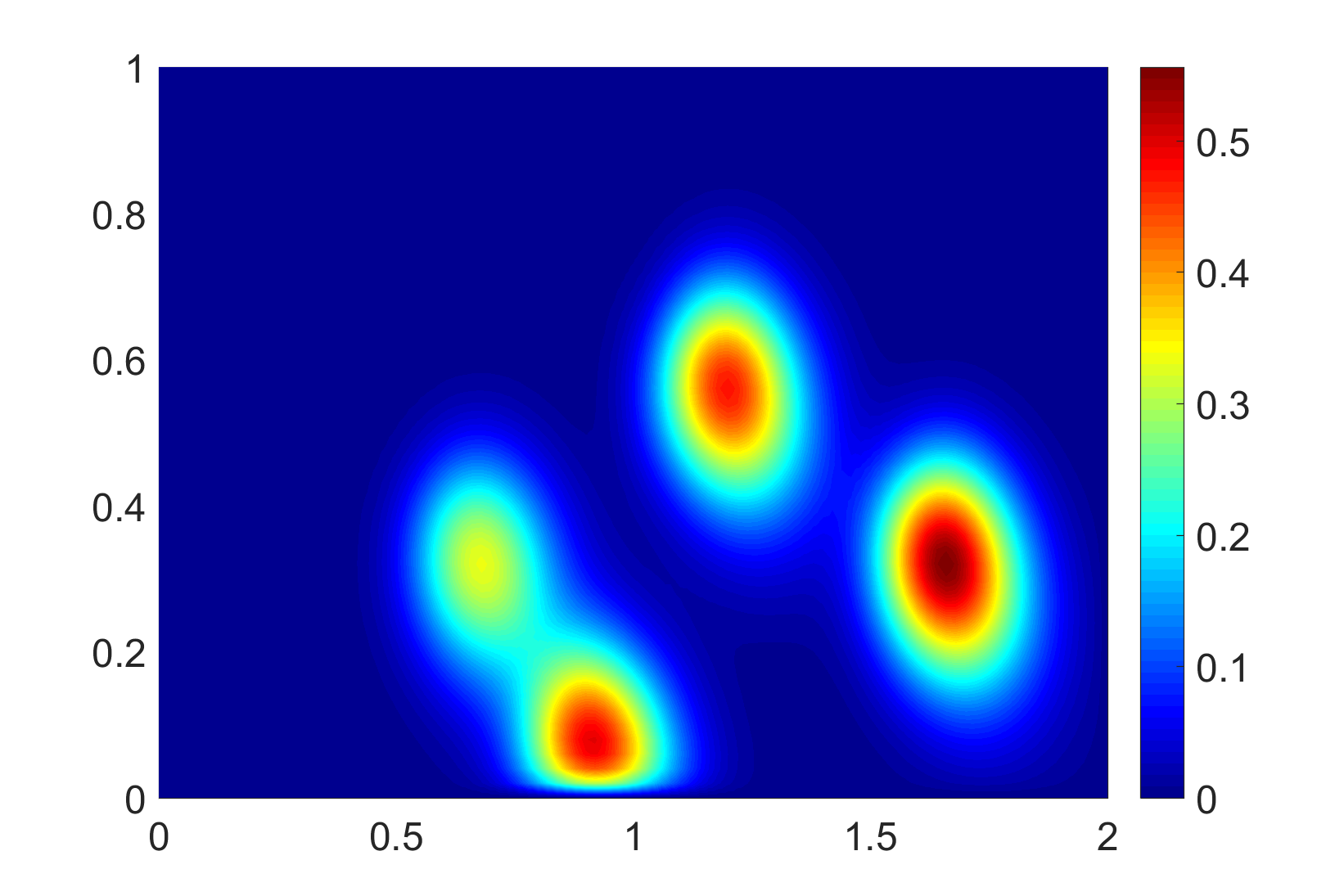}}
\end{figure}

Similarly, the results in Table \ref{tab:performance_comparison} show that also in Case II and Case III, Algorithm \ref{alg:pcalgPDHG} is the most efficient one. Moreover, problem (\ref{eq:SIopt}) allows for a much less expensive numerical resolution than the one in \cite{monge2019sparse}. The recovered initial datum $\widehat{u}_0^*$ by Algorithm \ref{alg:Adjoint} and the corresponding final state $\widehat{u}^*(\cdot,T)$ are displayed in Figure \ref{Case2} (Case II) and Figure \ref{Case3} (Case III). We observe that the locations and the intensities of the sparse initial sources are also recovered very accurately for heterogeneous materials and coupled models.

\begin{figure}[htpb]
	\caption{\small Sparse initial sources identification by Algorithm \ref{alg:Adjoint} for Case II ($d = 0.08$ on $\Omega_1 = (0, 1) \times(0,1)$ and $d = 0.05$ on $\Omega_2 = (1,2) \times (0,1)$; $v = (1, 2)^\top$ on $\Omega$) with a reachable target $u_T$ at $T=0.1$. }\label{Case2}
	\centering
	\subfigure[Reference initial state (front view)]{\includegraphics[width=0.32\textwidth]{./figures/reference_initial_above}}
	\subfigure[Reference initial state (above view)]{\includegraphics[width=0.32\textwidth]{./figures/reference_initial_front}}
	\subfigure[Reachable target $u_T$ ]{\includegraphics[width=0.32\textwidth]{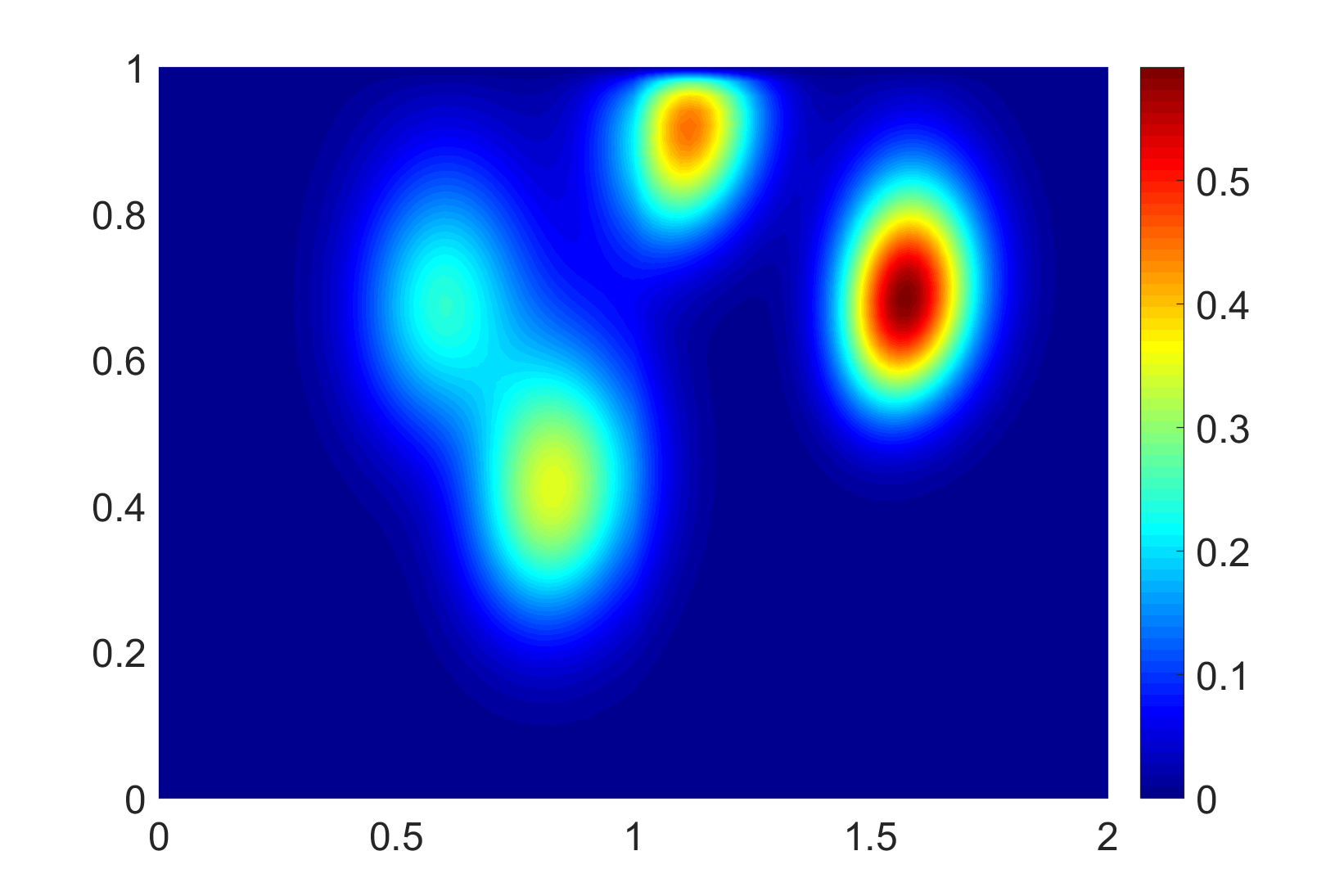}}\\
	\subfigure[Recovered initial state (front view)]{\includegraphics[width=0.32\textwidth]{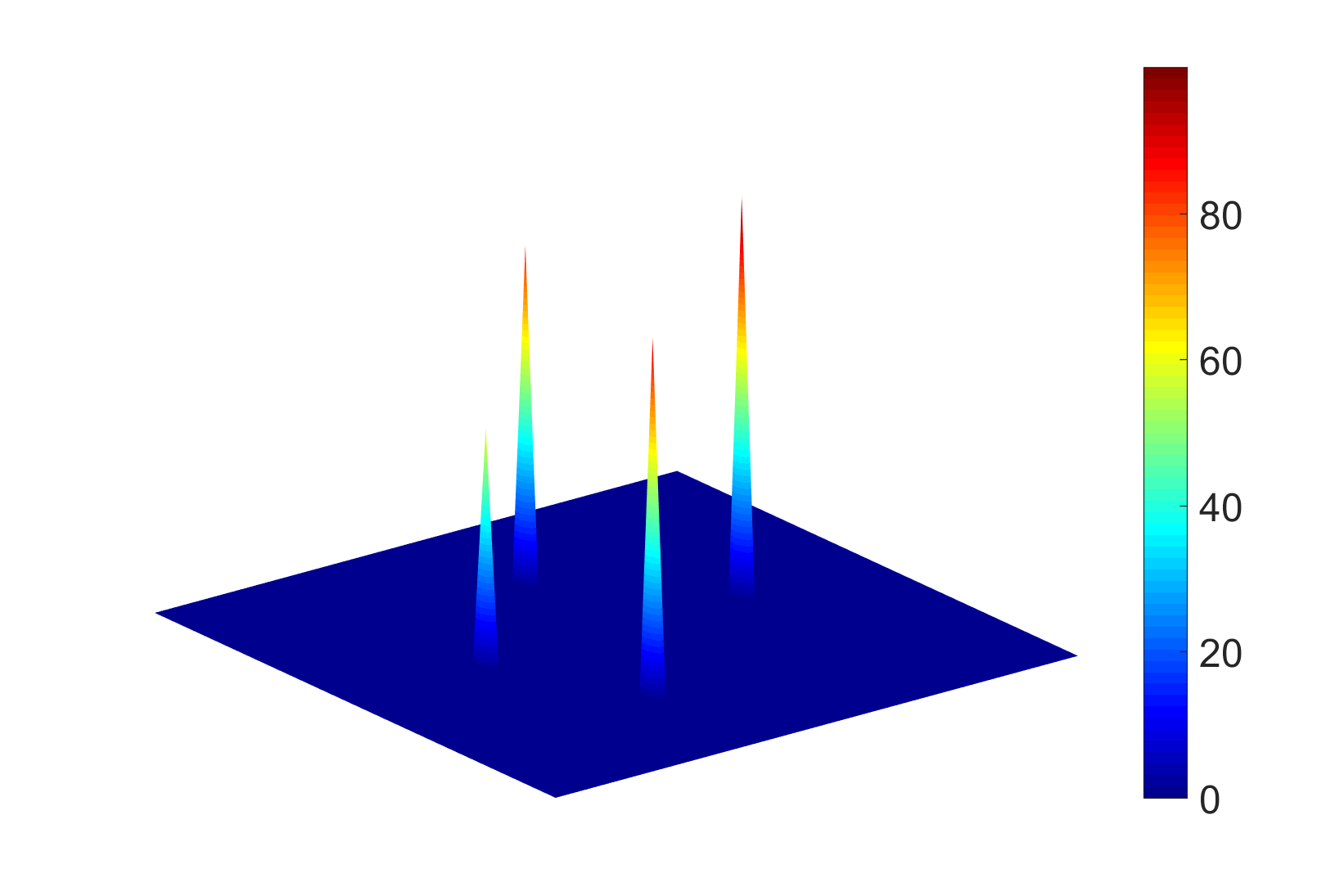}}
\subfigure[Recovered initial state (above view)]{\includegraphics[width=0.32\textwidth]{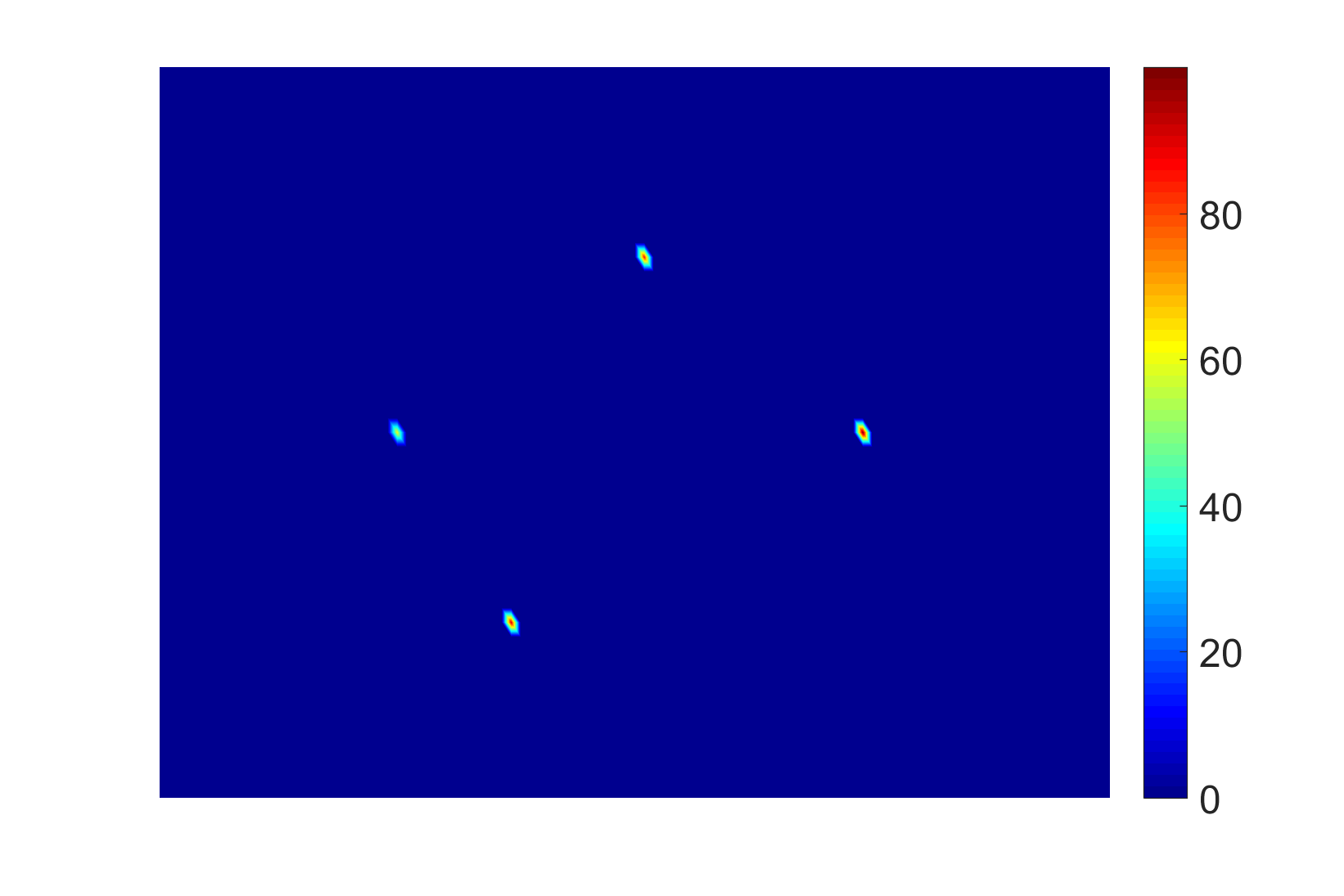}}
	\subfigure[Recovered final state]{\includegraphics[width=0.32\textwidth]{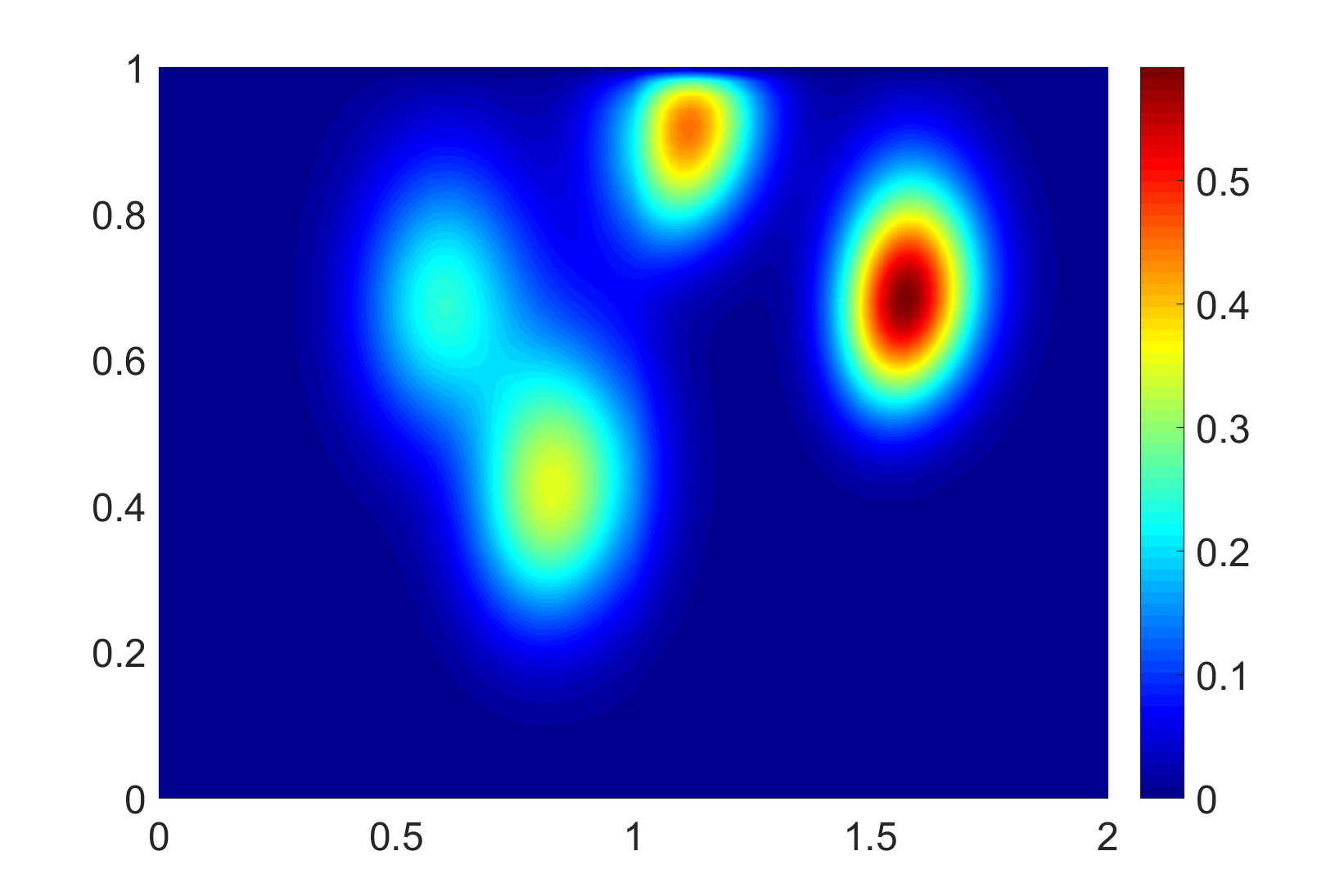}}
	\centering
\end{figure}

\begin{figure}[htpb]
	\caption{\small Sparse initial sources identification by Algorithm \ref{alg:Adjoint} for Case III ($d=0.05$ on $\Omega$; $v = (0,0)^\top$ on $\Omega_1 = (0,1) \times(0,1)$ and $v = (0,-3)^\top$ on $\Omega_2 = (1,2) \times (0,1)$) with a reachable target $u_T$ at $T=0.1$. }\label{Case3}
	\centering
	\subfigure[Reference initial state (front view)]{	\includegraphics[width=0.32\textwidth]{./figures/reference_initial_above}}
	\subfigure[Reference initial state (above view)]{\includegraphics[width=0.32\textwidth]{./figures/reference_initial_front}}
	\subfigure[Reachable target $u_T$]{\includegraphics[width=0.32\textwidth]{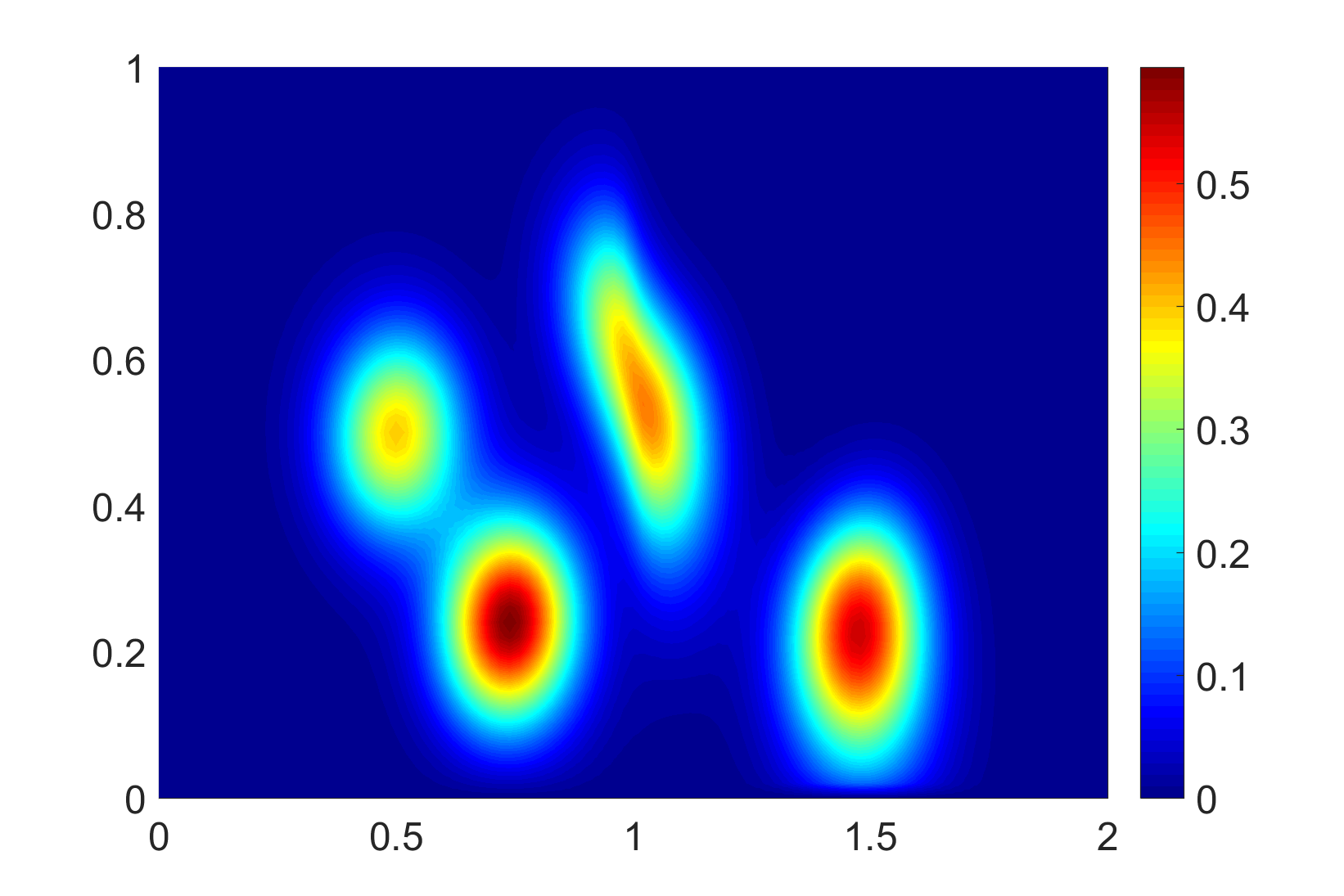}}\\
	\subfigure[Recovered initial state (front view)]{\includegraphics[width=0.32\textwidth]{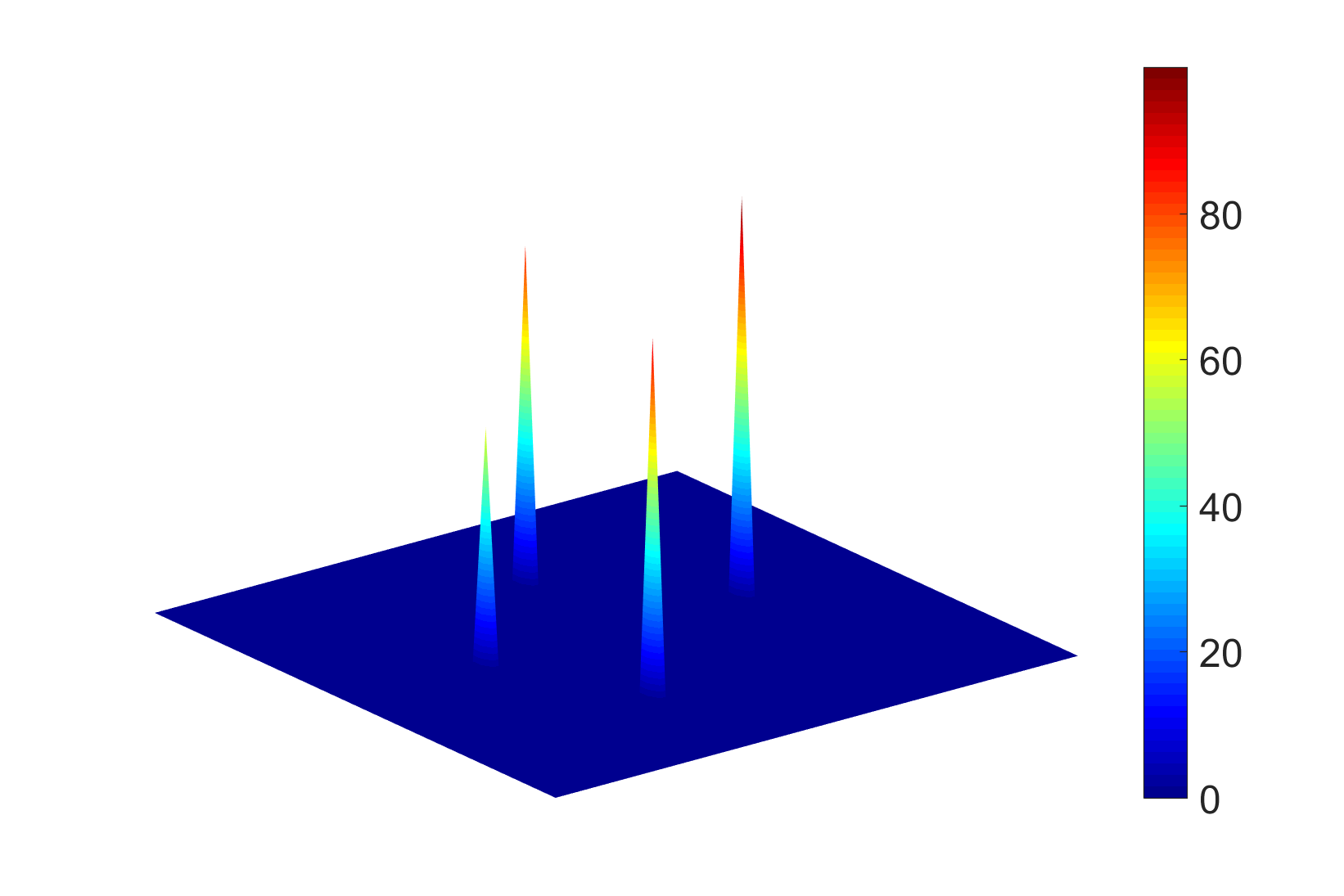}}
	\subfigure[Recovered initial state (above view)]{\includegraphics[width=0.32\textwidth]{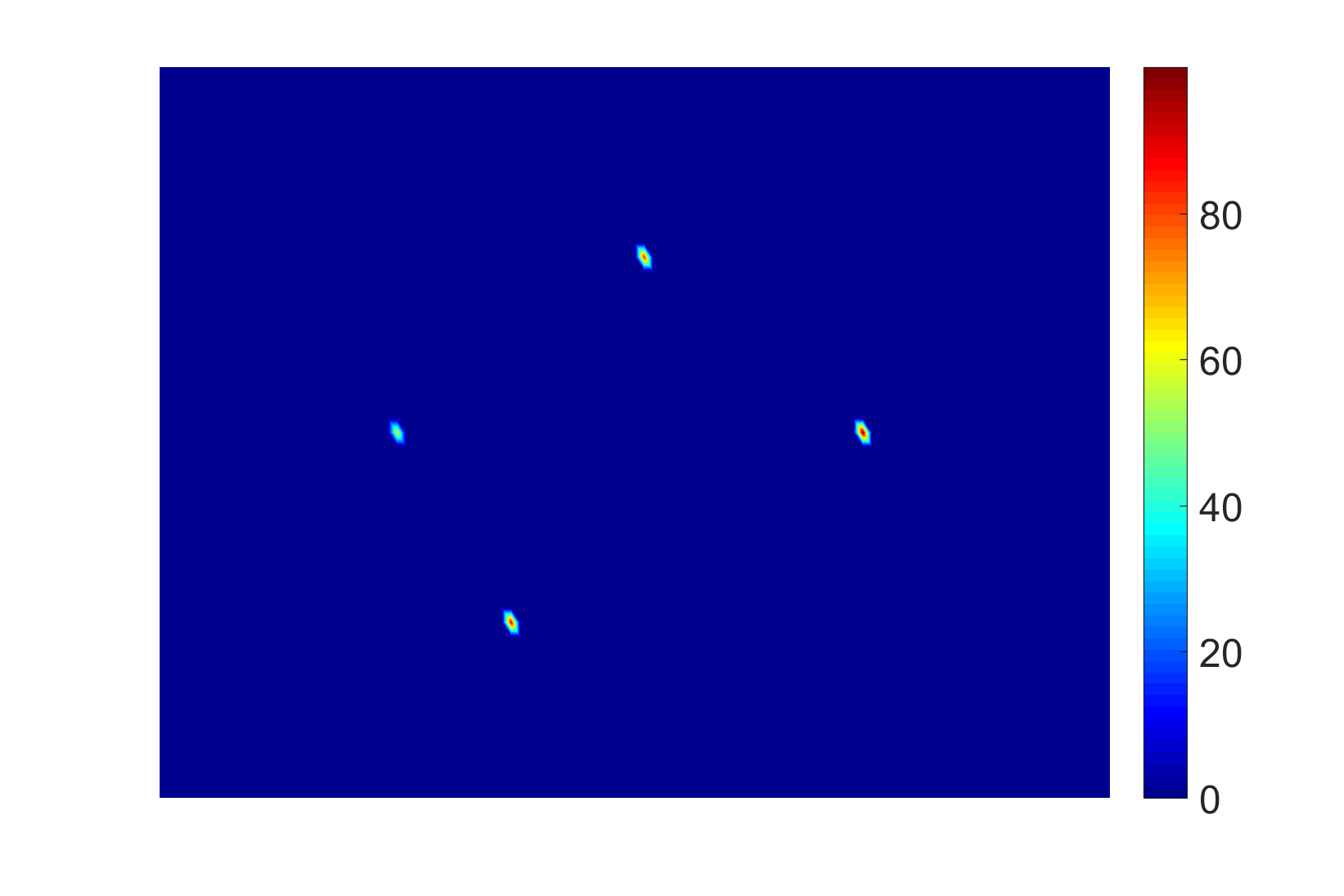}}
		\subfigure[Recovered final state ]{\includegraphics[width=0.32\textwidth]{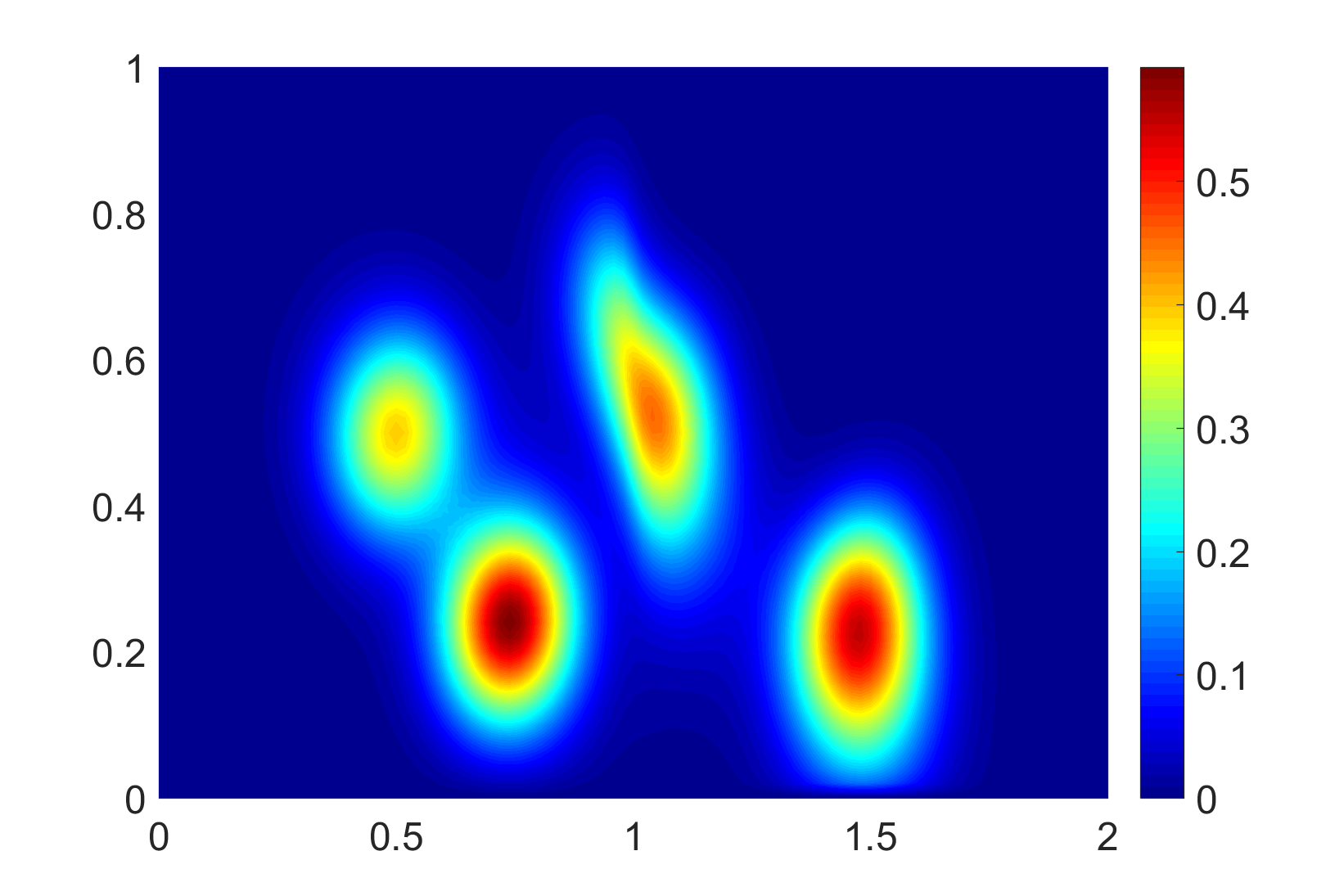}}
	\centering
\end{figure}

\subsection{Noisy observation $u_T$}

In this subsection, we aim to validate the effectiveness and efficiency of Algorithm \ref{alg:Adjoint} for identifying sparse initial sources from some noisy observations. For convenience, we still consider the reference initial datum $\widehat{u}_0$ in (\ref{initial_data}), and the noisy observations at $T=0.1$ are given by $u_T=\mathcal{L}u_0+\delta,$ where $\delta\in L^2(\Omega)$ is a noise term satisfying
$
	\frac{\|\mathcal{L}u_0-u_T\|_{L^2(\Omega)}}{\|\mathcal{L}u_0\|_{L^2(\Omega)}}\approx10\%.
$

As in the previous subsections, we employ Algorithm \ref{alg:pcalgPDHG} to solve the optimal control problem (\ref{eq:SIopt}). We observe that the iteration numbers of Algorithm \ref{alg:pcalgPDHG} for all test cases are almost the same as the reachable target case. Furthermore, mesh-independent property can also be observed. Hence, we can conclude that the numerical efficiency of Algorithm \ref{alg:pcalgPDHG} is robust with respect to noisy observations.

The  initial datum $\widehat{u}_0^*$ recovered from the noisy observations $u_T$ by Algorithm \ref{alg:Adjoint} and the associated final state $\widehat{u}^*(\cdot,T)$ for Case I-III are respectively presented in Figures \ref{Case1_noisy}, \ref{Case2_noisy} and \ref{Case3_noisy}. It is easy to observe that both the locations and the intensities of the sparse initial source are recovered accurately from the noisy observations.


\begin{figure}[htpb]
	\caption{\small Sparse initial sources identification by Algorithm \ref{alg:Adjoint} for Case I ($d=0.05,v=(2,-2)^{\top}$ on $\Omega$) with a noisy observation $u_T$ at $T=0.1$.}\label{Case1_noisy}
	\centering
	\subfigure[Reference initial state (front view)]{\includegraphics[width=0.32\textwidth]{./figures/reference_initial_above}}
	\subfigure[Reference initial state (above view)]{\includegraphics[width=0.32\textwidth]{./figures/reference_initial_front}}
	\subfigure[Noisy observation $u_T$ ]{\includegraphics[width=0.32\textwidth]{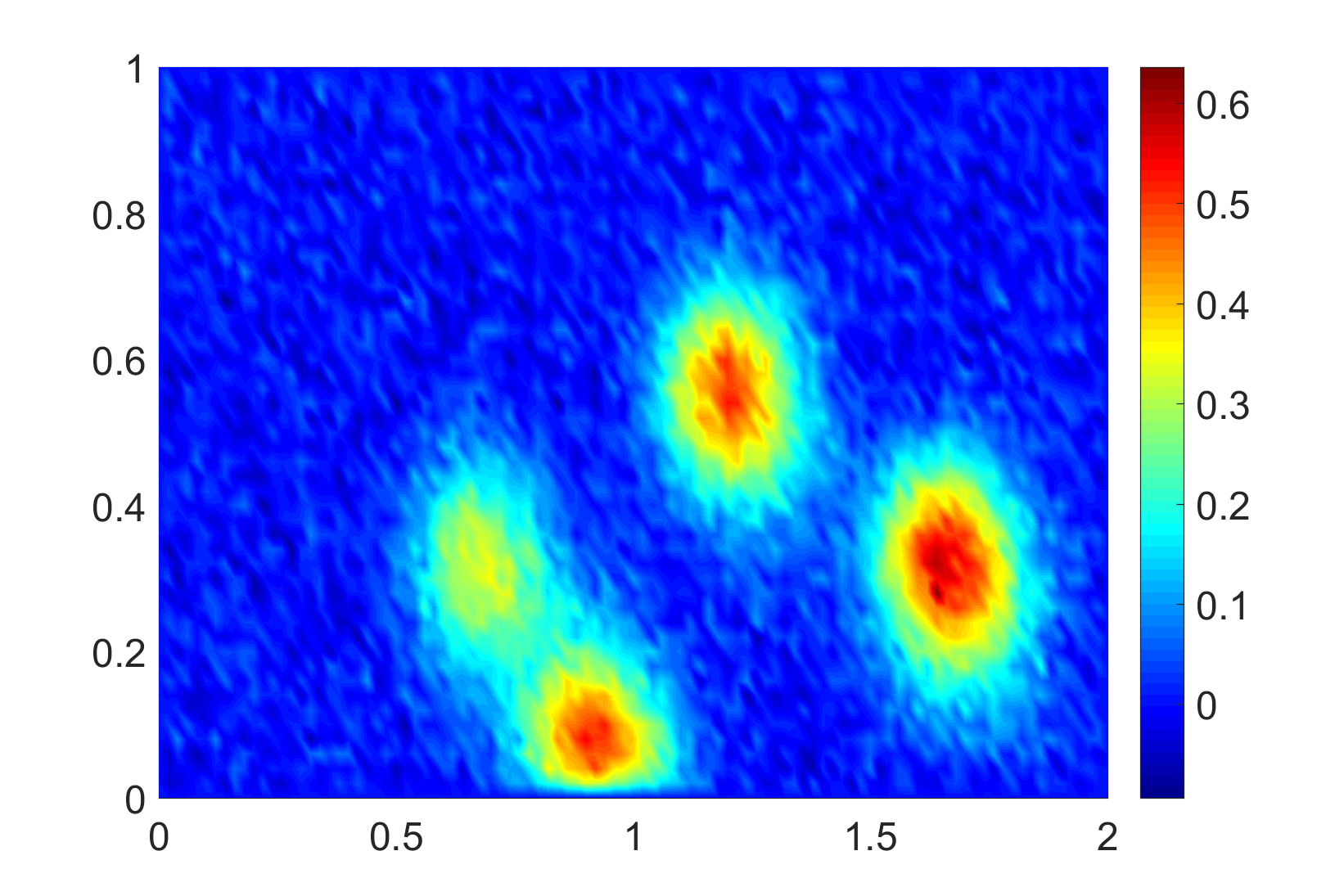}}\\
	\subfigure[Recovered initial state (front view)]{\includegraphics[width=0.32\textwidth]{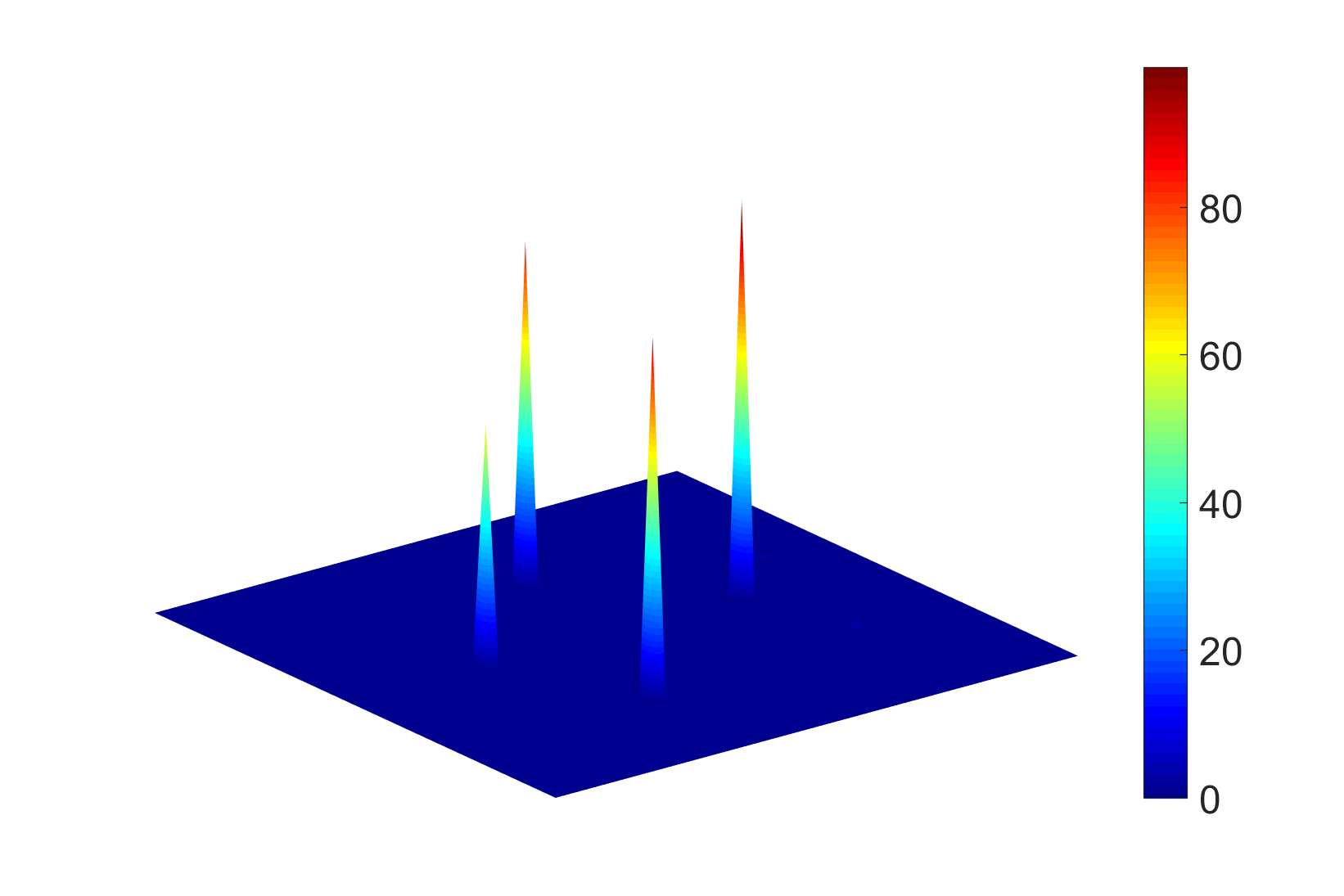}}
	\subfigure[Recovered initial state (above view)]{\includegraphics[width=0.32\textwidth]{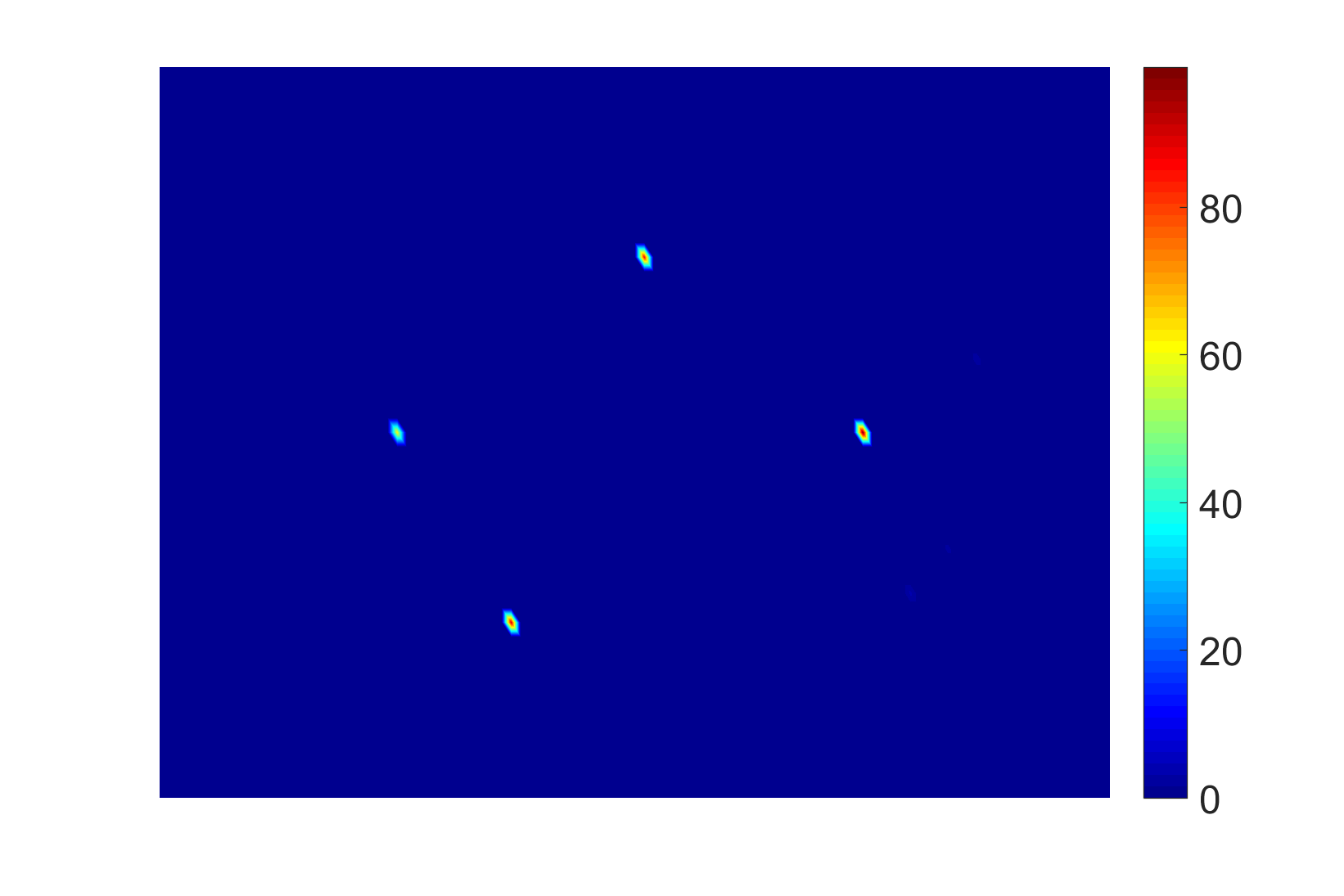}}
	\subfigure[Recovered final state]{\includegraphics[width=0.32\textwidth]{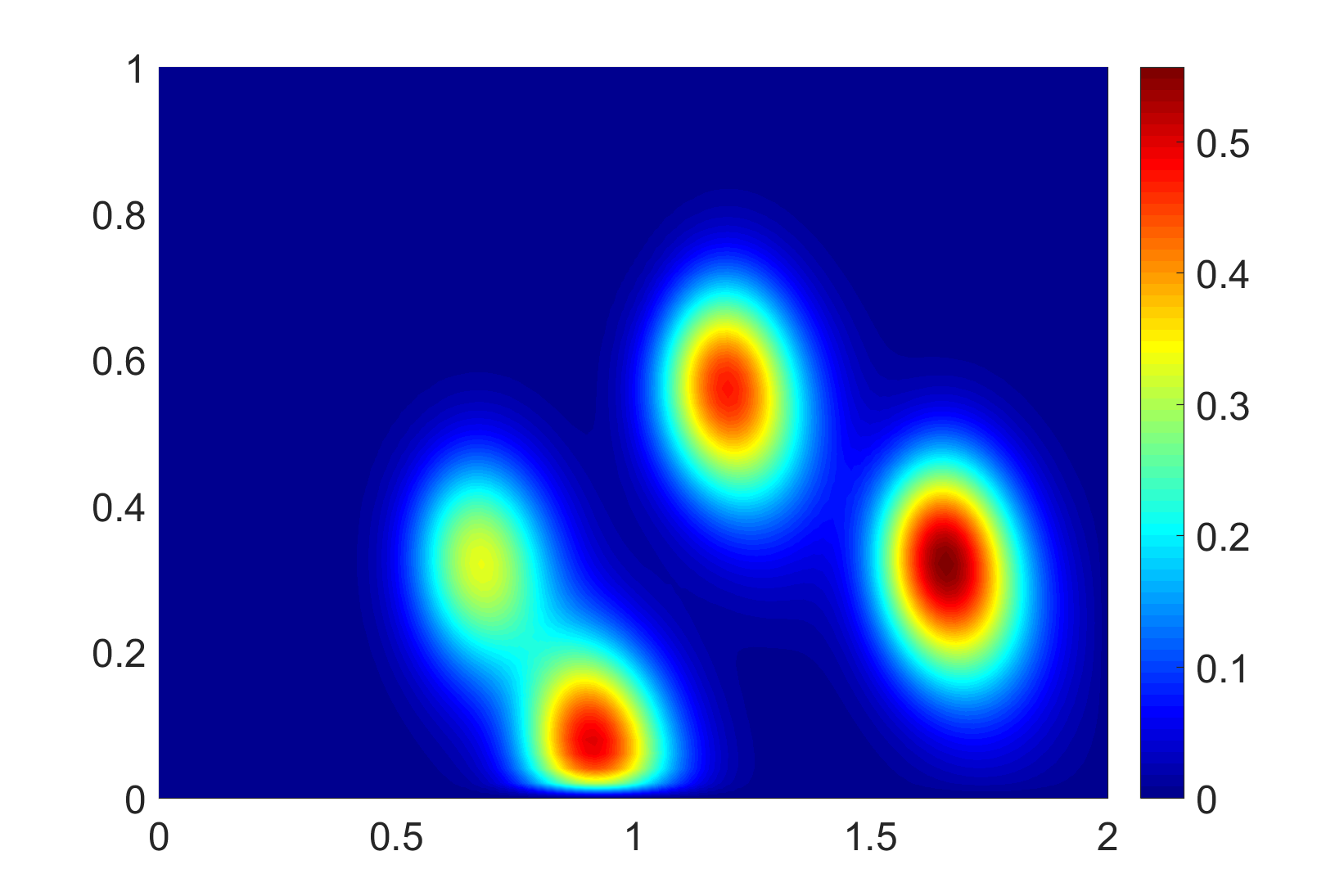}}
	\centering
\end{figure}

\begin{figure}[htpb]
	\caption{\small Sparse initial sources identification by Algorithm \ref{alg:Adjoint} for Case II ($d = 0.08$ on $\Omega_1 = (0, 1) \times(0,1)$ and $d = 0.05$ on $\Omega_2 = (1,2) \times (0,1)$; $v = (1, 2)^\top$ on $\Omega$) with a noisy observation $u_T$ at $T=0.1$.}\label{Case2_noisy}
	\centering
	\subfigure[Reference initial state (front view)]{\includegraphics[width=0.32\textwidth]{./figures/reference_initial_above}}
	\subfigure[Reference initial state (above view)]{\includegraphics[width=0.32\textwidth]{./figures/reference_initial_front}}
	\subfigure[Noisy observation $u_T$ ]{\includegraphics[width=0.32\textwidth]{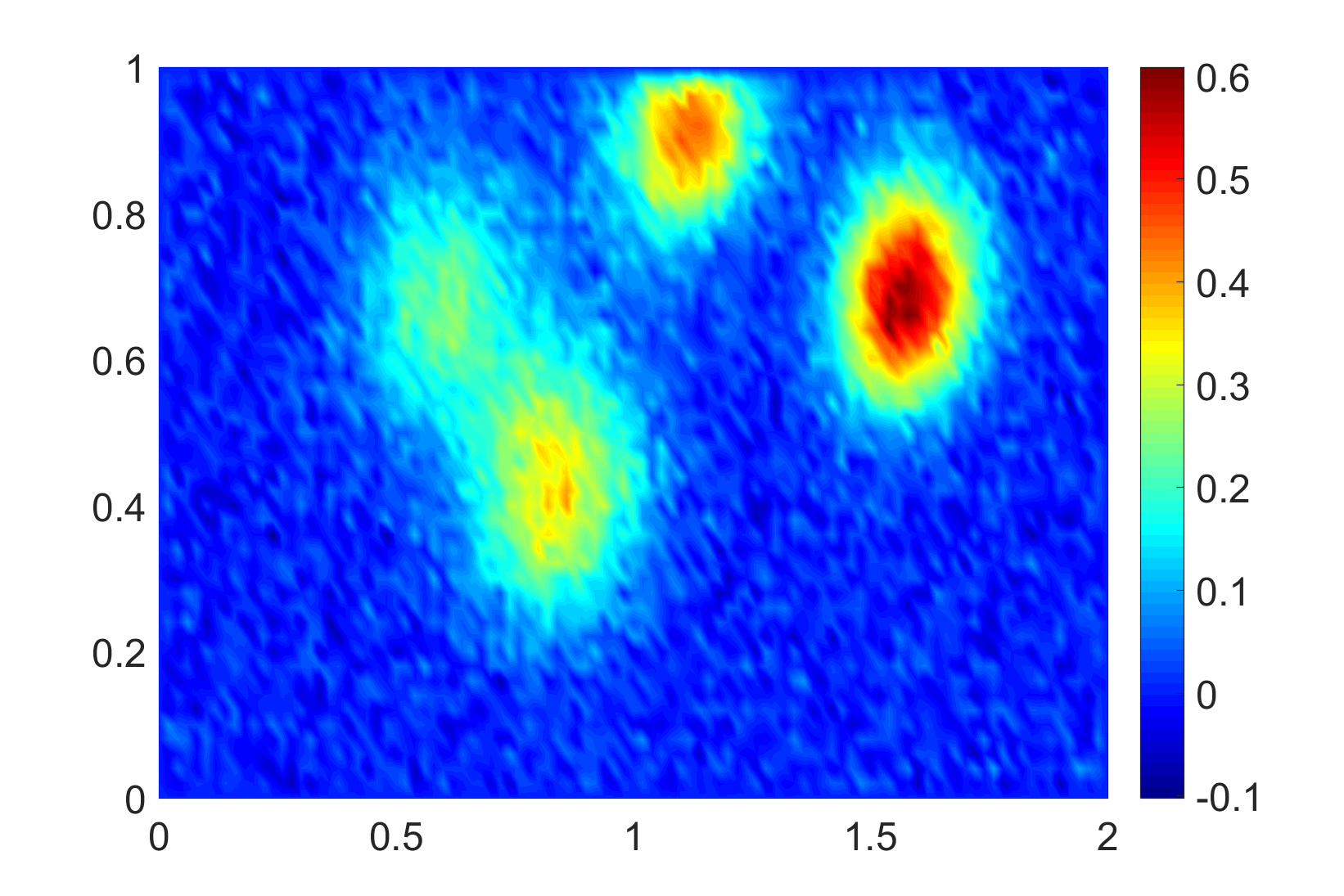}}\\
	\subfigure[Recovered initial state (front view)]{\includegraphics[width=0.32\textwidth]{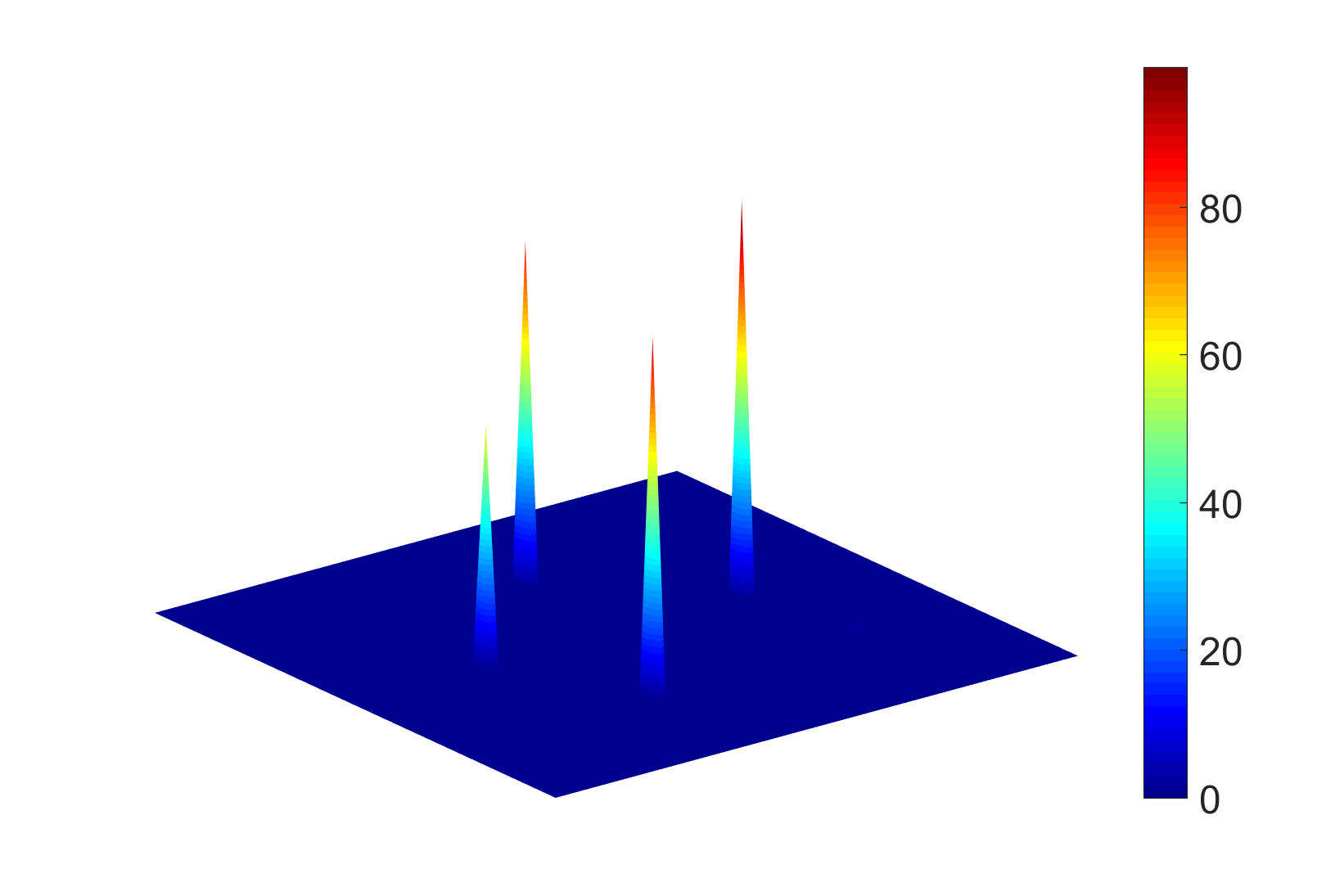}}
	\subfigure[Recovered initial state (above view)]{\includegraphics[width=0.32\textwidth]{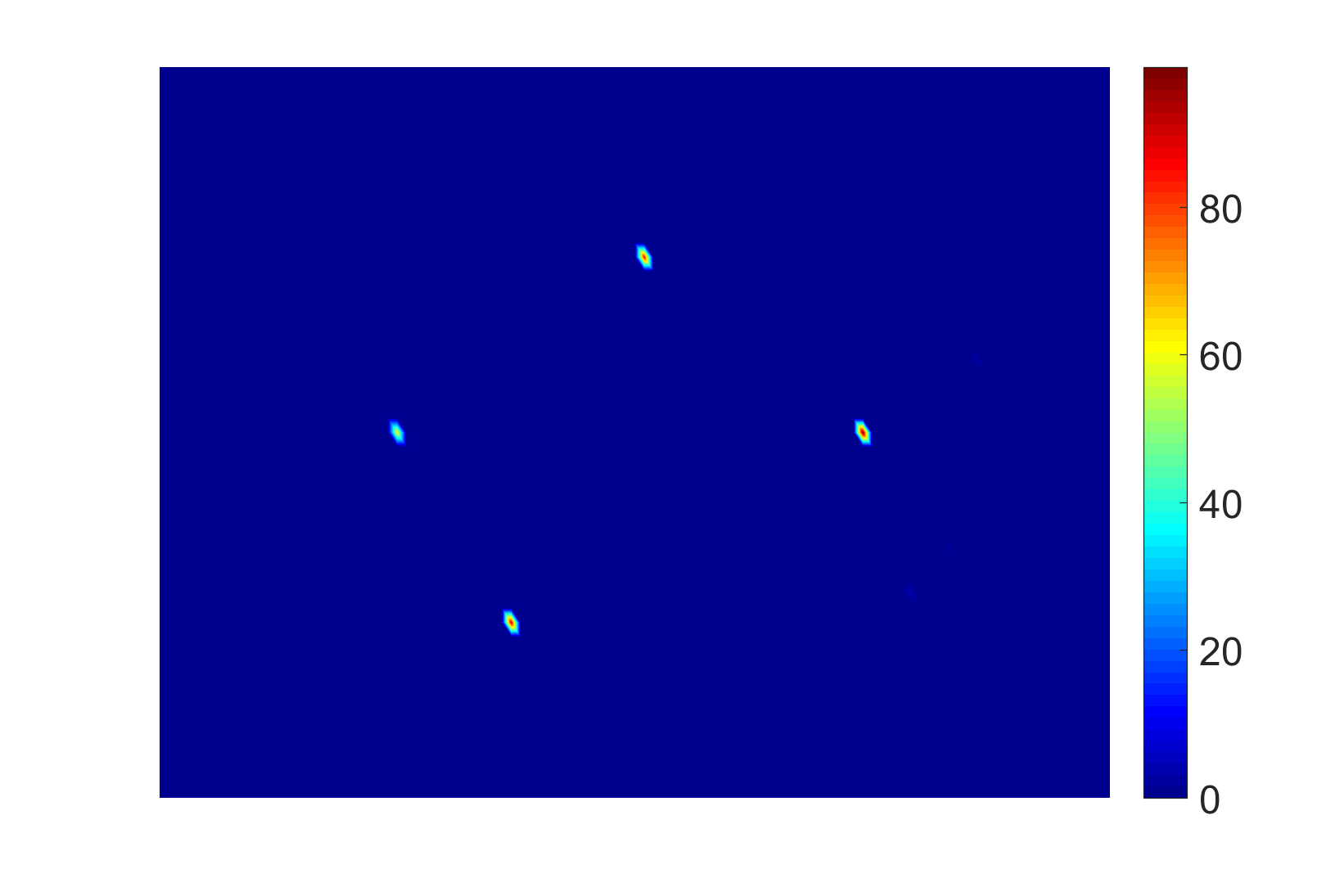}}
	\subfigure[Recovered final state ]{\includegraphics[width=0.32\textwidth]{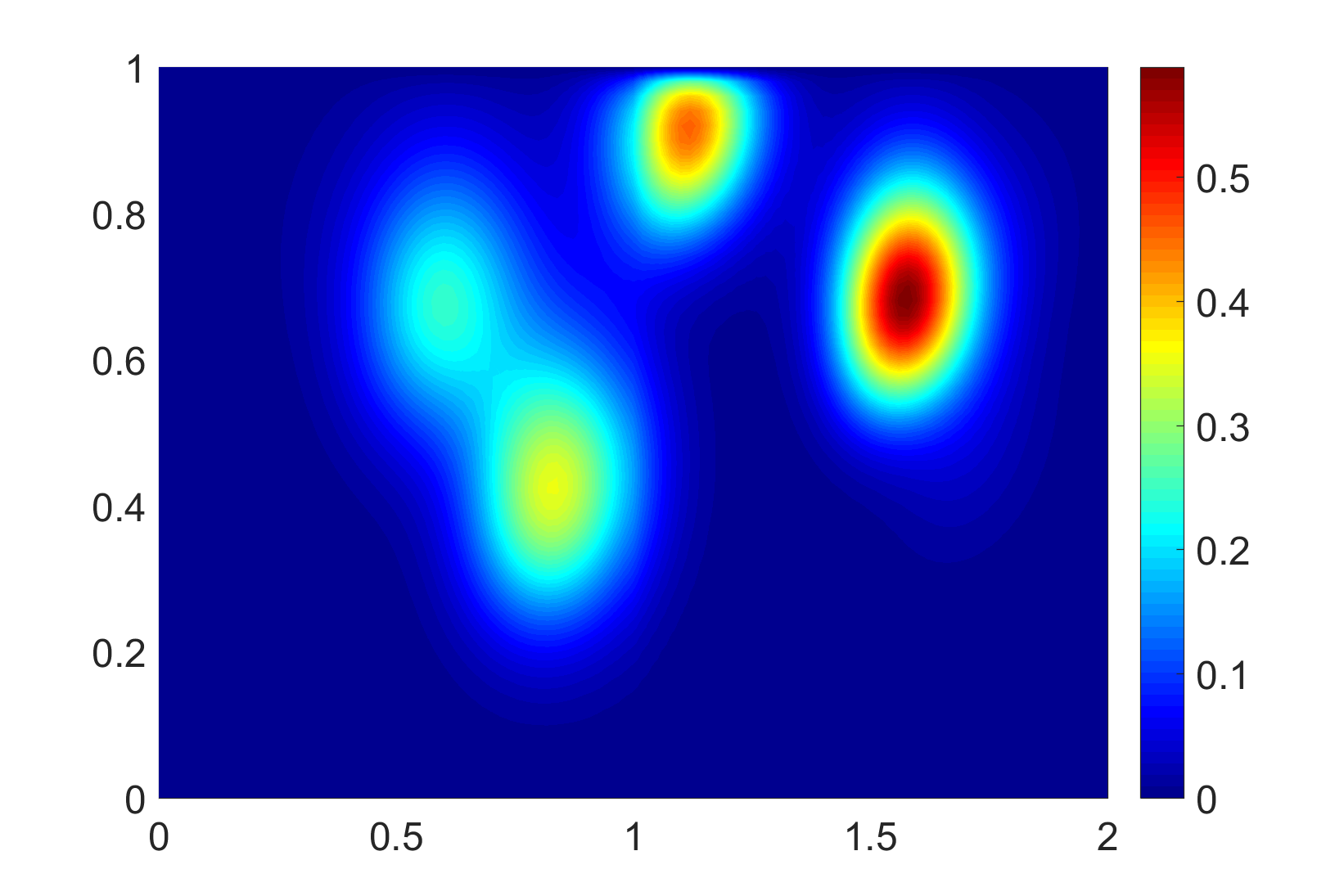}}
	\centering
\end{figure}
%
\begin{figure}[htpb]
	\caption{\small Sparse initial sources identification by Algorithm \ref{alg:Adjoint} for Case III ($d=0.05$ on $\Omega$; $v = (0,0)^\top$ on $\Omega_1 = (0,1) \times(0,1)$ and $v = (0,-3)^\top$ on $\Omega_2 = (1,2) \times (0,1)$) with a noisy observation $u_T$ at $T=0.1$.}\label{Case3_noisy}
	\centering
	\subfigure[Reference initial state (front view)]{	\includegraphics[width=0.32\textwidth]{./figures/reference_initial_above}}
	\subfigure[Reference initial state (above view)]{\includegraphics[width=0.32\textwidth]{./figures/reference_initial_front}}
	\subfigure[Noisy observation  $u_T$]{\includegraphics[width=0.32\textwidth]{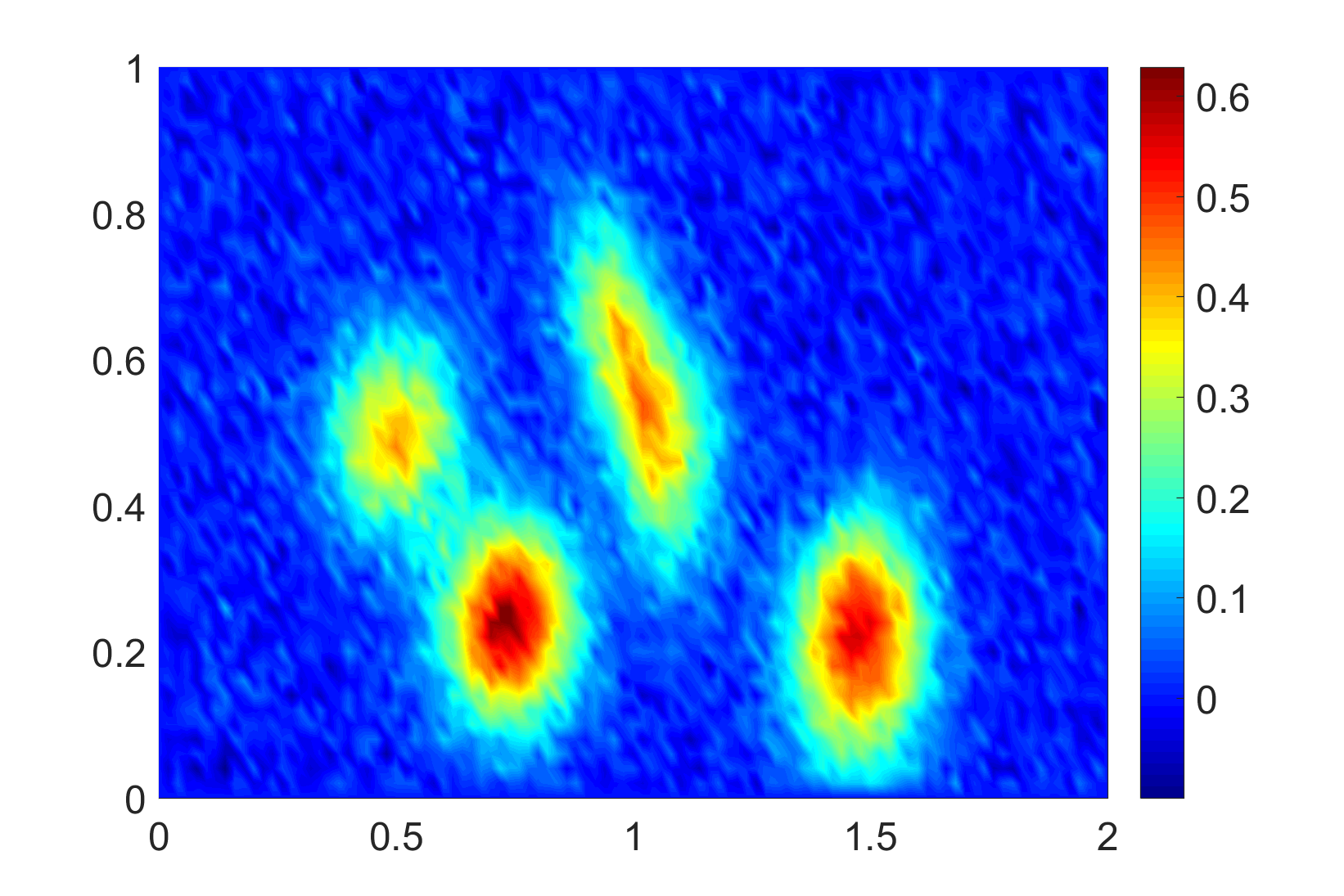}}\\
	\subfigure[Recovered initial state (front view)]{\includegraphics[width=0.32\textwidth]{./figures/recovered_initial_front_c3}}
	\subfigure[Recovered initial state (above view)]{\includegraphics[width=0.32\textwidth]{./figures/recovered_initial_above_c3}}
	\subfigure[Recovered final state ]{\includegraphics[width=0.32\textwidth]{./figures/recovered_final_c3}}
	\centering
\end{figure}

\subsection{Long time horizon cases}

Our simulations have shown that Algorithm \ref{alg:Adjoint} is capable of accurately recovering the sparse initial source from a reachable target or noisy observation $u_T$ at $T=0.1$. On the other hand, if the final time $T$ increases, Problem \ref{SIproblem} becomes strongly ill-posed and Algorithm \ref{alg:Adjoint} cannot identify a sparse initial condition correctly, as it can be appreciated in Figure \ref{fig:long_time}. We observe that the recovered final state $u^*(T)$ is close to the target $u_T$, but the recovered initial source $\widehat{u}_0^*$ and the reference $\widehat{u}_0$ do not coincide. This validate the extreme ill-posedness of the sparse
initial source identification problem in long time horizons, as it shows that a small perturbation on the final
state may cause an arbitrarily large error on the initial datum.

\begin{figure}[htpb]\caption{\small Sparse initial sources identification by Algorithm \ref{alg:Adjoint} for Case I ($d=0.05,v=(2,-2)^{\top}$ on $\Omega$) with a reachable target $u_T$ at {$T=1$}.}\label{fig:long_time}
	\centering
	\subfigure[Reference initial datum $\widehat{u}_0$]{\includegraphics[width=0.24\textwidth]{./figures/reference_initial_above}}
	\subfigure[Referene final state $u_T$]{\includegraphics[width=0.24\textwidth]{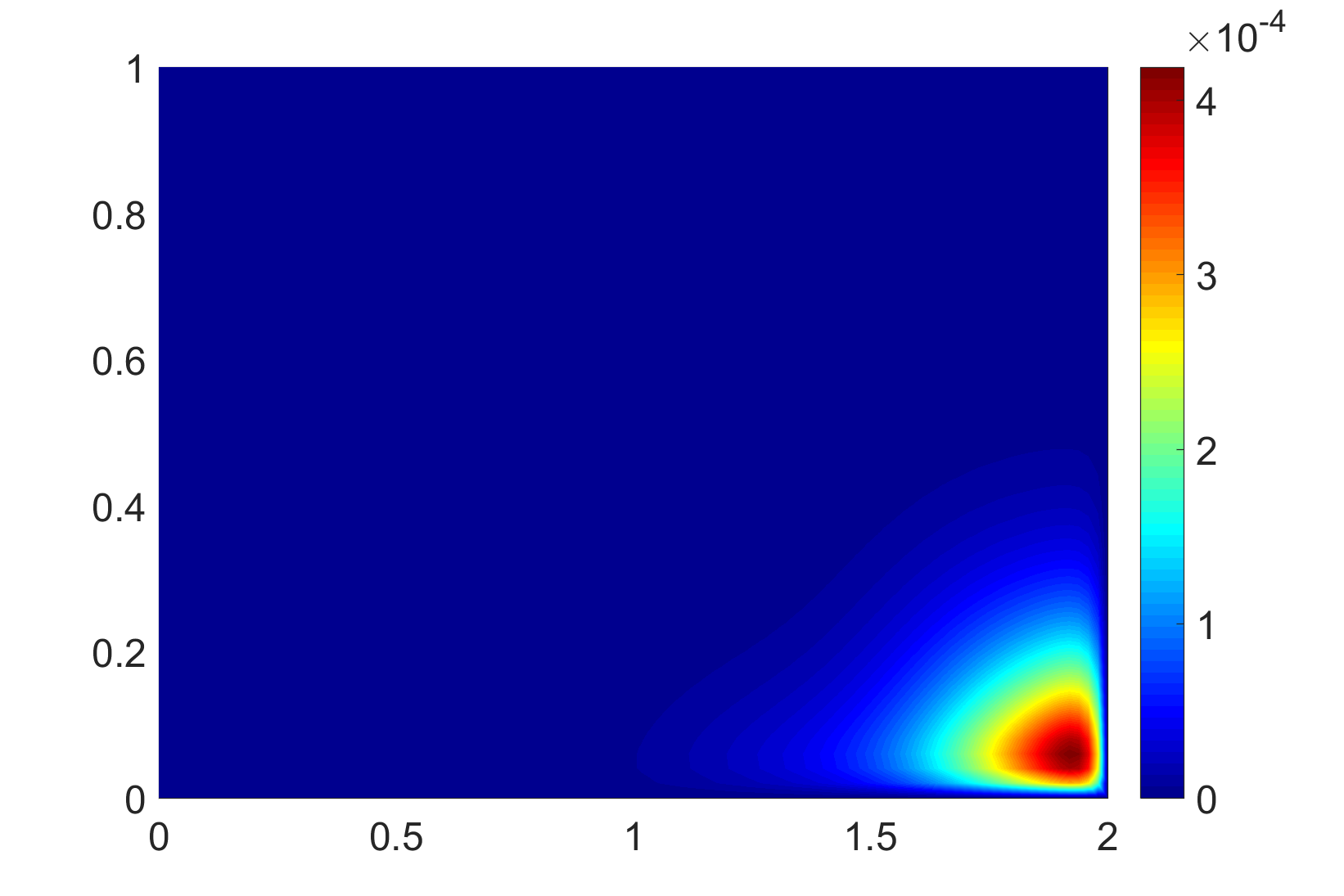}}
	\subfigure[Recovered initial datum $\widehat{u}_0^*$]{\includegraphics[width=0.24\textwidth]{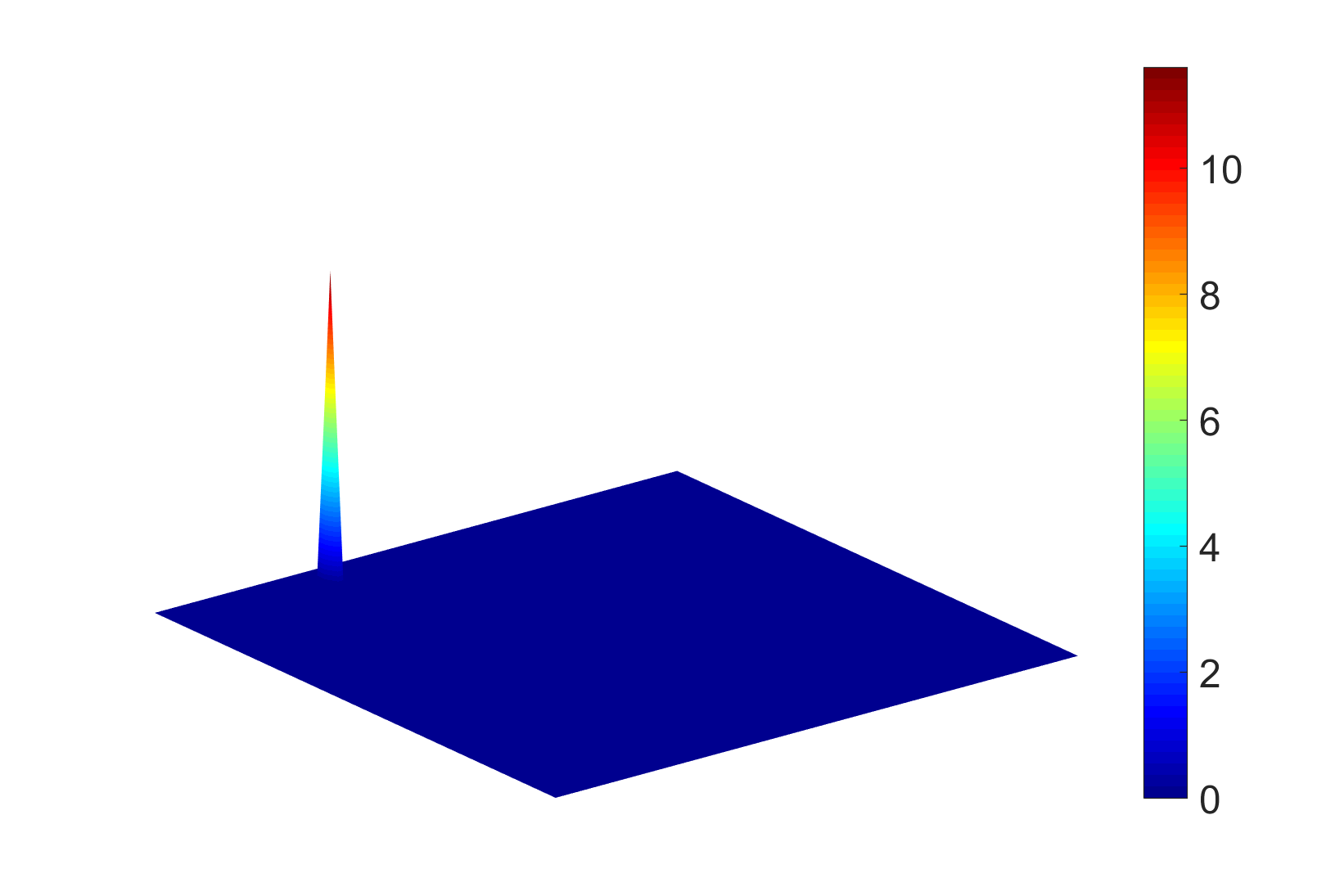}}
	\subfigure[Recovered final state $u(T)$]{\includegraphics[width=0.24\textwidth]{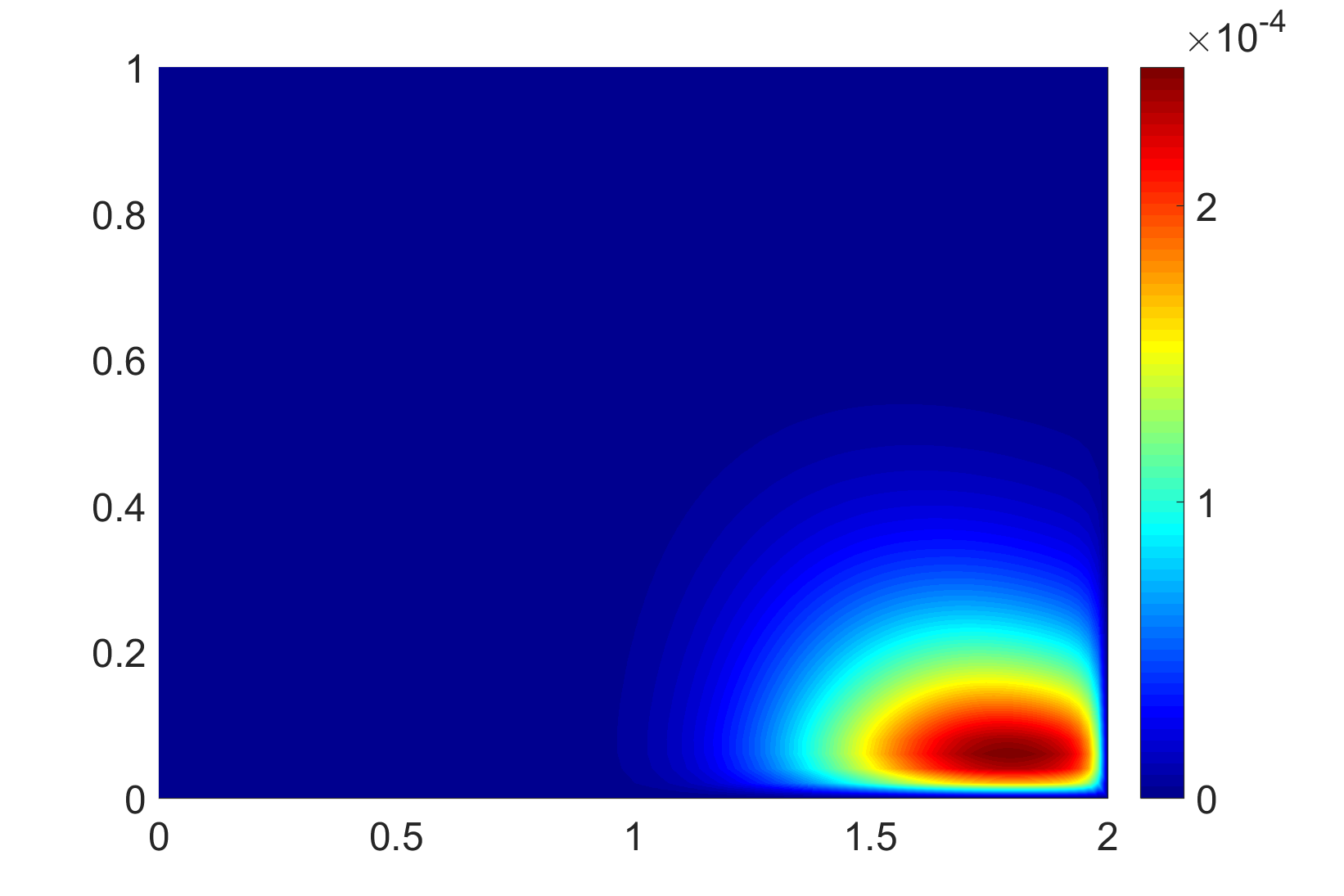}}
\end{figure}

The above issue caused by long time horizons has also been observed in some research works on Backward Heat Conduction Problems (BHCPs), see e.g., \cite{liu2004,mera2005}. Typically, a BHCP aims at estimating an initial condition of the heat equation for a given final state distribution, which is closely related to Problem \ref{SIproblem} but without the sparsity assumption (\ref{eq:deltas}). Based on the group preserving scheme \cite{liu2001}, a Lie-group shooting method was proposed in \cite{chen2018}. When the initial condition to be estimated is smooth or its support is sufficiently large, this Lie-group shooting method can address BHCPs in long time horizons successfully. However, the Lie-group shooting method cannot be extended directly to Problem \ref{SIproblem} because the initial condition to be recovered therein is nonsmooth and has a support of Lebesgue measure zero. We also combined the group preserving scheme into Algorithm \ref{alg:Adjoint} and obtained a new numerical approach for addressing Problem \ref{SIproblem}. By some numerical simulations, we found that this new approach cannot improve the performance of Algorithm \ref{alg:Adjoint} when $T$ is large, while it is less efficient than Algorithm \ref{alg:Adjoint} when $T$ is small.

Additionally, we note that the admissible final time at which the sparse initial source can be identified numerically varies from case to case. It is highly related to the diffusivity parameter, the velocity field of the advection, the geometry of the domain, and the locations and intensities of the initial source to be identified, etc. To elaborate, we remove two Dirac deltas from (\ref{initial_data}) and consider the following reference initial datum:
\begin{equation*}
\widehat{u}_0=100\delta(1.5,0.5)+60\delta(0.5,0.5).
\end{equation*}

We set $d=0.05$ and $v = (0,0)^\top$ on $\Omega$, and $T=1$. We implement Algorithm \ref{alg:Adjoint} to this test case and the numerical results are reported in Figure \ref{fig:longtime2}. We observe that the initial datum can be accurately recovered from the final target $u_T$ at $T=1$. Compared with the results in Figure \ref{fig:long_time}, it is easy to see that the admissible final time varies from case to case.
\begin{figure}[htpb]
	\caption{\small Sparse initial sources identification by Algorithm \ref{alg:Adjoint} for Case III ($d=0.05$ and $v = (0,0)^\top$ on $\Omega$) with a reachable target $u_T$ at $T=1$.}\label{fig:longtime2}
	\centering
	\subfigure[Reference initial state (front view)]{	\includegraphics[width=0.32\textwidth]{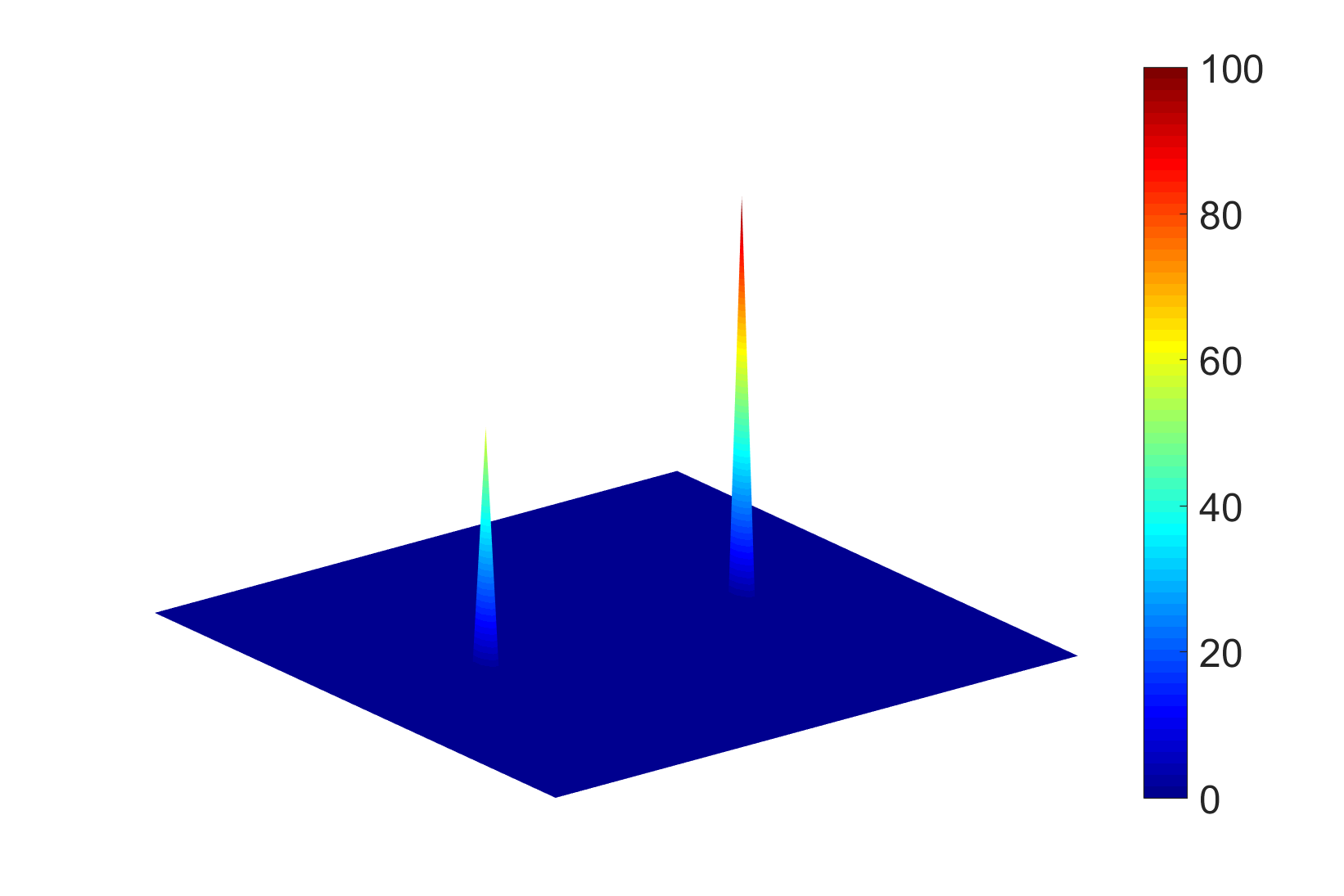}}
	\subfigure[Reference initial state (above view)]{\includegraphics[width=0.32\textwidth]{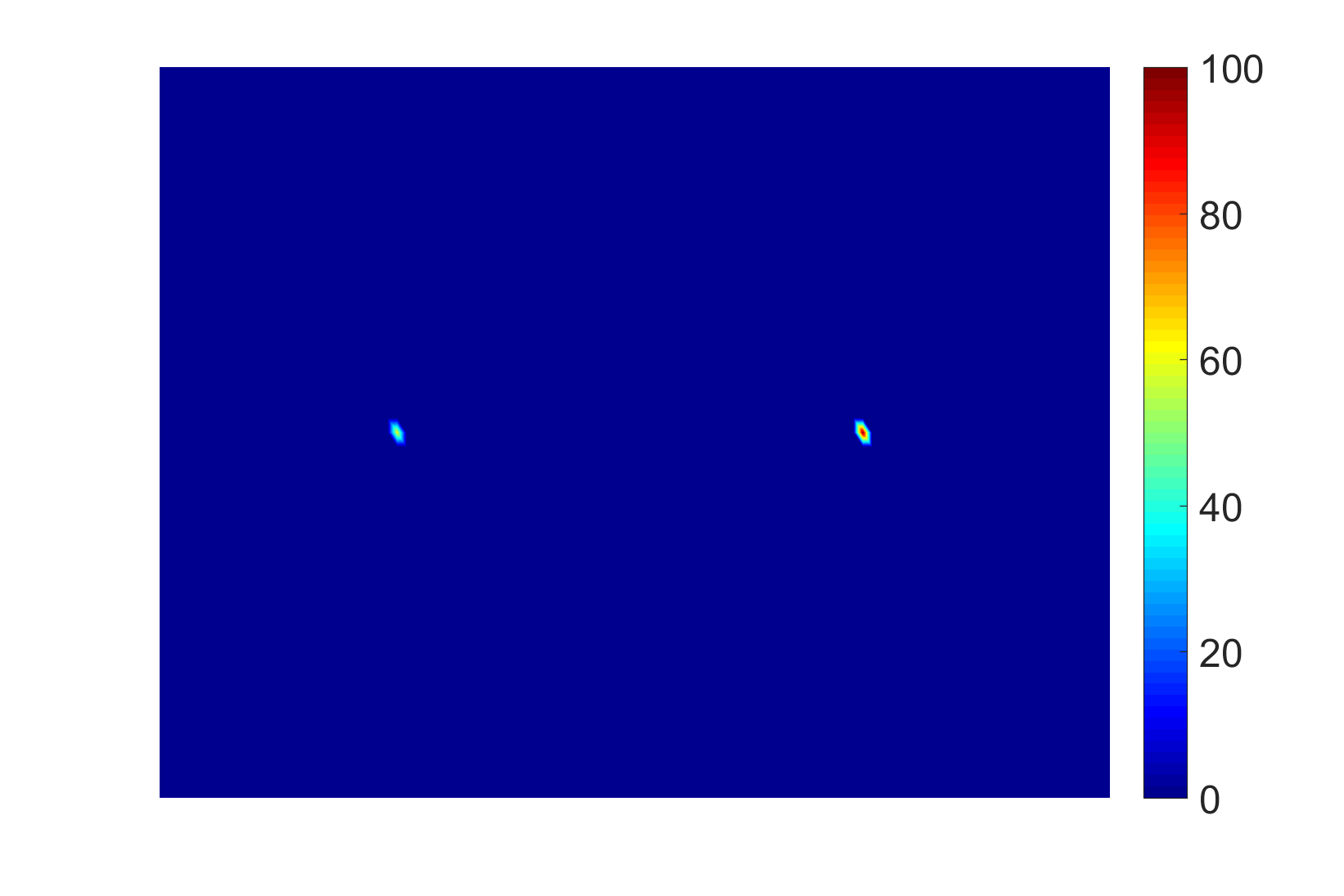}}
	\subfigure[Reachable target  $u_T$]{\includegraphics[width=0.32\textwidth]{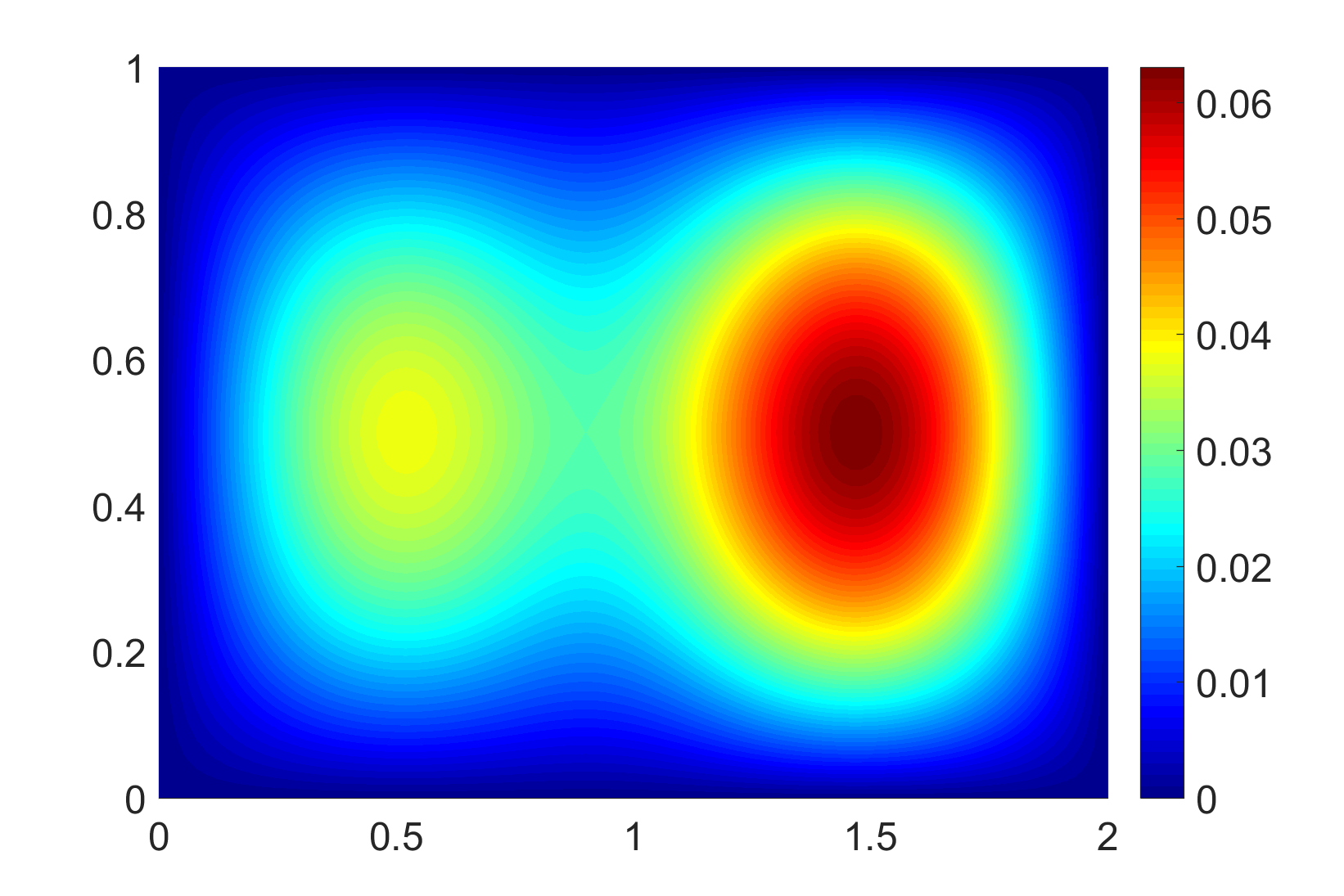}}\\
	\subfigure[Recovered initial state (front view)]{\includegraphics[width=0.32\textwidth]{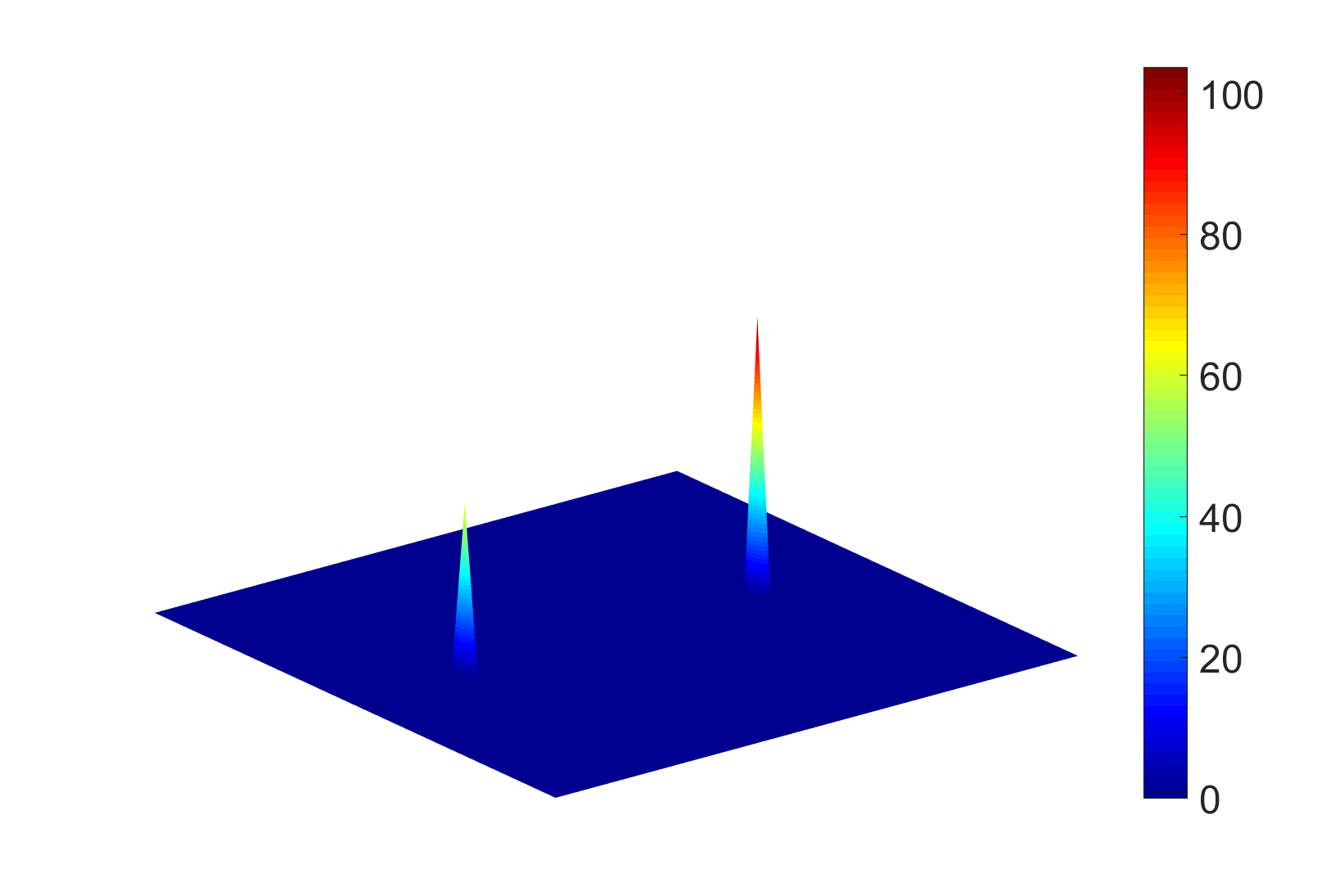}}
	\subfigure[Recovered initial state (above view)]{\includegraphics[width=0.32\textwidth]{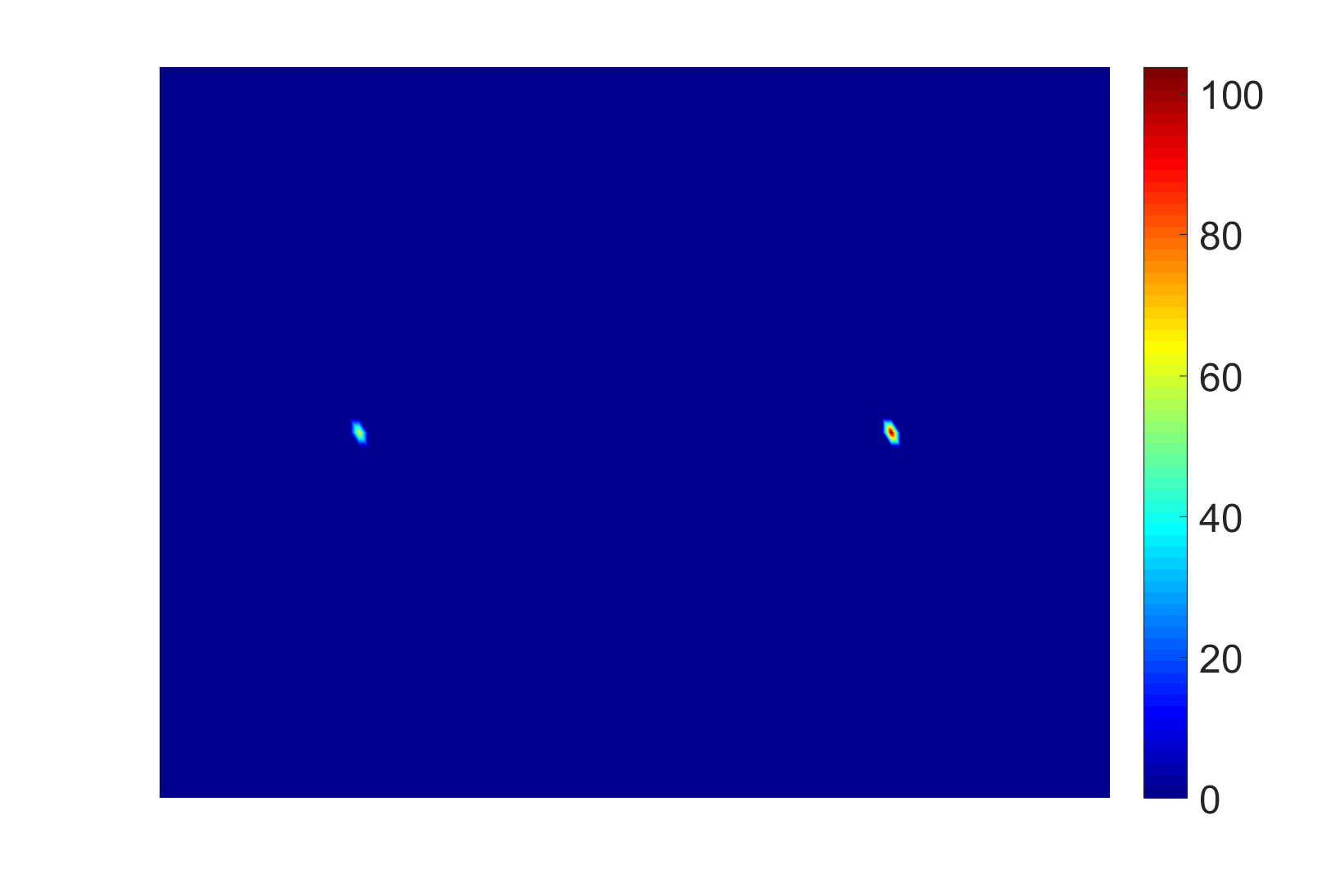}}
	\subfigure[Recovered final state ]{\includegraphics[width=0.32\textwidth]{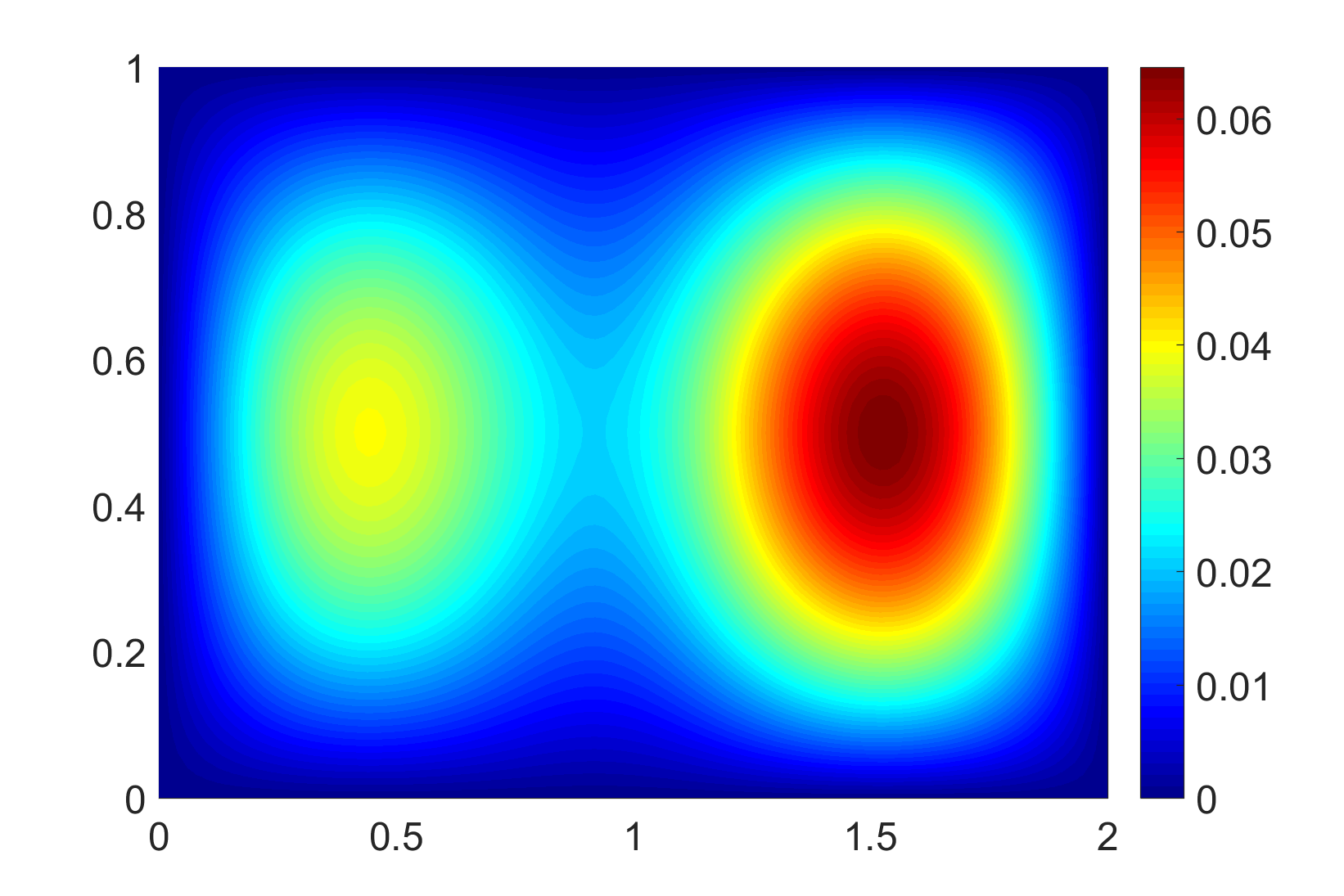}}
	\centering
\end{figure}

\section{Conclusions and Perspectives}\label{se:conclusion}

In this paper, we discussed the sparse initial source identification of diffusion-advection equations. The initial source is assumed to be a finite combination of Dirac measures indicating the locations, with their weights representing the intensities; and the locations and intensities are required to be identified. We designed an algorithm capable of identifying a sparse initial condition and leading the solution of our model to match with a prescribed final target in a given time horizon $T$. The algorithm we proposed to solve the initial source identification problem is comprised of two stages. Firstly, we formulated an optimal control problem with a cost functional consisting of three terms:
\begin{itemize}
	\item[1.] a least squares term seeking for an initial condition $u_0$ such that the corresponding solution, at time $t = T$ , is as close as possible to the desired target;
	\item[2.] an $L^1$-regularization term of the initial condition $u_0$ to promote sparsity;
	\item[3.] an  $L^2$-regularization term, introduced to guarantee the well-posedness of the problem while improving the conditioning of the optimal control problem;
\end{itemize}
and we introduced a generalized primal-dual algorithm to solve the optimal control problem.
Secondly, an optimization problem in terms of the locations and a least squares fitting corresponding to the intensities are considered to find the optimal locations and intensities of the initial source, respectively. In our numerical simulations, by comparing with the approach in \cite{monge2019sparse}, the effectiveness and efficiency of the proposed two-stage numerical approach were validated by several test cases. We observed that, when the final time is not large, the initial sources from reachable targets or noisy observations were accurately identified, even for some heterogeneous materials or coupled models. {When the final time becomes larger, the problem becomes increasingly ill-posed, and the sparse initial source may not be identified correctly. By some preliminary numerical tests, we found that the admissible final time, at which the sparse initial source can be identified accurately, varies from case to case.  To the best of our knowledge, there is still no numerical approach in the literature that can address sparse initial source identification problems in arbitrarily long time horizons. }

 Nevertheless, our work left unaddressed several key aspects of initial source identification problems, which are beyond the scope of the paper and will be subject of future investigation.

\begin{itemize}
	\item[1.] A natural extension of this work is to design novel and efficient algorithms allowing to address the sparse initial source identification of advection-diffusion systems in some relatively longer time horizons. We observe from Figure \ref{fig:long_time} (c) that the recovered location is close to the boundary of the domain, and this is mainly caused by the advection, which is the transport of a substance by bulk motion. Meanwhile, the recovered intensity is affected by the diffusion of the system. Hence, the sparse initial source identification of diffusion-advection systems can be viewed as a two-scale process: one is the inverse transport to determine the locations of the initial source, and the other is to determine the intensities of the initial source from the diffusion process. It is thus natural to consider some multiscale methods, for which some further investigation is needed.
		
\item[2.] It would be interesting to address a complete analysis of the maximum admissible final time at which the sparse initial source  can still be identified. This is highly related to the diffusivity parameter, the velocity field of the advection, the geometry of the domain, and the locations and intensities of the initial source to be identified. For instance, it is easy to see that a smaller diffusivity parameter or velocity field admits a larger maximum final time.
	
	\item[3.] In Section \ref{se:numerical},  the $L^2$- and $L^1$-regularization parameters were chosen empirically. Although we observe that the proposed two-stage approach works well for different roughly selected regularization parameter, it is important to discuss the optimal combination of these two regularizations. In particular, some regularization parameter
	choice rules have to be deliberately designed in order to find an optimal balance between the $L^2$-regularization that aims to avoid ill-conditioning and the $L^1$-regularization that promotes sparsity.
		
	\item[4.] To further simplify the implementation and to improve the numerical efficiency, it would be attractive to address the sparse initial source identification problem in one shot. In this regard, one may consider modifying the optimal control problem (\ref{eq:SIopt}) by taking into account the sparsity assumption (\ref{eq:deltas}) and designing some more sophisticated numerical approaches.
	
	\item [5.] {In Section  \ref{se:location}, a heuristic approach was studied for identifying the locations. Its numerical efficiency inspires us to investigate its related theoretical arguments in the future.}
 	 	
	\item[6.] Finally, it is worth designing algorithms for the sparse initial source identification of equations that are nonlinear or modeled on more complicated geometries. For instance,
	recall (\ref{eq:den_least}) that the identification of the optimal intensities relies on the linearity of the diffusion-advection equation (\ref{eq:mainPb}). Hence, the proposed two-stage numerical approach cannot be directly extended to the sparse initial source identification of nonlinear systems \cite{MT2013} and some more sophisticated techniques have to be involved in developing efficient numerical algorithms in this specific setting.
	
\end{itemize}

\section*{Acknowledgment}
The authors wish to acknowledge Leon Bungert (Hausdorff Center for Mathematics,
University of Bonn, Bonn, Germany$  $
) for fruitful discussions on the topics of the paper.  The authors are grateful to three anonymous referees for their very valuable comments which have helped them improve the paper substantially.


\begin{thebibliography}{10}
	\bibitem{BBC1986}
	 {\sc Beck, J.V., Blaekwell, B., and Clair, C.R.}  {\em Inverse Heat Conduction: Illposed Problems,} Wiley, New York, 1985.
	\bibitem{bregman1967relaxation}
	{\sc Bregman, L.~M.}
	\newblock The relaxation method of finding the common point of convex sets and
	its application to the solution of problems in convex programming.
	\newblock {\em USSR Comput. Math. Math. Phys. 7}, 3 (1967), 200-217.
	
	\bibitem{casas2017review}
	{\sc Casas, E.} A review on sparse solutions in optimal control of partial differential equations. {\em SeMA Journal}, 74 (2017),  319-344.
	
	
	\bibitem{casas2012approximation}
	{\sc Casas, E., Clason, C., and Kunisch, K.}
	\newblock Approximation of elliptic control problems in measure spaces with
	sparse solutions.
	\newblock {\em SIAM J. Control Optim. 50}, 4 (2012), 1735-1752.
	
	\bibitem{casas2013parabolic}
	{\sc Casas, E., Clason, C., and Kunisch, K.}
	\newblock Parabolic control problems in measure spaces with sparse solutions.
	\newblock {\em SIAM J. Control Optim. 51}, 1 (2013), 28-63.
	
	\bibitem{CK2016}
	{\sc Casas, E., and Kunisch, K.}  Parabolic control problems in space-time measure spaces. {\em ESAIM: Contr. Optim. Ca. 22}, 2 (2016), 355-370.
	
	\bibitem{casas2019using}
	{\sc Casas, E., and Kunisch, K.}
	\newblock Using sparse control methods to identify sources in linear
	diffusion-convection equations.
	\newblock {\em Inverse Probl. 35}, 11 (2019), 114002.
	
	\bibitem{casas2015sparse}
	{\sc Casas, E., and Vexler, B.~Zuazua, E.}
	\newblock Sparse initial data identification for parabolic {PDE} and its finite
	element approximations.
	\newblock {\em AIMS 5}, 3 (2015), 377-399.
	
	\bibitem{casas2013spike}
	{\sc Casas, E., and Zuazua, E.}
	\newblock Spike controls for elliptic and parabolic {PDE}s.
	\newblock {\em Syst. Control Lett. 62}, 4 (2013), 311-318.
	
	\bibitem{chambolle2011first}
	{\sc Chambolle, A., and Pock, T.}
	\newblock A first-order primal-dual algorithm for convex problems with
	applications to imaging.
	\newblock {\em J. Math. Imag. Vis. 40}, 1 (2011), 120-145.
	
	
	\bibitem{ck2011}
	{\sc Clason C., and Kunisch K.} A duality-based approach to elliptic control problems in non-reflexive Banach spaces. {\em ESAIM: Contr. Optim. Ca. 17,} 2011 (1), 243-266.
	
	\bibitem{duval2015}
{\sc Duval, V., and Peyr{\'e}, G.} Exact support recovery for sparse spikes deconvolution. {\em Found. Comput. Math., 15,} 5 (2015), 1315-1355.
	\bibitem{chen2018}
	{\sc Chen, Y. W.}  {A modified Lie-group shooting method for multi-dimensional backward heat conduction problems under long time span.} {\em Int. J. Heat Mass Transf. 127}, (2018), 948-960.
	\bibitem{el2005identification}
	{\sc El~Badia, A., Ha-Duong, T., and Hamdi, A.}
	\newblock Identification of a point source in a linear
	advection-dispersion-reaction equation: application to a pollution source
	problem.
	\newblock {\em Inverse Probl. 21}, 3 (2005), 1121.
	
	\bibitem{engl1996}
	{\sc Engl, H. W., Hanke, M., and Neubauer, A.} {\em Regularization of Inverse Problems,} vol. 375. Springer Science \& Business Media, 1996.
	
	
	
	\bibitem{glowinski1975approximation}
	{\sc Glowinski, R., and Marroco, A.}
	\newblock Sur l'approximation, par {\'e}l{\'e}ments finis d'ordre un, et la
	r{\'e}solution, par p{\'e}nalisation-dualit{\'e} d'une classe de
	probl{\`e}mes de {D}irichlet non lin{\'e}aires.
	\newblock {\em ESAIM: Math. Model. Numer. Anal. 9}, R2 (1975), 41-76.
	
	\bibitem{glowinski2020admm}
	{\sc Glowinski, R., Song, Y., and Yuan, X.}
	\newblock An {ADMM} numerical approach to linear parabolic state constrained
	optimal control problems.
	\newblock {\em Numer. Math.\/} 144 (2020), 1-36.
	
	\bibitem{glowinski2022}
	{\sc Glowinski, R., Song, Y., Yuan, X.,  Yue, H.} Application of the Alternating Direction Method of Multipliers to Control Constrained Parabolic Optimal Control Problems and Beyond. {\em Ann. Appl. Math. 38}, 2 (2022), 115-158.
	
	\bibitem{GER1983}
	{\sc Gorelick, S. M., Evans, B., and Remson, I.} Identifying sources of groundwater pollution: An optimization approach. {\em Water Resources Research, 19}, 3 (1983), 779-790.
	
	\bibitem{GT1979}
	{\sc Gol'shtein, E. G., and Tret'yakov, N. V.}  Modified Lagrangians in convex programming and their generalizations. {\em Math. Program. Stud.}, 10 (1979), 86-97.
	
	
	\bibitem{gurarslan2015solving}
	{\sc Gurarslan, G., and Karahan, H.}
	\newblock Solving inverse problems of groundwater-pollution-source
	identification using a differential evolution algorithm.
	\newblock {\em Hydrogeol. J. 23}, 6 (2015), 1109-1119.
	

	
	\bibitem{he2017algorithmic}
	{\sc He, B., Ma, F., and Yuan, X.}
	\newblock An algorithmic framework of generalized primal-dual hybrid gradient
	methods for saddle point problems.
	\newblock {\em J. Math. Imag. Vis. 58}, 2 (2017), 279-293.
	
%
	
	\bibitem{heyuan2012SINUM}
	{\sc He, B., and Yuan, X.} On the $O(1/n)$ convergence rate of Douglas-Rachford alternating direction method. {\em SIAM J. Numer. Anal. 50}, 2 (2012), 700-709.
	
	
	\bibitem{he2012convergence}
	{\sc He, B., and Yuan, X.}
	\newblock Convergence analysis of primal-dual algorithms for a saddle-point
	problem: from contraction perspective.
	\newblock {\em SIAM J. Imag. Sci. 5}, 1 (2012), 119-149.
	
	\bibitem{heyuan2015NM}
	{\sc He, B., and Yuan, X.} On non-ergodic convergence rate of Douglas-Rachford alternating direction method of multipliers. {\em Numer. Math. 130}, 3 (2015), 567-577.
	
	\bibitem{hinze2008optimization}
	{\sc Hinze, M., Pinnau, R., Ulbrich, M., and Ulbrich, S.}
	\newblock {\em Optimization with PDE Constraints}, vol.~23.
	\newblock Springer Science \& Business Media, 2008.
	
	\bibitem{isakov2017}
	{\sc Isakov, V. } {\em Inverse Problems for Partial Differential Equations}, vol. 127 of Applied Mathematical Sciences, Springer, Cham, third ed., 2017.
	
	\bibitem{justen2009general}
	{\sc Justen, L., and Ramlau, R.}
	\newblock A general framework for soft-shrinkage with applications to blind
	deconvolution and wavelet denoising.
	\newblock {\em Appl. Comput. Harmon. Anal. 26}, 1 (2009), 43-63.
	
	
	\bibitem{KHB2020}
	{\sc Koulouri, A., Heins, P., and Burger, M.}  Adaptive superresolution in deconvolution of sparse peaks. {\em 	IEEE Trans. Signal Process. 69}, (2020), 165-178.
	
	\bibitem{kunisch2014measure}
	{\sc Kunisch, K., Pieper, K., and Vexler, B.}
	\newblock Measure valued directional sparsity for parabolic optimal control
	problems.
	\newblock {\em SIAM J. Control Optim. 52}, 5 (2014), 3078-3108.
	
	\bibitem{leykekhman2020numerical}
	{\sc Leykekhman, D., Vexler, B., and Walter, D.}
	\newblock Numerical analysis of sparse initial data identification for
	parabolic problems.
	\newblock {\em ESAIM: Math. Model. Numer. Anal. 54}, 4 (2020), 1139-1180.
	
	\bibitem{li2006determining}
	{\sc Li, G., Tan, Y., Cheng, J., and Wang, X.}
	\newblock Determining magnitude of groundwater pollution sources by data
	compatibility analysis.
	\newblock {\em Inverse Probl. Sci. Eng. 14}, 3 (2006), 287-300.
	
	
	\bibitem{li2014heat}
	{\sc Li, Y., Osher, S., and Tsai, R.}
	\newblock Heat source identification based on $l^1$ constrained
	minimization.
	\newblock {\em Inv. Problems Imag. 8}, 1 (2014), 199-221.
	
	\bibitem{liu2001}
	{\sc Liu, C. S.}  { Cone of non-linear dynamical system and group preserving schemes.} {\em 	Int. J. Heat Mass Transf. 36}, 7 (2001), 1047-1068.
	
	\bibitem{liu2004}
	{\sc Liu, C. S.}  { Group preserving scheme for backward heat conduction problems}. {\em 	Int. J. Heat Mass Transf. 47}, 12-13 (2004), 2567-2576.
	
	
	\bibitem{MT2013}	{\sc Mamonov, A. V., and Tsai, Y. R.}  Point source identification in nonlinear advection-diffusion-reaction systems. {\em Inverse Probl. 29}, 3 (2013), 035009.
	
	
	
	\bibitem{mera2005}
	  {\sc Mera, N.S.} The method of fundamental solutions for the backward heat
	conduction problem.  {\em Inv. Probl. Sci. Eng. 13}, (2005), 65–78.
	
	\bibitem{monge2019sparse}
	{\sc Monge, A., and Zuazua, E.}
	\newblock Sparse source identification of linear diffusion-advection equations
	by adjoint methods.
	\newblock {\em Syst. Control Lett. 145,\/} (2020), 104801.
	
	
	\bibitem{nesterov2004introductory}
	{\sc Nesterov, Y.}
	\newblock {\em Introductory Lectures on Convex Optimization: A Basic Course}.
	\newblock Springer Science \& Bussines Media, New York, 2004.
	
	\bibitem{ohnaka1989}
	{\sc Ohnaka. K., and Uosaki, K.} Boundary element approach for identification of point forces of
	distributed parameter systems, {\em Internat. J. Control}, 49 (1989), 119-127.
	
	\bibitem{ozis2000}
	{\sc {\" O}zi{\c{s}}ik, M. N., and Orlande, H. R. B.} Inverse Heat Transfer: Fundamentals and Applications,
	Hemisphere Pub, 2000.
	
	
%
	
	
	\bibitem{schindele2017proximal}
	{\sc Schindele, A., and Borz{\`\i}, A.}
	\newblock Proximal schemes for parabolic optimal control problems with sparsity
	promoting cost functionals.
	\newblock {\em Int. J. Control 90}, 11 (2017), 2349-2367.
	
	
	\bibitem{stadler2009elliptic}
	{\sc Stadler, G.}
	\newblock Elliptic optimal control problems with ${L}^1$-control cost and
	applications for the placement of control devices.
	\newblock {\em Comput. Optim. Appl. 44}, 2 (2009), 159-181.
	
	
	
	\bibitem{ulbrich2011semismooth}
	{\sc Ulbrich, M.}
	\newblock {\em Semismooth Newton Methods for Variational Inequalities and
		Constrained Optimization Problems in Function Spaces}.
	\newblock SIAM, 2011.
	
	\bibitem{ww2010}
	Wachsmuth, G., and Wachsmuth, D. { Convergence and regularization results for optimal control problems with sparsity functional}. {\em ESAIM: Contr. Optim. Ca. 17}, 3 (2011), 858-886.
	
	
\end{thebibliography}
\end{document}